\documentclass[12pt, a4paper]{book}
\usepackage[utf8]{inputenc}
\usepackage{graphicx}
\usepackage{amssymb,latexsym,amsthm,amsmath,enumerate,amscd,
	epsfig,wasysym,textcomp}
\usepackage{tikz}
\addtocontents{toc}{\protect\setcounter{tocdepth}{1}}

\long\def\symbolfootnote[#1]#2{\begingroup%
	\def\thefootnote{\fnsymbol{footnote}}\footnote[#1]{#2}\endgroup}

\newcommand{\tr}{\ensuremath{{}^t\!}}
\newcommand{\tra}{\ensuremath{{}^t}}
\newcommand{\C}{\mathfrak C}

\newcommand{\E}{\mathcal E}
\newcommand{\F}{\mathcal F}

\makeatletter
\def\imod#1{\allowbreak\mkern10mu({\operator@font mod}\,\,#1)}
\makeatother

\newtheorem{theorem}{Theorem}[section]
\newtheorem{lemma}[theorem]{Lemma}
\newtheorem{corollary}[theorem]{Corollary}
\newtheorem{proposition}[theorem]{Proposition}
\newtheorem{definition}[theorem]{Definition}
\newtheorem{question}[theorem]{Question}
\newtheorem{notation}[theorem]{Notation}
\newtheorem{problem}[theorem]{Problem}
\newtheorem*{theorem*}{Theorem}
\theoremstyle{definition}
\newtheorem{remark}[theorem]{Remark}
\newtheorem{example}[theorem]{Example}
\numberwithin{equation}{section}

\newcommand{\ignore}[1]{}

\newcommand{\mynote}[1]{}

\begin{document}
\pagestyle{myheadings}
%\frontmatter    
\begin{titlepage}
	\begin{center}
		\vspace*{1cm}
		\textbf{\LARGE\fontsize{22}{1.5} COMPUTATIONS IN CLASSICAL GROUPS}\\
		\vspace{2cm}
		\large A thesis 
		\\Submitted in partial fulfillment of the requirements
		\\of the degree of
		\\\textbf{Doctor of Philosophy}
		\vspace{1cm}
		\\By
		
		\vspace{1.3cm}
		\Large\fontsize{15}{1.5} \textbf{Sushil Bhunia}
		\\\large 20123166
		\vspace{2cm}
		\begin{center}
			\includegraphics[width=2cm,height=2cm]{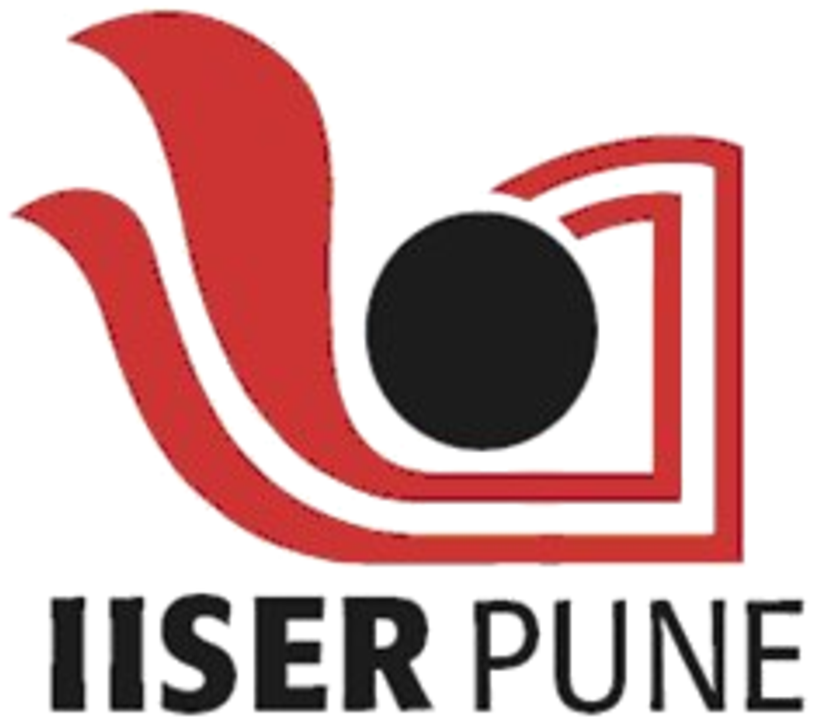}
		\end{center}
		\vspace{1cm}
		
		\normalsize INDIAN INSTITUTE OF SCIENCE EDUCATION AND RESEARCH PUNE
		\vfill
		{August, 2017}
	\end{center}
\end{titlepage}

%%%%%%%%%%%%%%%%%%%%%%%%%%%%%%%%%%%%%%%%%%%%%%%%%%%%%%%%%%%%%%%%%%%
\vspace*{\fill}
\begingroup
\begin{center}
	\textbf{\it\LARGE Dedicated to \\[10pt] My Didima (Grandmother)}
\end{center}
\vskip10cm
\endgroup
\vspace*{\fill}
\cleardoublepage
%%%%%%%%%%%%%%%%%%%%%%%%%%%%%%%%%%%%%%%%%%%%%%%%%%%%%%%%%%%%%%%%%%%
\chapter*{Certificate}
Certified that the work incorporated in the thesis entitled \textit{``Computations in Classical Groups"},
submitted by \textit{Sushil Bhunia} was carried out by the candidate, under my supervision. 
The work presented here or any part of it has not been included in any other thesis submitted previously 
for the award of any degree or diploma from any other university or institution.\\[25mm]
{\em Date: August 28, 2017} \hfill {\em Dr. Anupam Kumar Singh}\\ \vspace{-10mm}
\begin{flushright}
	{Thesis Supervisor}
\end{flushright}
\cleardoublepage
%%%%%%%%%%%%%%%%%%%%%%%%%%%%%%%%%%%%%%%%%%%%%%%%%%%%%%%%%%%%%%%%%%%
\chapter*{Declaration}
I declare that this written submission represents my ideas in my own words and where others' ideas have been included, 
I have adequately cited and referenced the original sources. I also declare that I have adhered to all principles of 
academic honesty and integrity and have not misrepresented or fabricated or falsified any idea/data/fact/source in my 
submission. I understand that violation of the above will be cause for disciplinary action by the institute and can 
also evoke penal action from the sources which have thus not been properly cited or from whom proper permission has not 
been taken when needed.\\[25mm]
{\em Date: August 28, 2017} \hfill {\em Sushil Bhunia}\\ \vspace{-10mm}
\begin{flushright}
	{\em Roll Number: 20123166}
\end{flushright}
\cleardoublepage
%%%%%%%%%%%%%%%%%%%%%%%%%%%%%%%%%%%%%%%%%%%%%%%%%%%%%%%%%%%%%%%%%%
\chapter*{Acknowledgements}\addcontentsline{toc}{chapter}{\textbf{Acknowledgements}}\chaptermark{Acknowledgements}
To my life-coach, my late grandmother Bishnupriya Bera: because I owe it all to you. 

First and foremost, I would like to express my sincere gratitude to my thesis 
supervisor Dr. Anupam Singh for the continuous support, for his patience, motivation, 
enthusiasm and encouragement. He was always ready to discuss with me. He trusted my ability 
and was patient enough to explain anything to me. I could not have imagined 
having a better guide for my Ph.D. The 
questions studied in this thesis are formulated by him.

Besides my supervisor, I would like to thank the rest of my research 
advisory committee: Prof. K. N. Raghavan and Dr. Baskar Balasubramanyam for 
their insightful comments and encouragement. 
I had the opportunity to talk mathematics with several people. 
I would like to thank them for their support and encouragement. In 
particular, I would like to thank Prof. Dipendra Prasad, Prof. Amritanshu Prasad, 
Prof. Maneesh Thakur and Prof. Benjamin Martin. I had very helpful mathematical 
discussions with Dr. Ayan Mahalanobis, Dr. Krishnendu Gongopadhyay, and Dr. Ronnie Sebastian.

I owe my understanding of mathematics to many mathematicians at IISER Pune, especially to 
Dr. Diganta Borah, Dr. Chandrasheel Bhagwat, Dr. Rabeya Basu, Dr. A. Raghuram, Dr. Steven Spallone,
Dr. Vivek Mohan Mallick, Dr. Kaneenika Sinha, Dr. Tejas Kalelkar, Dr. Krishna Kaipa 
and Dr. Amit Hogadi. I am grateful to all of them. 
I am thankful to CSIR for the financial support in the form of the research fellowship. 
I would like to acknowledge the support of the institute and its administrative staff 
members for their cooperation, special thanks are due to 
Mrs. Suvarna Bharadwaj, Mr. Tushar Kurulkar and Mr. Kalpesh Pednekar. 

I am grateful to my teachers starting from my school days till date 
having faith in me and guiding me in right direction. Especially to Mr. Chittaranjan Chaudhuri, 
Mr. Suvendu Dandapat, Mr. Gautambabu in school and Prof Himadrisekhar Sarkar in Jadavpur University.

I thank all my school friends, batchmates in Jadavpur University and IISER Pune, with whom 
I shared good times and bad times as well. 
%I may not be able to list them all individually, 
%but I hope that they will recognize themselves. 
Many of you will recognize yourselves, and I hope that you will 
forgive me for not naming you individually.
I thank my friends in Jadavpur University. 
In particular, Debmalya, Prahlad, Srimoyee, Simi, Mousumi, Barnali, Nirupam, Chiranjit, 
Dibakar, Dishari, Ibrahim, Manoranjanda, Biswajitda, and Gautamda. 
All the students of Mathematics at IISER Pune deserve a note of 
appreciation for being enthusiastic about discussing mathematics with me. I thank 
Rohit, Yasmeen, Hitesh, Rashmi, Sudhir, Manidipa, Pralhad, Prabhat, 
Makarand, Jatin, Neha, Jyotirmoy, Milan, Tathagata, Debangana, Ayesha, Girish, 
Advait, Chitrabhanuda, Dilpreet, Uday, Ratna, and others for their help and discussions.
With a special mention to Mr. Rohit Joshi. It was fantastic to 
have the opportunity to discuss mathematics with him. 
A special acknowledgment goes to my office mate of many years: Ms. Manidipa Pal. 
She was a true friend ever since we began to share an office in $2012$. 
I must thank Mr. Uday Baskar Sharma for correcting my English.
I thank all the security persons in IISER and my special thanks to IISER football team.

Finally, I must express my deepest gratitude to my parents for 
providing me with unconditional support and constant encouragement throughout 
my years of study and through the process of research and write this thesis and 
my life in general, without whom this thesis would not have existed.
I am also grateful to my other family members who have supported me along 
the way. Especially to my younger brother Samir, my sisters Swapna, Bandana, Gangotryi and my nephew Chunai, Munai. 
Last but not the least to my Kakima and Sir. Also a mention to Pramita.
It is not possible to express my gratitude towards them in words.

The comments of referees has been very helpful to improve this thesis.

\noindent{\it\hfill Sushil Bhunia}

\pagenumbering{arabic}
\setcounter{page}{5}
%%%%%%%%%%%%%%%%%%%%%%%%%%%%%%%%%%%%%%%%%%%%%%%%%%%%%%%%%%%%%%%%%%
\tableofcontents
%\pagenumbering{arabic}
%%%%%%%%%%%%%%%%%%%%%%%%%%%%%%%%%%%%%%%%%%%%%%%%%%%%%%%%%%%%%%%%%%%
\chapter*{Abstract}\addcontentsline{toc}{chapter}{\textbf{Abstract}}\chaptermark{Abstract}
In this thesis, we develop algorithms similar to the Gaussian elimination algorithm in symplectic 
and split orthogonal similitude groups. As an application to this algorithm, we compute the 
spinor norm for split orthogonal groups. Also, we get similitude character for symplectic and split 
orthogonal similitude groups, as a byproduct of our algorithms. 

Consider a perfect field $k$ with ${\rm char}\, k \neq 2$, which has a non-trivial Galois automorphism of order $2$.
Further, suppose that the fixed field $k_0$ has the property that there are only finitely many field extensions of any 
finite degree. In this thesis, we prove that the number of $z$-classes in the unitary group defined 
over $k_0$ is finite. Eventually, we count the number of $z$-classes in the unitary group over a finite field 
$\mathbb{F}_q$, and prove that this number is same as that of the general linear group over $\mathbb{F}_q$ 
when $q$ is large enough.
%%%%%%%%%%%%%%%%%%%%%%%%%%%%%%%%%%%%%%%%%%%%%%%%%%%%%%%%%%%%%%%%%%%%
\chapter*{Notation}\addcontentsline{toc}{chapter}{\textbf{Notation}}\chaptermark{Notation}
%\begin{align*}
$k : \text{a field}\; (\mathrm{char}\;\neq 2)\\ 
k^{\times} : k \setminus \{0\} \\
\bar{k} : \text{algebraic closure of}\; k \\
\mathbb{Z} : \text{integers}  \\
\mathbb{Q} : \text{rational numbers} \\
\mathbb{R} : \text{real numbers} \\
\mathbb{C} : \text{complex numbers} \\ 
\mathbb{Q}_{p} : \text{$p$-adic fields} \\
\mathbb{F}_q : \text{finite fields with $q$ elements}\\
R : \text{a commutative ring with $1$}\\
R^{\times} : \text{units of a ring $R$}\\
(V,B) : \text{bilinear or sesquilinear form on $V$}\\
\beta : \text{the matrix of $B$ relative to a basis} \\
dV : \text{discriminant of $(V,B)$}\\
Q : \text{a quadratic form} \\
\otimes : \text{tensor product}\\
\oplus : \text{direct sum}\\
\perp : \text{orthogonal sum}\\
\cong : \text{isomorphism}\\
\mathcal{Z}_{G}(g) : \text{centralizer of $g$ in $G$}\\
\mathcal{Z}(G) : \text{center of $G$}\\
\mathrm{Aut}\;(V) : \text{set of all automorphisms of $V$}\\
M(n,k) : \text{matrix algebra over $k$}\\
GL(V)\; \text{or} \;GL(n,k) : \text{general linear group}\\
SL(V) \; \text{or}\; SL(n,k) : \text{special linear group}\\
GSp(V,B)\; \text{or}\; GSp(n,k) : \text{symplectic similitude group}\\
Sp(V,B)\; \text{or}\; Sp(n,k) : \text{symplectic group}\\
%\end{align*}
%\begin{align*}
GO(V,B)\; \text{or}\; GO(n,k) : \text{orthogonal similitude group}\\
O(V,B)\; \text{or}\; O(n,k) : \text{orthogonal group} \\
U(V,B)\; \text{or}\; U(n,k) : \text{unitary group}\\
\mathrm{Gal}\;(L/k) : \text{Galois group of a field $L$ over $k$}\\
\mathrm{det}(g) : \text{determinant of a matrix $g$}\\
\tr g : \text{transpose of a matrix $g$}\\
\tr g^{-1} : \text{transpose inverse of a matrix $g$}\\
p(n) : \text{number of partitions of $n$}\\
\mathrm{diag}(\lambda_1, \ldots, \lambda_n) : \text{diagonal matrix}\\
\emph{italic} : \text{definition}\\
\qed : \text{end of a proof}$
%\end{align*}
%%%%%%%%%%%%%%%%%%%%%%%%%%%%%%%%%%%%%%%%%%%%%%%%%%%%%%%%%%%%%%%%%%%
\chapter{Introduction}\label{chapter1}
This thesis deals with the subject of classical groups.
Specifically, we deal with the Gaussian elimination 
for some similitude groups, and conjugacy classes of 
centralizers for certain classical groups. We give a
concrete algorithm for symplectic and split orthogonal 
similitude groups analogous to the usual row and column 
operations to solve the word problem. Also, we give structure of centralizers and 
classes of centralizers in unitary groups to complete 
the story for classical groups, at least as far as 
the topics we deal with are concerned.
\vskip 4mm
\noindent
\section*{What is the Gaussian elimination?}
\vskip 4mm
\noindent
Gaussian elimination is a very old technique in Mathematics. It appeared
in print as chapter eight in a Chinese mathematical text called, 
``The nine chapters of the mathematical art''. It is believed, a part 
of that book was written as early as $150$ BCE. For a historical 
perspective on Gaussian elimination, we refer to a nice work 
by Grcar~\cite{Gc}.

In computational group theory, one is always looking for algorithms 
that solve the word problem (for definition see p. 4, section 1.4~\cite{Ob}). 
Algorithms for word problem are useful 
in other programs in computational group theory, namely, the group 
recognition program and the membership problems.
Extensive work on these programs are being done by several people Leedham-Green and 
O'Brien~\cite{LO}, and Guralnick et.al.~\cite{GKKL}.
Thus, one of the main objectives of this thesis is to give 
an algorithm, on similar lines as the row-column operations for general 
linear groups, to solve the word problem for similitude groups. In this thesis, 
we work with Chevalley generators~\cite{Ca1}. Chevalley generators for the special 
linear group $SL(n,k)$ are elementary transvections, which are used to do 
the Gaussian elimination for $GL(n,k)$. The similitude groups are thought of as an 
analog of what $GL(n,k)$ is for $SL(n,k)$. So, for the Gaussian elimination of 
symplectic and split orthogonal similitude groups, we use the Chevalley generators.

These Chevalley generators for classical groups are well-known for a very long time. 
However, its use in row-column operations in symplectic and split orthogonal similitude 
groups is new. We develop row-column operations, very similar to the Gaussian elimination 
algorithm for general linear groups. We call our algorithms \textbf{Gaussian elimination} in symplectic 
and split orthogonal similitude groups respectively. 

In a nutshell, Gaussian elimination is nothing but a series of 
row and column operations. For details see Chapter~\ref{chapter6}.
The algorithms that we develop in this thesis work
for a split bilinear form $B$ (see (4) in Example~\ref{splitb}). First, we define 
elementary matrices (see Section~\ref{elementarymatrices}), which give 
elementary operations (see~\ref{elementaryoperations}) for similitude groups. 
We prove the following result: 
\begin{theorem}[Theorem~\ref{maintheorem1}]
	Every element of the symplectic similitude group $GSp(2l,k)$ or split  
	orthogonal similitude group $GO(n,k)$ (here $n=2l$ or $2l+1$),  
	can be written as a product of elementary matrices and 
	a diagonal matrix. Furthermore, the diagonal matrix is of the following form:
	\begin{enumerate}
		\item In $GSp(2l,k)$, 
		$\mathrm{diag}(\underbrace{1,\ldots,1}_{l},\underbrace{\mu(g),\ldots, \mu(g)}_{l})$, where $\mu(g) \in k^{\times}$.
		\item In $GO(2l,k)$,  
		$\mathrm{diag}(\underbrace{1,\ldots,1,\lambda}_{l},\underbrace{\mu(g),\ldots, \mu(g),\mu(g)\lambda^{-1}}_{l})$, 
		where $\mu(g), \lambda \in k^{\times}$.
		\item In $GO(2l+1,k)$,  
		$\mathrm{diag}(\alpha(g),\underbrace{1,\ldots,1,\lambda}_{l},\underbrace{\mu(g),\ldots, \mu(g), \mu(g)\lambda^{-1}}_{l})$, 
		where $\alpha(g)^2=\mu(g)$ 
		and $\mu(g), \lambda \in k^{\times}$. 
	\end{enumerate}
\end{theorem}
\noindent
\section*{What is the spinor norm and why study them?}
\vskip 4mm
\noindent
Let $k$ be a field with ${\rm char}\, k \neq 2$. 
The \textbf{spinor norm} is a group homomorphism 
$\Theta : O(n,k)\rightarrow k^{\times}/k^{\times2}$ defined by  
$\Theta(g)=\displaystyle \prod_{i=1}^m Q(u_{i})k^{\times2}$, 
where $g =\sigma_{u_{1}}\sigma_{u_{2}}\cdots \sigma_{u_{m}}$ 
using Cartan-Dieudonne theorem 
(see Section~\ref{classicalspinornorm}) and $Q$ is the quadratic form associated 
to the bilinear form $B$.
In the connection to group recognition project, 
Scott H. Murray and Colva M. Roney-Dougal~\cite{MR} studied spinor norm earlier. The 
definition of the spinor norm is not friendly to compute.
Hahn, Wall, and Zassenhaus~\cite{Ha, Wa1, Za} developed a theory to compute the spinor norm.
In this thesis, we will give an efficient algorithm to compute the 
spinor norm using Gaussian elimination algorithm. 
From Gaussian elimination algorithm, one can compute the spinor norm easily. 
Since the commutator subgroup of the orthogonal 
group is the kernel of the spinor norm restricted to the special orthogonal group, so the following theorem also gives 
a membership test for the commutator subgroup in the orthogonal group.
We prove the following result:
\begin{theorem}[Theorem~\ref{maintheorem2}]
	Let $g\in O(n,k)$ (here $n=2l$ or $n=2l+1$). Suppose Gaussian elimination reduces $g$ to 
	$\mathrm{diag}(\underbrace{1,\ldots,1,\lambda}_{l \; \text{or}\; l+1},\underbrace{1,\ldots,1,\lambda^{-1}}_{l})$, where
	$\lambda \in k^{\times}$. 
	Then the spinor norm $\Theta (g)=\lambda k^{\times 2}$.
\end{theorem}
\noindent
\section*{What are the $z$-classes and why study them?}
\vskip 4mm
\noindent
Let $G$ be a group. The elements $x$ and $y \in G$ are said to be $z$-equivalent denoted as 
$x\sim_{z} y$ if their centralizers in $G$ are conjugate, i.e., $\mathcal{Z}_G(y)=g\mathcal{Z}_G(x)g^{-1}$ for some 
$g\in G$, where $\mathcal{Z}_G(x):=\{g\in G \mid gx=xg \}$ denotes centralizer of $x$ in $G$. 
Clearly $\sim_{z}$ is an equivalence relation on $G$. 
The equivalence classes with respect to this relation are called \textbf{$z$-classes}. 
It is easy to see that if two elements 
of a group $G$ are conjugate then their centralizers are conjugate, thus they are also $z$-equivalent. 
However, in general, the converse is not true. In fact, a group may have infinitely many conjugacy classes 
but finitely many $z$-classes (see Example~\ref{singlezclass}). 
In this thesis, we explore the $z$-classes for classical groups.  
In~\cite{St2}, R. Steinberg proved the following:
\begin{theorem}[Steinberg]\label{steinberg}
	Let $G$ be a reductive algebraic group defined over an algebraically closed 
	field $k$ of good characteristic, then the number of $z$-classes in $G$ is finite.
\end{theorem}
\begin{question}
	What can we say about the finiteness of $z$-classes for algebraic group $G$ defined over an arbitrary field $k$?
\end{question}
\noindent
To study this we assume that the field $k$ satisfies the following property:
\begin{definition}[Property FE]
	A perfect field $k$ of ${\rm char}\, k \neq 2$  has the property FE if $k$ has only finitely many field 
	extensions of any fixed finite degree. 
\end{definition}
Examples of such fields are, algebraically closed fields (for example, $\mathbb C$), real numbers $\mathbb R$, local 
fields (for example, $\mathbb Q_p$), and finite fields $\mathbb F_q$. 
From now on we assume that $k$ has property FE unless stated otherwise.
In~\cite{Si}, A. Singh studied $z$-classes for real compact groups of type $G_2$.
Ravi S. Kulkarni proved the following (see Theorem 7.4~\cite{Ku}):
\begin{theorem}[Kulkarni]
	Let $V$ be an $n$-dimensional vector space over a field $k$ with 
	the property FE, then the number of $z$-classes in $GL(n,k)$ is finite.
\end{theorem}
\noindent
K. Gongopadhyay and Ravi S. Kulkarni proved the following (Theorem 1.1~\cite{GK}):
\begin{theorem}[Gongopadhyay-Kulkarni]\label{gongopadhyay}
	Let $V$ be an $n$-dimensional vector space over a field 
	$k$ with the property FE, equipped with a non-degenerate symmetric or skew-symmetric bilinear form $B$. 
	Then, there are only finitely many $z$-classes in orthogonal groups $O(V,B)$ and symplectic groups $Sp(V,B)$.
\end{theorem}
This result generalizes Steinberg's result mentioned above (Theorem~\ref{steinberg}). 
In this thesis, we extend this result 
to the unitary groups. We prove the following result:
\begin{theorem}[Theorem~\ref{maintheorem3}]
	Let $k$ be a perfect field of ${\rm char}\, k\neq 2$ with a non-trivial Galois automorphism of order $2$. 
	Let $V$ be a finite dimensional vector space over $k$ with a non-degenerate hermitian form $B$. 
	Suppose the fixed field $k_0$ has the property FE,  
	then the number of $z$-classes in the unitary group $U(V,B)$ is finite.
\end{theorem}
The FE property of the field is necessary for the above theorem.
For example, the field of rationals $\mathbb Q$ does not have property FE. We show that the above theorem 
is no longer true over $\mathbb Q$ (see Example~\ref{nonexample}).

If we look at character table of $SL(2,q)$ (for example see~\cite{B} and ~\cite{Pr}), 
we notice conjugacy classes and irreducible characters bunched together (see p. 404 in~\cite{Gr}). 
One observes a similar pattern in the work of Srinivasan~\cite{Sr} for $Sp(4,q)$. In~\cite{Gr}, 
Green studied the complex representations of $GL(n,q)$ where he introduced the function $t(n)$ for the 
`types of characters/classes' (towards the end of section 1 on page 407-408) 
which is same as the number of $z$-classes in $GL(n,q)$.

In Deligne-Lusztig theory, where one studies representation theory of finite groups of Lie type, 
$z$-classes of semisimple elements play an important role. 
In~\cite{Ca2} Carter and in~\cite{Hu2} Humphreys defined genus of 
an algebraic group $G$ defined over $k$. Two semisimple elements have same \emph{genus} if they are 
$z$-equivalent in $G(k)$. Thus understanding $z$-classes 
for finite groups of Lie type, especially semisimple $z$-classes, and their counting is of importance in 
representation theory (see~\cite{Fl, FG, Ca2, DM}). A. Bose, in~\cite{Bo}, calculated the genus number for 
simply connected simple algebraic groups over an algebraically closed field, and compact simple Lie groups.
In this thesis we prove the following: 
%\begin{theorem}[Theorem~\ref{maintheorem4}]
%The number of $z$-classes in $U(n,q)$ is same as the number of $z$-classes in $GL(n,q)$ if $q>n$.
%\end{theorem}
\begin{theorem}[Theorem~\ref{maintheorem4}]
	The number of $z$-classes in $U(n,q)$ is same as the number of $z$-classes in $GL(n,q)$ if $q>n$. 
	Thus, the number of $z$-classes for either group can be read off 
	by looking at the coefficients of the function 
	$\displaystyle\prod_{i=1}^{\infty} z(x^i)$, where $z(x)=\displaystyle\prod_{j=1}^{\infty}\frac{1}{(1-x^j)^{p(j)}}$ 
	and $p(j)$ is the number of partitions of $j$.
\end{theorem}  
Along the way, we also prove some counting results (see for example, Proposition~\ref{countzgln}, 
Proposition~\ref{countprod}, Theorem~\ref{countzun1}).

\noindent
\textbf{A chapter wise description:} 
A conscious effort is made to make this thesis self-contained and reader-friendly.
The results in Chapters $2$ to $5$ are all well-known. 
They are preliminary in nature, and almost all basic results are recalled 
in the first four chapters, which are used in this thesis. After covering the 
preliminaries in the first four chapters, we report on author's 
research work in the next four chapters. 
Finally, in the last chapter, we give some further research problems. 
That pretty much summarizes the thesis giving glimpses into the main results 
proved in the various chapters.
%%%%%%%%%%%%%%%%%%%%%%%%%%%%%%%%%%%%%%%%%%%%%%%%%%%%%%%%%%%%%%%%%%%
\chapter{Classical Groups}\label{chapter2}
This chapter is the most basic and at the same time 
most essential part of this thesis. 
In this chapter, we will discuss the groups that are popularly known as the classical groups, 
as they were named by Hermann Weyl. 
Let $k$ be a field. Let $V$ be an $n$-dimensional vector space over $k$. 
We denote the set of all invertible linear transformations of $V$ by $GL(V)$.
The set $GL(V)$ is a group under the multiplication defined by the composition of maps.
Let us fix a basis $\{e_1,\ldots,e_n \}$ of $V$. Then we can identify $GL(V)$ with $GL(n,k)=
\{g\in M(n,k)\mid \mathrm{det}\;(g)\neq 0\}$, the set of all $n \times n$ invertible matrices. 
This group is called the \emph{general linear group}.
All further groups discussed are subgroups of $GL(V)$. 
The \emph{special linear group} $SL(n,k):=\{g\in GL(n,k)\mid \mathrm{det}\;(g)=1\}$. 
In Weyl's words, ``each group stands in its own right and does not 
deserve to be looked upon merely as a subgroup of something else, be it even 
\emph{Her All-embracing Majesty} $GL(n)$''.
The exposition in this chapter is mostly based on the 
book by Larry C. Grove~\cite{Gv}. In Section 2.1 
we describe reductive algebraic groups. Section 2.2 
covers the basic definitions and some very basic 
properties of classical groups, especially for 
symplectic and orthogonal groups. Also in this section, we 
introduce the notion of the spinor norm. In the 
last section, we describe the unitary groups and some 
important examples, which will be useful later in this thesis. 
\section{Reductive Algebraic Groups}
There are several excellent references for this topic, 
Borel~\cite{Br}, Springer~\cite{Sp} and Humphreys~\cite{Hu1}, to name a few.
We fix a perfect field $k$ ($\mathrm{char}\,k \neq 2 $) for this section, 
and $\bar{k}$ denotes the algebraic closure of $k$.
An \emph{algebraic group} $G$ defined over $\bar{k}$ is a 
group as well as an affine variety over $\bar{k}$ such that the maps 
$\mu \colon G\times G \rightarrow G$, and $i \colon G \rightarrow G$ 
given by $\mu(g_1,g_2)=g_1g_2$, and $i(g)=g^{-1}$ are morphisms of 
varieties. An \emph{algebraic group} $G$ is defined over $k$, 
if the polynomials defining the underlying affine variety $G$ 
are defined over $k$, with the maps $\mu$ and $i$ defined over 
$k$, and the identity element $e$ is a $k$-rational point of $G$. 
We denote the $k$-rational points of $G$ by $G(k)$. 
Any algebraic group
$G$ is a closed subgroup of $GL(n,k)$ for some $n$. 
Hence algebraic groups are called \emph{linear algebraic groups}.

An element in $GL(n,k)$ is called \emph{semisimple} (respectively, \emph{unipotent}) 
if it is diagonalizable over $\bar{k}$ (respectively, if all its eigenvalues are equal 
to $1$). We have $G\hookrightarrow GL(n,k)$. An element $g\in G$ is said to be \emph{semisimple} 
(respectively, \emph{unipotent}) if the image of $g$, under the above inclusion, is 
semisimple (respectively, unipotent) in $GL(n,k)$. 
An algebraic group $G$ is said to be \emph{unipotent} if all its elements 
are unipotent. The \emph{radical} of an algebraic group $G$ over $k$ is defined 
to be the largest closed, connected, solvable, normal subgroup of $G$, denoted by $R(G)$.
We call $G$ to be a \emph{semisimple} algebraic group if $R(G)=\{e\}$. 
The \emph{unipotent radical} of $G$ is defined to be the largest, closed, connected, unipotent, normal 
subgroup of $G$ and denoted by $R_u(G)$.
We call a connected group $G$ to be \emph{reductive} if $R_u(G)=\{e\}$. 
For example, the group $GL(n,k)$ is a reductive group, whereas $SL(n,k)$ is 
a semisimple group. A semisimple algebraic group is always a reductive group.
In next section, we see more examples of algebraic groups, namely, classical groups.
\subsection{Jordan decomposition}
Recall that an element $g\in GL(n,k)$ can be written as $g=g_sg_u=g_ug_s$, in a unique way,
where $g_s \in GL(n,k)$ is semisimple, and $g_u \in GL(n,k)$ is unipotent. 
This decomposition is called the \emph{Jordan decomposition} for invertible matrices. We have the following 
Jordan decomposition in linear algebraic groups. 
We need the following (Theorem 2.4.8~\cite{Sp}), 
\begin{theorem}[Jordan decomposition]\label{jordandecomposition}
	Let $G$ be a linear algebraic group defined over a perfect field $k$ and let $g \in G$. 
	Then there exist unique elements $g_s, g_u \in G$ such that 
	$g=g_sg_u=g_ug_s$. Furthermore, if $\phi \colon G\rightarrow H$ is a 
	homomorphism of linear algebraic groups, then $\phi(g_s)=\phi(g)_s$ and 
	$\phi(g_u)=\phi(g)_u$.
\end{theorem}
The elements $g_s$ and $g_u$ are called the \emph{semisimple part} and 
the \emph{unipotent part} of $g$ respectively.
\section{Symplectic and Orthogonal Groups}\label{sec2.1}
In this section, we follow Larry C. Grove~\cite{Gv}, and 
define two important classes of groups, which preserve certain bilinear form.
Let $k$ be a field of ${\rm char}\,k \neq 2$. Let $V$ be an $n$-dimensional vector space over $k$.
\begin{definition}
	A \emph{bilinear form} on $V$ is a function $B: V \times V \rightarrow k$ 
	satisfying 
	\begin{enumerate}
		\item $B(u+v,w)=B(u,w)+B(v,w)$
		\item $B(u,v+w)=B(u,v)+B(u,w)$
		\item $B(au,v)=aB(u,v)=B(u,av)$
	\end{enumerate}
	for all $u,v,w \in V $ and all $a\in k$. 
\end{definition}
If $B$ is a bilinear form on $V$ and $\{e_1,e_2,\ldots,e_n\}$ is a basis for $V$, set $b_{ij}:=B(e_i,e_j)$ for 
all $1\leq i,j\leq n$. Then $\beta :=(b_{ij})$ is called the matrix of $B$ relative to $\{e_1,e_2,\ldots,e_n\}$. 
If $u,w \in V$, write $u=\sum_{i}a_ie_i$, and $w=\sum_j b_je_j$, so that $u$ and $w$ are represented by column vectors 
$\textbf{u}={}^{t}(a_1 \cdots a_n)$ and $\textbf{w}={}^{t}(b_1 \cdots b_n)$. Then $B(u,w)={}^{t}{\textbf{u}}
\beta \textbf{w}$ for all $u,w \in V$, where $\textbf{u}, \textbf{w}$ 
are the column vectors with the entries being the 
components of $u,w$ with respect to the given basis $\{e_1,e_2,\ldots,e_n\}$ of $V$. If 
$\{f_1,f_2,\ldots,f_n\}$ is another basis for $V$, write $f_j=\sum_i p_{ij}e_i $, where $p_{ij} \in k$, for all $j=1,2,
\ldots,n$. 
Then $B(f_i,f_j)=\displaystyle \sum_{k,l}p_{ki}B(e_k,e_l)p_{lj}
=\displaystyle \sum_{k,l}p_{ki}b_{kl}p_{lj}$, which is the $(i,j)$-entry of ${}^{t}{P}\beta P$, 
where $P=(p_{ij})\in GL(n,k)$, is the change of basis matrix. 
We say two $n\times n$ matrices $M,N$ are \emph{congruent} 
if $N={}^{t}PMP$, 
for some $P\in GL(n,k)$. So $\mathrm{det}N=\mathrm{det}P \mathrm{det}M \mathrm{det}P$.
Define $k^{\times 2}:=\{a^2 \mid a\in k^{\times}\}$. 
Then $k^{\times 2}$ is a subgroup of $k^{\times}$.
\begin{notation}
	A vector space $V$ having a bilinear form $B$ will be denoted by $(V,B)$.
\end{notation}
\begin{definition}
	Define the \emph{discriminant} of $(V,B)$ to be 
	\[dV:= \left\{
	\begin{array}{ll}
	0 & \text{if} \; \mathrm{det}\beta=0, \\
	(\mathrm{det}\beta) k^{\times 2} & \text{otherwise}.\
	\end{array}\right. \] 
\end{definition}
\noindent
Observe that the discriminant $dV \in k^{\times}/(k^{\times 2})$, is independent of the choice of basis. 
\begin{definition}
	The bilinear form $(V,B)$ is said to be \emph{non-degenerate} if $dV \neq 0$.
\end{definition}
\begin{definition}
	A subspace $W$ of $V$ is said to be \emph{non-degenerate} if $\mathrm{rad}\; W:=W\cap W^{\perp}=\{0\}$, where
	$W^{\perp}=\{v \in V \mid B(w,v)=0 \; \forall w \in W\}$.
\end{definition}
\noindent
Unless otherwise specified, we assume from now on that $(V,B)$ is a non-degenerate bilinear form.
\begin{definition}
	Two bilinear forms $(V_1,B_1)$ and $(V_2,B_2)$ are said to be \emph{equivalent}, 
	denoted by $(V_1, B_1) \approx (V_2, B_2)$, if there 
	exists a vector space isomorphism $\sigma :V_1 \rightarrow V_2$ such that $B_2(\sigma u, \sigma v)=B_1(u,v)$ for all 
	$u,v \in V_1$. 
\end{definition}
\begin{remark}
	We call the above $\sigma$ an \emph{isometry} with respect to $B_1$ and $B_2$.
\end{remark}
\subsection{Symplectic groups}
\begin{definition}
	A bilinear form $B$ is said to be \emph{skew-symmetric} or \emph{alternating} if 
	$ B(u,v)=-B(v,u)$
	for all $u,v \in V$. \\
	Alternatively, this definition is equivalent to $B(u,u)=0$ for all $u \in V$.
	In matrix terminology, the bilinear form $B$ is skew-symmetric if and only if any representing matrix $\beta$ is 
	skew-symmetric, i.e., $\tr \beta=- \beta $.
\end{definition}
\noindent
For the remainder of this section $(V,B)$ will denote a non-degenerate
alternating bilinear form.
\begin{definition}
	A pair $\{u,v\}$ of vectors is said to be a \emph{hyperbolic pair} if $B(u,u)=0=B(v,v)$ and $B(u,v)=1=-B(v,u)$.
\end{definition}
The restriction of $B$ to the subspace generated by $u,v$ has representing matrix 
$\begin{pmatrix}
0 & 1 \\ -1 & 0
\end{pmatrix}
$ relative to $\{u,v \}$.
\begin{proposition}[Theorem 2.10~\cite{Gv}]
	If $B$ is a non-degenerate alternating bilinear form on $V$, then there exists a basis 
	$\{e_1,\ldots,e_l,e_{-1},\ldots,e_{-l}\}$ of $V$ relative to which the representing matrix has the following form
	$\beta=\begin{pmatrix}
	0 & I_{l} \\-I_{l} & 0
	\end{pmatrix}
	$, where $\{e_i,e_{-i}\}$ is a hyperbolic pair for all $i=1,2,\ldots, l$.
\end{proposition}
\begin{definition}[Symplectic group]\label{defsymplecticgroup}
	The \emph{symplectic group} is denoted by 
	$Sp(V,B):=\{T \in GL(V) \mid B(Tu,Tv)=B(u,v) \; \forall u,v \in V \}$.
\end{definition}
\noindent
In matrix terminology, the symplectic group is defined as:
$$ Sp(n,k)=Sp(2l,k):=\{ g \in GL(n,k) \mid \tr g\beta g=\beta \}, $$
where $\beta=\begin{pmatrix}
0 & I_l \\ -I_l & 0
\end{pmatrix}
$.
\begin{definition}[Symplectic similitude group]
	The \emph{symplectic similitude group} with respect to the matrix $\beta$ 
	as in Definition~\ref{defsymplecticgroup}, is defined by
	$GSp(n,k)=\{g\in GL(n,k) \mid \tr g \beta g=\mu(g)\beta, \text{for some}\; \mu(g) \in k^{\times}\}$, 
	where $\mu : GSp(n,k) \rightarrow k^{\times}; g\mapsto \mu(g)$,  
	is a group homomorphism with $\mathrm{ker}\;\mu=Sp(n,k)$; 
	$\mu$ is called a \emph{similitude character}.
\end{definition}
\subsection{Orthogonal groups}
\begin{definition}
	A bilinear form $B$ is said to be \emph{symmetric} if 
	$ B(u,v)=B(v,u)$
	for all $u,v \in V$.
	In matrix terminology, the bilinear form $B$ is symmetric 
	if and only if any representing matrix $\beta$ is symmetric, i.e.,
	$\tr \beta=\beta $.
\end{definition}
\begin{definition}
	If $B$ is a symmetric bilinear form on $V$, then  
	$Q:V \rightarrow k$ defined by $Q(v)=\frac{B(v,v)}{2}$, is called a \emph{quadratic form}
	associated to $B$.
\end{definition}
\noindent
Thus $B(u,v)=Q(u+v)-Q(u)-Q(v)$ for all $u,v \in V$. So the bilinear form $B$ is completely determined by the 
quadratic form $Q$ and vice-versa.

\noindent
For the remainder of this section $(V,B)$ will denote a non-degenerate symmetric bilinear form.
\begin{definition}[Orthogonal group]
	The \emph{orthogonal group} is defined by 
	\begin{align*}
	O(V,B):&=\{T \in GL(V) \mid B(Tu,Tv)=B(u,v) \; \forall u,v \in V \} \\
	&=\{T \in GL(V) \mid Q(Tv)=Q(v) \; \forall v \in V \}.
	\end{align*}         
\end{definition}
\noindent
In matrix terminology, the orthogonal group is defined as:
$$ O(n,k):=\{ g \in GL(n,k) \mid \tr g\beta g=\beta \}.$$
\begin{remark}
	Equivalent forms give conjugate groups in $GL(n,k)$, i.e., if  
	$\beta_2=\tr g\beta_1 g$ for some $g\in GL(n,k)$ then  
	$O(V_2,\beta_2)=g^{-1}O(V_1,\beta_1)g$.
\end{remark}
\begin{definition}
	A vector $v\in V$ is called \emph{isotropic} if $Q(v)=0$, and \emph{anisotropic} 
	if $Q(v) \neq 0$. A vector space $V$ is called 
	\emph{isotropic} if $Q(v)=0$ for some $0 \neq v\in V$ and 
	$V$ is called \emph{totally isotropic} if $Q(v)=0$ for all $v\in V$.
\end{definition}
\begin{definition}
	The dimension of a maximal totally isotropic subspace of a quadratic space is called the \emph{Witt index}. 
\end{definition}
Let $u\in V$ be any non-zero anisotropic vector, and define a linear transformation $\sigma_u$ via 
$$ \sigma_u(v):=v-\frac{2B(u,v)}{B(u,u)}u$$ for all $v\in V$. Then $\sigma_u \in O(V,B)$. We call $\sigma_u$ is the 
\emph{reflection} in the hyperplane orthogonal to $u$. 
The following theorem is well-known, that the orthogonal group is generated by reflections. 
We have (see Theorem 6.6~\cite{Gv}):
\begin{theorem}[E. Cartan-Dieudonn\'{e}]\label{cartandieudonne}
	If $V$ is an $n$-dimensional vector space, equipped with a non-degenerate symmetric bilinear form $B$, 
	then every element of $O(V,B)$ is a product of at most $n$ reflections.
\end{theorem}
\begin{definition}[Orthogonal similitude group]
	The \emph{orthogonal similitude group} with respect to an invertible symmetric matrix $\beta$ is defined by
	$GO(n,k)=\{g\in GL(n,k) \mid \tr g \beta g=\mu(g)\beta, \text{for some}\; \mu(g) \in k^{\times} \}$, 
	where $\mu : GO(n,k)\rightarrow k^{\times}; g\mapsto \mu(g)$, 
	is a group homomorphism with $\mathrm{ker}\;\mu=O(n,k)$; $\mu$ is called a \emph{similitude character}.
\end{definition}
\begin{lemma}\label{kfinite}
	Let $k$ have the property FE, then $k^{\times}/k^{\times 2}$ is finite.
\end{lemma}
\begin{proof}
	If possible suppose that $k^{\times}/k^{\times 2}$ is infinite, then 
	there are infinitely many $a_{i} \in k^{\times}$ not in $k^{\times 2}$. 
	So $k[\sqrt{a_i}]\ncong k[\sqrt{a_j}]$ as field. Hence there are infinitely 
	many field extentions of degree $2$, which contradicts the fact that $k$ has the 
	property FE.
\end{proof}

\begin{example}\label{splitb}
	\begin{enumerate}
		\item Let $V$ be an $n$-dimensional vector space over $\mathbb{C}$ 
		equipped with a non-degenerate symmetric bilinear form $B$. It is known 
		that any two non-degenerate symmetric bilinear forms on $V$ are equivalent, 
		i.e., there is a basis for $V$ relative to which $\beta=I_n$. 
		So the corresponding orthogonal group is denoted by 
		$$O(n, \mathbb{C})=\{g \in GL(n, \mathbb{C})\mid \tra g g=I_n\}.$$
		\item Let $V$ be an $n$-dimensional vector space over $\mathbb{R}$ 
		equipped with a non-degenerate symmetric bilinear form $B$. 
		In this situation, non-degenerate symmetric bilinear forms are 
		classified by their signature, 
		i.e., there is a basis for $V$ relative to which 
		$\beta=\begin{pmatrix}I_r&0\\0&-I_s\end{pmatrix}$. 
		So the corresponding orthogonal groups are denoted by 
		$$O(r,s):=\{g \in GL(n, \mathbb R) \mid \tra g\beta g=\beta \},$$ where $r+s=n$.
		\item Let $V$ be an $n$-dimensional vector space over $\mathbb{F}_q$ 
		equipped with a non-degenerate symmetric bilinear form $B$. Then there 
		is a basis for $V$ relative to which $\beta=\mathrm{diag}(\underbrace{1,\ldots,1}_{n-1},\lambda)$, 
		where $\lambda \in \mathbb{F}_{q}^{\times}$. Thus, up to equivalence, there 
		are two such forms corresponding to a square and non-square elements of $\mathbb{F}_{q}^{\times}$. 
		So the corresponding orthogonal groups are denoted by 
		$$O(n,q)=\{g\in GL(n,q)\mid \tra g \beta g=\beta\}.$$
		\item Let $V$ be an $n$-dimensional vector space over $k$. Up to equivalence, there is a unique non-degenerate symmetric 
		bilinear form $B$ of maximal Witt index over $k$. This is called the \emph{split form}. 
		More explicitly we can fix a basis $\{e_1,\ldots,e_l,e_{-1},\ldots, e_{-l}\}$ for even dimension, and 
		$\{e_{0},e_1,\ldots,e_{l},e_{-1},\ldots,e_{-l}\}$ for odd dimension, so that the matrix 
		of $B$ is as follows:
		\[\beta= \left\{
		\begin{array}{ll}
		\begin{pmatrix}  0 & I_{l} \\ I_{l} & 0  \end{pmatrix} & \text{if}\ n=2l,\\
		\begin{pmatrix}      2 & 0 & 0\\0 & 0 & I_{l}\\0 & I_{l} & 0 \end{pmatrix} & \text{if}\ n=2l+1.
		\end{array}\right. \]                                                                                 
		The orthogonal group corresponding to this form is called a \emph{split orthogonal group}. 
		In this thesis, we will 
		work with only the split orthogonal groups, and this group will be denoted by $O(n,k)$.
	\end{enumerate}
\end{example}
\subsection{Spinor norm}\label{classicalspinornorm}
For $u\in V$ with $Q(u)\neq 0$, we defined the reflection $\sigma_u$ 
by $\sigma_{u}(v)=v-2\frac{B(u,v)}{B(u,u)}u$ along $u$, which is 
an element of the orthogonal group. We know from Theorem~\ref{cartandieudonne} 
that every element of the orthogonal group 
$O(n,k)$ can be written as a product of at most $n$ reflections. 
Let $g\in O(n,k)$ then 
$g =\sigma_{u_{1}}\sigma_{u_{2}}\cdots \sigma_{u_{m}}$ ($m\leq n$),
where $Q(u_i) \neq 0$ for all $i=1,2,\ldots, m$.
We are now in a position to define the spinor norm. 
To show that this is well-defined (see p. 75, Proposition 9.1 in~\cite{Gv}) 
map, we need Clifford algebra theory.
\begin{definition}
	The \emph{spinor norm} is a group homomorphism 
	$\Theta : O(n,k)\rightarrow k^{\times}/k^{\times2}$ defined by  
	$\Theta(g):=\displaystyle \prod_{i=1}^m Q(u_{i})k^{\times2}$, 
	where $g=\sigma_{u_1}\cdots \sigma_{u_m}$.
\end{definition}
Thus for a reflection, we have  $\Theta(\sigma_{u})=Q(u)k^{\times2}$.
However, for computational purposes, this 
definition is difficult to use. 
In Chapter~\ref{chapter4}, we will define the spinor norm using 
Wall's theory, and we will give an efficient algorithm in Chapter~\ref{chapter7} 
to compute the spinor norm.
\section{Unitary Groups}\label{sec2.2}
For the material covered here, we refer to the books~\cite{Kn} and~\cite{Gv}.
Let $R$ be a commutative ring with $1$. An \emph{involution} on $R$ is an automorphism $J: a \mapsto \bar{a}$ 
of $R$ of order $2$. Thus: 
\begin{center}
	$\overline{a+b}=\bar{a}+\bar{b}, \overline{ab}=\bar{a}\bar{b}, \bar{\bar{a}}=a$, 
\end{center}
for all $a,b \in R$. Set $R_{0}:=Fix(J)=\{a\in R \mid \bar{a}=a \}$.
Let $V$ be a free $R$-module of rank $n$. 
In this section, we discuss the unitary groups which are also one of the classical groups. The 
\emph{General Linear Group} $GL(V)$ is a group of all $R$-linear isomorphism of the module 
$V$ over $R$. In matrix terminology it consists of all $n\times n$ invertible matrices and denoted as $GL(n,R)$.
\begin{definition}
	A \emph{sesquilinear form} on $V$, with respect to $J$, is a function $B: V \times V \rightarrow R$ 
	satisfying 
	\begin{enumerate}
		\item $B(u+v,w)=B(u,w)+B(v,w)$
		\item $B(u,v+w)=B(u,v)+B(u,w)$
		\item $B(au,v)= \bar{a} B(u,v)=B(u,\bar{a}v)$
	\end{enumerate}
	for all $u,v,w \in V $ and all $a\in R$. 
\end{definition}
If $B$ is a sesquilinear form on $V$ and $\{e_1,e_2,\ldots,e_n\}$ is a free basis for $V$, set $b_{ij}:=B(e_i,e_j)$ for 
all $1\leq i,j\leq n$. Then $\beta :=(b_{ij})$ is called the matrix of $B$ relative to $\{e_1,e_2,\ldots,e_n\}$. 
If $u,w \in V$, write $u=\sum_{i}a_ie_i$, and $w=\sum_j b_je_j$, 
so that $u$ and $w$ are represented by column vectors 
$\textbf{u}={}^{t}(a_1 \cdots a_n)$ and $\textbf{w}={}^{t}(b_1 \cdots b_n)$. Then $B(u,w)={}^{t}\bar{\textbf{u}}
\beta \textbf{w}$ for all $u,w \in V$, where $\textbf{u}, \textbf{w}$ are the column vectors with the entries being the 
components of $u,w$ with respect to the given basis $\{e_1,e_2,\ldots,e_n\}$ of $V$. If 
$\{f_1,f_2,\ldots,f_n\}$ is another free basis for $V$, write $f_j=\sum_i p_{ij}e_i $, where $p_{ij} \in R$, for all $j=1,2,
\ldots,n$. 
Then $B(f_i,f_j)=\displaystyle \sum_{k,l}\bar{p}_{ki}B(e_k,e_l)p_{lj}
=\displaystyle \sum_{k,l}\bar{p}_{ki}b_{kl}p_{lj}$, which is the $(i,j)$-entry of ${}^{t}\bar{P}\beta P$, 
where $P=(p_{ij})\in GL(n,R)$. We say two $n\times n$ matrices $M,N$ are \emph{congruent} if $N={}^{t}\bar{P}MP$, 
for some $P\in GL(n,R)$. So $\mathrm{det}(N)=\mathrm{det}(P) \overline{\mathrm{det}(P)} \mathrm{det}(M)$.
\noindent
Define $R^{1+J}:=\{a\bar{a}\mid a\in R^{\times}\}$ is a subgroup of $R_{0}^{\times}$.
\begin{notation}
	A free module $V$ having a sesquilinear form $B$ will be denoted by $(V,B)$.
\end{notation}
\begin{definition}
	Define the \emph{discriminant} of $(V,B)$ to be 
	\[dV:= \left\{
	\begin{array}{ll}
	(\mathrm{det}\beta) R^{1+J} & \text{if}\ \;(\mathrm{det}\beta) \in R^{\times},\\
	0 & \text{otherwise}.\ 
	\end{array}\right. \]  
\end{definition}
\noindent
Note that $dV\in R^{\times}/R^{1+J}$ is independent of the choice of basis. 
The sesquilinear form $B$ is said to be \emph{non-degenerate} 
if $dV \in R^{\times}$.

Another way to look at the sesquilinear form is the following.
Denote the dual of $V$ by $V^{*}:=\mathrm{Hom}_{R}(V,R)$.
The form $B$ induces a map $h_{B}:V\rightarrow V^{*}, h_{B}(v)(w)=B(v,w)$ for all $v,w \in V$, which is $R$-linear. 
Conversely, an $R$-linear homomorphism $h:V \rightarrow V^{*}$ defines a sesquilinear form $B_{h}:V \times V \rightarrow 
R, B_{h}(u,v)=h(u)(v)$ for all $u,v \in V$. We call $h_{B}$ the \emph{adjoint} of $B$. Since $h_{B_{h}}=h$ and $B_{h_{B}}=B$,  
a sesquilinear form is determined by its adjoint and vice-versa. If $h_{B}$ is an 
$R$-module isomorphism between $V$ and $V^{*}$ then $B$ is non-degenerate. 
The above two definitions for non-degeneracy are equivalent.
Let $B_1$ and $B_2$ be two sesquilinear forms on $V_1$ and $V_2$ respectively. 
Two forms are said to be \emph{equivalent}, denoted by $(V_1, B_1) \approx (V_2, B_2)$, if there 
exists a $R$-module isomorphism $\sigma :V_1 \rightarrow V_2$ such that $B_2(\sigma u, \sigma v)=B_1(u,v)$ for all 
$u,v \in V_1$. We call $\sigma$ an \emph{isometry} with respect to $B_1$ and $B_2$.
\begin{definition}
	A sesquilinear form $B$ is said to be \emph{hermitian} if $B(u,v)=\overline{B(v,u)}$ for all $u,v \in V$. 
	In matrix terminology, the sesquilinear form $B$ is hermitian 
	if and only if any representing matrix $\beta$ is hermitian, i.e.,
	$\tr \bar{\beta}=\beta $.
\end{definition}
\begin{definition}
	A sesquilinear form $B$ is said to be \emph{skew-hermitian} 
	if $B(u,v)=- \overline{B(v,u)}$ for all $u,v \in V$. 
	In matrix terminology, the sesquilinear form $B$ is skew-hermitian 
	if and only if any representing matrix $\beta$ is skew-hermitian, i.e.,
	$\tr \bar{\beta}=-\beta $.
\end{definition}
\begin{remark}
	If $B$ is a skew-hermitian form, then $B_1:=aB$ is a hermitian form for some $a\in R^{\times}$ with $\bar{a}=-a$.
	So the corresponding isometry group will be same whether we consider hermitian or skew-hermitian form.
\end{remark}

\noindent
For the remainder of this section $(V,B)$ 
will denote a non-degenerate hermitian form.
\begin{definition}[Unitary group]
	The \emph{unitary group} is defined as follows:
	$U(V,B):=\{T \in GL(V) \mid B(Tu,Tv)=B(u,v) \; \forall u,v \in V \}$.
\end{definition}
\noindent
In matrix terminology, the unitary group is defined as:
$$ U(n,R_0):=\{ g \in GL(n,R) \mid \tr \bar{g} \beta g=\beta \}.$$
\noindent
Most of the time we will consider unitary groups over fields.
\begin{definition}[Unitary similitude group]
	The \emph{unitary similitude group} with respect to an invertible hermitian matrix $\beta$ is defined by
	$GU(n,k_0)=\{g\in GL(n,k) \mid \tr \bar{g} \beta g=\mu(g)\beta, \text{for some}\; \mu(g) \in k_{0}^{\times} \}$, 
	where $\mu : GU(n,k_0)\rightarrow k_{0}^{\times}; g\mapsto \mu(g)$, 
	is a group homomorphism with $\mathrm{ker}\;\mu=U(n,k_0)$; $\mu$ is called a \emph{similitude character}.
\end{definition}
\begin{example}\label{splitu}
	\begin{enumerate}
		\item Let $V$ be an $n$-dimensional vector space over $\mathbb C$ with $\overline{a+ib}=a-ib$. 
		In this situation, hermitian forms are classified by signature and 
		given by $\beta= \begin{pmatrix} I_r &0 \\0 & -I_s \end{pmatrix}$. 
		So the corresponding unitary groups are denoted by 
		$$U(r,s):=\{g \in GL(n, \mathbb C) \mid \tra \bar{g}\beta g=\beta \},$$
		where $\bar{g}:=(\bar{g}_{ij})$, where $\bar{g}_{ij}$ is the usual complex conjugate
		and $r+s=n$.
		\item Let $V$ be an $n$-dimensional vector space over a finite field $\mathbb{F}_{q^2}$ with $J:a \mapsto a^q$. 
		It is known that any two hermitian forms on $V$ are equivalent and thus we may choose $\beta=I_n$. 
		So the corresponding unitary group is, unique up to conjugation, and is  
		denoted by 
		$$U(n,q):=\{g \in GL(n, q^2) \mid \tr \bar{g}g=I_n\},$$ 
		where $\bar{g}:=(g_{ij}^{q})$. 
		\item Let $V$ be a free module over $R=k\times k$ of rank $n$ with 
		$J: (a,b)\mapsto \overline{(a,b)}=(b,a)$. Then $R_{0}=\{(a,b)\in R \mid (b,a)=(a,b) \}=\mathrm{diag}(k\times k)
		\cong k$. Then the unitary group defined over $R_{0}$ is 
		$$U(n,k)=\{g\in GL(n,R) \mid \tr \bar{g}\beta g=\beta\},$$ where 
		$g=(A,B)\in M(n,k)\times M(n,k)$ and $\tr \bar{g}=(\tr B, \tr A)$, and $\beta=(\beta_1, \beta_2)$. 
		In particular, if $\beta=(I_n, I_n)$ then $U(n,k)=\{(A,B) \in M(n,k)\times M(n,k) \mid \tr AB=I_n\} \cong GL(n,k)$.
		\item Let $V$ be an $n$-dimensional vector space over $k$ 
		with an involution $J:a \mapsto \bar{a}$.  
		Up to equivalence, there is a unique non-degenerate hermitian 
		form $B$ of maximal Witt index over $k$. This is called the \emph{split form}. 
		More explicitly we can fix a basis $\{e_1,\ldots,e_l,e_{-1},\ldots, e_{-l}\}$ for even dimension, and 
		$\{e_{0},e_1,\ldots,e_{l},e_{-1},\ldots,e_{-l}\}$ for odd dimension, so that the matrix 
		of $B$ is as follows:
		\[\beta= \left\{
		\begin{array}{ll}
		\begin{pmatrix}  0 & I_{l} \\ I_{l} & 0  \end{pmatrix} & \text{if}\ n=2l,\\
		\begin{pmatrix}      2 & 0 & 0\\0 & 0 & I_{l}\\0 & I_{l} & 0 \end{pmatrix} & \text{if}\ n=2l+1.
		\end{array}\right. \]                                                                                 
		The unitary group corresponding to this form is called a \emph{split unitary group}. 
		In Section 6.4 of Chapter~\ref{chapter6} we will 
		work with only the split unitary groups, and this group 
		will be denoted by $U(n,k_0)$, where $k_0$ is the fixed field.
	\end{enumerate}
\end{example}
%%%%%%%%%%%%%%%%%%%%%%%%%%%%%%%%%%%%%%%%%%%%%%%%%%%%%%%%%%%%%%%%%%
\chapter{Chevalley Groups}\label{chapter3}
This is another basic chapter of this thesis.
In the present chapter, we will 
take another approach to define the split classical groups. 
For the Gaussian elimination, which we will develop in Chapter~\ref{chapter6}, 
we need an analog of elementary matrices. These matrices are described in Section 3.2, 
which come from the theory of Chevalley groups (of adjoint type). 
In this theory, one begins with a complex simple 
Lie algebra $\mathfrak{g}$, a field $k$, and get 
a group $G(k)$ (see Section 3.1).
The theory was developed by Chevalley~\cite{Ch} himself, 
and further generalized by Robert Steinberg~\cite{St1}.
In our computations, we often imitate the notation from Carter~\cite{Ca1}.
\section{Construction of Chevalley Groups (adjoint type)}
Let $\mathfrak{g}$ be a complex simple Lie algebra. 
Since any two Cartan subalgebras of $\mathfrak{g}$ are conjugate, we fix a 
Cartan subalgebra $\mathfrak{h}$.
Then there is the adjoint representation of $\mathfrak{g}$, 
$$\mathrm{ad} : \mathfrak{g} \rightarrow \mathfrak{gl}(\mathfrak{g}) $$ 
given by 
$\mathrm{ad}X(Y)=[X,Y]$.
Since $\mathfrak{h}$ is Abelian, $\mathrm{ad}(\mathfrak{h})$ is a 
commuting family of semisimple linear transformations of $\mathfrak{g}$. 
Hence $\mathrm{ad}(\mathfrak{h})$ is simultaneously diagonalizable. 
Thus we have (see p.35~\cite{Ca1}): 
\begin{theorem}[Cartan decomposition]\label{cartandecomposition}
	With this notation, we have,  
	$$\mathfrak{g}=\mathfrak{h}\bigoplus \displaystyle\sum_{\alpha \in \Phi}\mathfrak{g}_{\alpha},$$ 
	where $\mathfrak{g}_{\alpha}=\{X\in \mathfrak{g}\mid \mathrm{ad}H(X)=\alpha(H)X, \; \forall H \in \mathfrak{h}\}$ 
	are root spaces and $\Phi$ is a root system with respect to $\mathfrak{h}$. 
\end{theorem}
We call this decomposition the \emph{Cartan decomposition} of $\mathfrak{g}$ with respect to $\mathfrak{h}$.
The classification of finite dimensional complex simple 
Lie algebras gives four infinite families 
$A_l (l\geq 1), B_l (l\geq 2), C_l (l\geq 3)$ and $D_l (l\geq 4)$ called 
\emph{classical types}, and five \emph{exceptional types} $G_2, F_4, E_6, E_7$ and $E_8$.
Chevalley proved that, there exists a basis of $\mathfrak{g}$ such that 
all the structure constants, which define $\mathfrak{g}$ as a Lie algebra, are integers. 
The following (Theorem 4.2.1~\cite{Ca1}) is 
a key theorem to define Chevalley groups.
\begin{theorem}[Chevalley basis theorem]\label{Chevalleybasistheorem}
	Let $\mathfrak{g}$ be a simple Lie algebra over $\mathbb{C}$, 
	$\mathfrak{h}$ be a Cartan subalgebra, and 
	$$ \mathfrak{g}=\mathfrak{h}\bigoplus \displaystyle\sum_{\alpha \in \Phi}\mathfrak{g}_{\alpha}$$ 
	be a Cartan decomposition of $\mathfrak{g}$. Let $h_{\alpha} \in \mathfrak{h}$ 
	be the co-root corresponding to the root $\alpha$. Then, for each root $\alpha \in \Phi$,
	an element $e_{\alpha}$ can be chosen in $\mathfrak{g}_{\alpha}$ such that 
	\begin{align*}
	[ e_{\alpha}, e_{-\alpha} ]&=h_{\alpha},\\
	[ e_{\alpha}, e_{\beta} ]&=\pm (r+1)e_{\alpha + \beta},
	\end{align*}
	where $r$ is the greatest integer for which $\beta-r \alpha \in \Phi$.
	
	The elements $\{h_{\alpha},\alpha \in \Pi; e_{\alpha}, \alpha \in \Phi\}$ 
	form a basis for $\mathfrak{g}$, called a \textbf{Chevalley basis}.
	The basis elements multiply together as follows:
	\begin{equation}\label{integers} 
	\left.
	\begin{split}
	[ h_{\alpha},h_{\beta} ] &= 0, \\
	[ h_{\alpha},e_{\beta} ] &= A_{\alpha \beta}e_{\beta}, \\
	[ e_{\alpha},e_{-\alpha} ] &= h_{\alpha}, \\
	[ e_{\alpha},e_{\beta} ] &= 0 & \text{if $\alpha +\beta \notin \Phi$}, \quad\\
	[ e_{\alpha},e_{\beta} ] &= \pm (r+1)e_{\alpha +\beta} & \text{if $ \alpha +\beta \in \Phi $}, \quad
	\end{split}
	\right\}
	\end{equation}
	where $A_{\alpha \beta}$ are Cartan integers and $\Pi$, a simple root system fixed for $\Phi$.
	
	The structure constants of the algebra with respect to a Chevalley basis are all integers.
\end{theorem}
The map $\mathrm{ad}e_{\alpha}$ is a nilpotent linear map on 
$\mathfrak{g}$. Let $t \in \mathbb{C}$, then $\mathrm{ad}(te_{\alpha})=
t(\mathrm{ad}e_{\alpha})$ is also nilpotent.
Thus $\mathrm{exp}(t(\mathrm{ad}e_{\alpha}))$ is an automorphism of 
$\mathfrak{g}$. 
%We define $x_{\alpha}(t):=\mathrm{exp}(t\mathrm{ad}e_{\alpha}) \in \mathrm{Aut}(\mathfrak{g})$.
We denote by $\mathfrak{g}_{\mathbb{Z}}$ the subset of $\mathfrak{g}$ of 
all $\mathbb{Z}$-linear combinations of the Chevalley basis elements of $\mathfrak{g}$.
By Equation (\ref{integers}), a  
Lie bracket can be defined for $\mathfrak{g}_{\mathbb{Z}}$. 
Thus $\mathfrak{g}_{\mathbb{Z}}$ is a Lie algebra over $\mathbb{Z}$.
Now let $k$ be any field. We define $\mathfrak{g}_{k}:=\mathfrak{g}_{\mathbb{Z}}\otimes_{\mathbb{Z}} k$.
Then $\mathfrak{g}_{k}$ is a Lie algebra 
over $k$ via the Lie multiplication 
$$[X\otimes 1_k, Y\otimes 1_k]:=[X,Y]\otimes 1_k,$$  
where $X,Y$ are Chevalley basis elements of $\mathfrak{g}$, and 
$1_k$ denote the identity element of $k$.

Now everything makes sense over an arbitrary field 
$k$. So we are in a position to define the 
Chevalley groups of adjoint type.
The Chevalley group of type $\mathfrak{g}$ over the field $k$, denoted by 
$G(k)$, is defined to be the subgroup of automorphisms of the Lie 
algebra $\mathfrak{g}_k$ generated by $\mathrm{exp}(t(\mathrm{ad}e_{\alpha}))$ for all 
$\alpha \in \Phi, t \in k$. In fact, the group $G(k)$ over $k$ 
is determined up to isomorphism by the simple Lie algebra 
$\mathfrak{g}$ over $\mathbb{C}$ and the field $k$.

Observe that (see Lemma 4.5.1, p.65~\cite{Ca1}), when $\mathfrak{g}$ is a linear Lie algebra,
$$\mathrm{exp}(t(\mathrm{ad}e_{\alpha}))X=\mathrm{exp}(te_{\alpha})X\mathrm{exp}(te_{\alpha})^{-1},$$ 
for all $X\in \mathfrak{g}_k$, for all $\alpha \in \Phi$, and for all $t\in k$.
We shall abuse the notation slightly 
and denote the matrix of the linear map by 
$\mathrm{exp}(te_{\alpha})$ itself.
Define $x_{\alpha}(t):=\mathrm{exp}(te_{\alpha})$. 
We call the $x_{\alpha}(t)$, \emph{elementary matrix}.
Let $\tilde{G}(k)$ be the group of matrices 
generated by the elements $x_{\alpha}(t)$ for all $\alpha \in \Phi$ 
and all $t\in k$. Thus there is a homomorphism 
from $\tilde{G}(k)$ onto $G(k)$ such that 
$$\mathrm{exp}(te_{\alpha})\mapsto \mathrm{exp}(t(\mathrm{ad}e_{\alpha}))$$
whose kernel is the center of $\tilde{G}(k)$.
Hence $\frac{\tilde{G}(k)}{\mathcal{Z}(\tilde{G}(k))}\cong G(k)$.
We work with $\tilde{G}(k)$ instead of the Chevalley group $G(k)$.
% The elements $\{ exp(te_{\alphpa})\mid \alpha \in \Phi, t\in k \}$ are called elementary matrices.
In (see Section 11.2~\cite{Ca1}), the classical Lie algebras 
and their Chevalley basis are described explicitly. 
Usually, row-column operations are defined by pre and post 
multiplication by certain elementary matrices. We are going to define the elementary matrices 
for symplectic and orthogonal groups, and more generally, for 
symplectic and orthogonal similitude groups.
\begin{example}[Cartan decomposition and Chevalley basis of $\mathfrak{sp}(2l, \mathbb{C})$]
	Let us consider the Lie algebra of type $C_l$:
	$$\mathfrak{g}:=\mathfrak{sp}(2l, \mathbb{C})=\{X\in \mathfrak{gl}(2l,\mathbb{C}) \mid \tr X\beta +\beta X=0\},$$
	where $\beta=\begin{pmatrix}0 & I_l \\ -I_l & 0 \end{pmatrix}$. We can write elements 
	of $\mathfrak{g}$ in block form. Let $X=\begin{pmatrix}A&B\\C&D\end{pmatrix} \in \mathfrak{g}$, where 
	$A, B, C, D$ are $l\times l$ matrices. 
	We use the condition that $X$ satisfies $\tr X\beta +\beta X=0$, then we get $\tr B=B, \tr C=C$ and $D=-\tr A$. The 
	set of diagonal matrices in $\mathfrak{g}$ is a Cartan subalgebra $\mathfrak{h}$ of $\mathfrak{g}$. The 
	elements of $\mathfrak{h}$ have form $H=\mathrm{diag}(\lambda_1,\ldots
	,\lambda_l,-\lambda_1,\ldots,-\lambda_l)$. We index the rows and columns by 
	$1,\ldots,l$ and $-1,\ldots,-l$. The elements $H_i=e_{ii}-e_{-i,-i}, 1\leq i \leq l$, form a 
	basis of $\mathfrak{h}$. 
	Then by Theorem~\ref{cartandecomposition}, we have, 
	$$\mathfrak{g}=\mathfrak{h}\oplus \displaystyle\sum_{\alpha \in \Phi}\mathbb{C}e_{\alpha},$$
	where 
	\begin{equation}\label{rootvectors}
	e_{\alpha} =\left\{
	\begin{split}
	& e_{ij}-e_{-j,-i}, & \quad 1\leq i\neq j\leq l,\\
	& e_{i,-j}+e_{j,-i}, & \quad 1\leq i< j\leq l,\\
	& e_{-i,j}+e_{-j,i}, & \quad 1\leq i< j\leq l,\\
	& e_{i,-i}, & \quad 1\leq i\leq l,\\
	& e_{-i,i}, & \quad 1\leq i\leq l.
	\end{split}
	\right.                
	\end{equation}
	The above decomposition is the Cartan decomposition of the Lie algebra 
	$\mathfrak{g}=\mathfrak{sp}(2l, \mathbb{C})$, and 
	a Chevalley basis for this Lie algebra is 
	$$\{H_i=e_{ii}-e_{-i,-i}, 1\leq i\leq l; e_{\alpha}, \alpha \in \Phi\},$$ where 
	$e_{\alpha}$ as in Equation (\ref{rootvectors}). 
	Observe that the above mentioned Chevalley basis is not unique, in fact, any 
	integral multiple of it is again a Chevalley basis.
	Now $e_{\alpha}$'s are nilpotent endomorphisms of 
	$\mathfrak{g}$ with $e_{\alpha}^2=0$. So $x_{\alpha}(t)=\mathrm{exp}\;(te_{\alpha})=I+te_{\alpha}$, elementary 
	matrix, is an automorphism of $\mathfrak{g}$. 
	Similarly, we can do this, for orthogonal Lie algebras. 
	For details see~\cite{Ca1}. In next section, we define these matrices explicitly.
\end{example}
\section{Elementary Matrices}\label{elementarymatrices}
First of all, let us describe the elementary matrices 
for symplectic and split orthogonal similitude groups. The genesis 
of these elementary matrices lies in the Chevalley basis theorem (Theorem~\ref{Chevalleybasistheorem}).
In what follows, the scalar $t$ varies over the field $k$, $n=2l$ or $n=2l+1$, and $1\leq i,j \leq l$. 
We define $te_{i,j}$ as the $n\times n$ matrix with $t$ in the $(i,j)$ position, and zero everywhere else. 
We simply use $e_{i,j}$ to denote $1e_{i,j}$. 
We often use the well-known matrix identity 
$e_{i,j}e_{k,l}=\delta_{j,k}e_{i,l}$, where 
$\delta_{j,k}$ is the Kronecker delta.
For more details on elementary matrices see~\cite{Ca1}. 
\begin{example}
	Elementary matrices (or elementary transvections) in $SL(n,k)$ are 
	$x_{ij}(t):=I+te_{ij}$, where $t \in k; 1\leq i\neq j\leq n$.
\end{example}
\subsection{Elementary matrices for $GSp(2l,k)\; (l\geq 2)$}
We index rows and columns by $1, \ldots, l, -1, \ldots, -l$. 
The elementary matrices are as follows:
\begin{eqnarray*}
	x_{i,j}(t)=&I+t(e_{i,j}-e_{-j,-i}) &\text{for}\;  i\neq j,\\
	x_{i,-j}(t)=&I+t(e_{i,-j}+e_{j,-i})&\text{for}\;  i<j,\\ 
	x_{-i,j}(t)=&I+t(e_{-i,j}+e_{-j,i})& \text{for}\;  i<j,\\
	x_{i,-i}(t)=&I+te_{i,-i}, \\
	x_{-i,i}(t)=&I+te_{-i,i}, 
\end{eqnarray*}
and in matrix format they look as follows:
\begin{align*}
E1 &: \begin{pmatrix}
R & 0 \\ 0 & \tr R^{-1} 
\end{pmatrix}, \text{where}\; R=I+te_{i,j}; i\neq j, \\
E2 &: \begin{pmatrix}
I & R \\ 0 & I
\end{pmatrix}, \text{where}\; R=t(e_{i,j}+e_{j,i}); \text{for} \;i < j\; \text{or}\; te_{i,i},\\
E3 &: \begin{pmatrix}
I & 0 \\ R & I
\end{pmatrix}, \text{where}\; R=t(e_{i,j}+e_{j,i}); \text{for}\; i < j\; \text{or}\; te_{i,i}.
\end{align*}
\subsection{Elementary matrices for $GO(2l,k)\; (l\geq 2)$}\label{elementarymate}
We index rows and columns by $1, \ldots, l, -1, \ldots, -l$. 
The elementary matrices are as follows:
\begin{align*}
x_{i,j}(t) &= \hspace{12mm} I + t(e_{i,j} - e_{-j,-i}) \hspace{12mm} \text{for $i \neq j$},\\
x_{i,-j}(t) &= \hspace{12mm} I + t(e_{i,-j} - e_{j,-i}) \hspace{12mm} \text{for $i < j$},\\
x_{-i,j}(t) &= \hspace{12mm} I + t(e_{-i,j} - e_{-j,i}) \hspace{12mm} \text{for $i < j$},\\
w_l &= \quad I - e_{l,l} - e_{-l,-l} - e_{l,-l} - e_{-l,l}, 
\end{align*}
and in matrix format they look as follows:
\begin{align*}
E1 &: \begin{pmatrix}
R & 0 \\ 0 & \tr R^{-1}
\end{pmatrix}, \text{where}\, R=I+te_{i,j}; i\neq j, \\
E2 &: \begin{pmatrix}
I & R \\ 0 & I
\end{pmatrix}, \text{where}\, R=t(e_{i,j}-e_{j,i}); \text{for} \, i < j, \\
E3 &: \begin{pmatrix}
I & 0 \\ R & I
\end{pmatrix}, \text{where}\, R=t(e_{i,j}-e_{j,i}); \text{for}\, i < j.
\end{align*}
\subsection{Elementary matrices for $GO(2l + 1, k)\; (l \geq 2)$}\label{elementarymato}
We index rows and columns by $0, 1,\ldots , l, -1, \ldots, -l$. 
The elementary matrices are as follows:
\begin{align*}
x_{i,j}(t) &= \hspace{12mm} I + t(e_{i,j} - e_{-j,-i} ) \hspace{12mm} \text{for $i \neq j$},\\
x_{i,-j}(t) &= \hspace{12mm} I + t(e_{i,-j} - e_{j,-i}) \hspace{12mm} \text{for $i < j$},\\
x_{-i,j}(t) &= \hspace{12mm} I + t(e_{-i,j} - e_{-j,i}) \hspace{12mm} \text{for $i < j$},\\
x_{i,0}(t) &= \quad I + t(2e_{i,0} - e_{0,-i} ) - t^2 e_{i,-i}, \\
x_{0,i}(t) &= \quad I + t(-2e_{-i,0} + e_{0,i}) - t^2 e_{-i,i}, 
\end{align*}
\begin{align*}
w_l &= \quad I - e_{l,l} - e_{-l,-l} - e_{l,-l} - e_{-l,l}, 
\end{align*}
and in matrix format they look as follows:
\begin{align*}
E1 &: \begin{pmatrix}
1&0&0 \\0&R&0 \\0& 0 & \tr R^{-1}
\end{pmatrix}, \text{where}\, R=I+te_{i,j}; i\neq j, \\
E2 &: \begin{pmatrix}
1&0&0\\0&I & R \\0& 0 & I
\end{pmatrix}, \text{where}\, R=t(e_{i,j}-e_{j,i}); \text{for}\, i < j, \\
E3 &: \begin{pmatrix}
1&0&0\\0&I & 0 \\0& R & I
\end{pmatrix}, \text{where}\, R=t(e_{i,j}-e_{j,i}); \text{for}\, i < j, \\
E4a &: \begin{pmatrix}1&0&-R \\2\tr R&I&-\tr RR \\0&0&I\end{pmatrix}, \text{where}\, R=te_i, \\
E4b &:\begin{pmatrix}1&R&0\\0&I&0\\-2\tr R&-\tr RR&I\end{pmatrix}, \text{where}\, R=te_i. 
\end{align*}
Here $e_i$ is the row vector with $1$ at $i^{\text{th}}$ place and zero elsewhere.

In~\cite{Re}, Ree proved that the above defined elementary matrices generate the symplectic group 
$Sp(2l,k)$ and the commutator subgroups of the orthogonal groups $O(2l,k)$ and $O(2l+1, k)$ respectively.
We will give an algorithmic proof of this fact via our Gaussian elimination algorithm (see Theorem~\ref{maintheorem1}). 
%%%%%%%%%%%%%%%%%%%%%%%%%%%%%%%%%%%%%%%%%%%%%%%%%%%%%%%%%%%%%%%%%%
\chapter{Conjugacy Classes of an Isometry}\label{chapter4}
G. E. Wall~\cite{Wa2}, and Springer-Steinberg~\cite{SS} classified isometries 
with respect to a symmetric, skew-symmetric and hermitian forms up to conjugacy. They 
associated certain forms to an isometry.
In this chapter, we define those forms associated with an element in the isometry group. 
In Section 4.1 we describe the Wall's form, 
which is associated with an 
orthogonal element, which will be used in Chapter~\ref{chapter7} to compute the 
spinor norm. In Section 4.2, we describe the other form associated to an element of the 
unitary group, which will be used in Chapter~\ref{chapter8} to prove the finiteness 
of $z$-classes in unitary group. The material in this chapter is based on 
the work of Wall and Springer-Steinberg, and is presented here for the sake of completeness.
\section{Wall's Form} 
In~\cite{Wa2}, Wall classified conjugacy classes in classical groups by associating a bilinear 
form and thus reducing the problem of conjugacy to the equivalence of bilinear forms.
Let $g \in O(n,k)$ and define 
the \emph{residual space} of $g$ by 
$V_{g}:=(1_{V}-g)(V)$, where $1_V$ denotes the identity linear map on $V$. 
Observe that $V_g$ is $\langle g \rangle$-stable.
\begin{definition}
	An element $g\in O(n,k)$ is said to be \emph{regular} if the residual space 
	$V_{g}$ is non-degenerate.  
\end{definition}
\begin{example}
	The reflection $\sigma_u \in O(n,k)$ is an example of a regular element, as 
	$V_{\sigma_u}=\langle u \rangle$, which is non-degenerate.
\end{example}
If $g\in O(n,k)$ then we have,
\begin{equation}\label{walleqnwell}
B((1_{V}-g)x,y)+B(x,(1_{V}-g)y)=B((1_{V}-g)x,(1_{V}-g)y)
\end{equation}
for all $x,y \in V$. 
This defines a map 
$$[,]_{g}:V_{g} \times V_{g} \rightarrow k \; \text{by} \; [u,v]_{g}:=B(u,y)$$ 
for all $u,v \in V_{g}$, 
where $v=(1_{V}-g)(y)$ for some $y \in V$. Thus to $g\in O(n,k)$ we associate $(V_g, [,]_g)$, called 
\emph{Wall's form}. We have (see p.6~\cite{Wa2}):
\begin{proposition}\label{welldefined}
	The map $[,]_{g}$ is a well-defined non-degenerate bilinear form on $V_{g}$, and $g$ is an isometry on $V_{g}$ with 
	respect to $[,]_{g}$. Furthermore, we have
	\begin{enumerate}
		\item $[u,v]_{g}+[v,u]_{g}=B(u,v)$,
		\item $[u,v]_{g}=-[v, gu]_{g}$
	\end{enumerate}
	for all $u,v \in V_{g}$.
\end{proposition}
\begin{proof}
	Let $u=(1_V-g)x=(1_V-g)x_1$ and $v=(1_V-g)y=(1_V-g)y_1$ be in $V_g$ 
	for some $x,y,x_1,y_1 \in V$.
	%then $[u,v]_g=[(1_v-g)x,(1_V-g)y]_g=B((1_V-g)x,y)$.
	%Also suppose that $u=(1_V-g)x_1$ and $v=(1_V-g)y_1$ for some $x_1, y_1 \in V$.
	Then we have 
	\begin{align*}
	B((1_V-g)x,y) &  =  B((1_V-g)x_1,y) \\
	& =  B((1_V-g)x_1,(1_V-g)y)-B(x_1,(1_V-g)y) & (\text{by}\;(\ref{walleqnwell}))\\
	& =  B((1_V-g)x_1,(1_V-g)y_1)-B(x_1,(1_V-g)y_1) \\
	& =  B((1_v-g)x_1,y_1) & (\text{by}\;(\ref{walleqnwell})).
	\end{align*}                 
	So the map $[,]_g$ is well-defined.
	Let $u\in V_g$, and if 
	$[u,v]_g=0$ for all $v \in V_g$, then $B(u, y)=0$ for all $y \in V$, 
	which implies $u=0$, as $B$ is nondegenerate. 
	Hence $[,]_g$ is nondegenerate.
	It follows immediately that $[,]_g$ is a bilinear form on $V_g$, as $B$ is so.
	Now $[g |_{V_g}u,g |_{V_g}v]=[g |_{V_g}u,g |_{V_g}(1_{V}-g)y]
	=B(g |_{V_g}u,g |_{V_g}y)=B(u, y)=[u,(1_{V}-g)y]_g=[u, v]_g$. 
	Hence $g$ is an isometry on $V_g$ with respect to the new form 
	$[,]_g$. Furthermore, 
	%\begin{enumerate}
	%\item 
	let $u,v \in V_{g}$ then $u=(1_{V}-g)x$ and $v=(1_{V}-g)y$ for some $x,y \in V$. 
	We have 
	\begin{align*}
	B(u,v)&=B(x-gx, y-g y) \\
	&=B(x,y)-B(x,g y)+B(x,y)-B(g x,y) \\
	&=B(x, (1_{V}-g)y)+B((1_{V}-g)x, y) \\
	%&=B(x, v)+B(u,y) \\
	&=B(v,x)+B(u,y) \\
	&=[v, (1_{V}-g)x]_{g} + [u,(1_{V}-g)y]_{g} \\ 
	&=[v, u]_{g} + [u, v]_{g}.
	\end{align*}
	Hence $[u, v]_{g}+[v, u]_{g}=B(u, v)$, which proves (1). 
	%\item 
	Now we have 
	\begin{align*}
	[(1_{V}-g)x, (1_{V}-g)y]_{g}&=B((1_{V}-g)x, y) \\
	&=-B((1_{V}-g)y, g x) \\
	&=-[(1_{V}-g)y, (1_{V}-g)g x]_{g}.
	\end{align*}
	%\begin{align*}
	%[v, g u]_{g}&=[(1_{V}-g)y, (1_{V}-g)g x]_{g} \\ 
	%&=B((1_{V}-g)y, g x) \\
	%&=-B((1_{V}-g)x, y) \\
	%&=-[(1_{V}-g)x, (1_{V}-g)y]_{g}.
	%\end{align*}
	Therefore $[u, v]_{g}=-[v, g u]_{g}$, proving (2). 
	Hence the Proposition.
	%\end{enumerate}
\end{proof}
We have seen that the Wall's form $[,]_g$ is always non-degenerate, but 
need not be symmetric. Here we give a criterion for the Wall's form $[,]_g$ 
to be symmetric. We have (see p.116~\cite{Ha}):
\begin{proposition}
	The Wall's form $[,]_g$ is symmetric if and only if $g^2=Id$.
\end{proposition}
\begin{proof}
	Suppose $[,]_g$ is symmetric, then $[u,v]_g=[v,u]_g$ for all $u,v \in V_g$. 
	Then by part 2 of the Proposition~\ref{welldefined}, we have 
	$-[v,gu]_g=[v,u]_g$. So $[v,(1_V +g)u]_g=0$ for all $u,v \in V_g$, which implies that 
	$(1_V+g)(1_V-g)x=0$ for all $x\in V$, where $u=(1_V-g)x$. Hence $g^2=Id$.
	
	Conversely, suppose that $g^2=Id$, then we have 
	\begin{align*}
	[u,v]_g &=-[v,gu]_g & (\text{by (2) of Proposition~\ref{welldefined}}) \\
	&=-[v,g(1_V-g)x]_g & (\text{where $u=(1_V-g)x$ for some $x \in V$})\\
	&=-B(v,gx) \\
	&=-B(gv,x) & (\text{since $g^2=Id$})\\
	&=-B(g(1_V-g)y,x) & (\text{where $v=(1_V-g)y$ for some $y\in V$})\\
	&=B((1_V-g)y,x) & (\text{since $g^2=Id$})\\
	&=[(1_v-g)y,(1_V-g)x]_g \\
	&=[v,u]_g.
	\end{align*}
	Therefore the form $[,]_g$ is symmetric.
\end{proof}
Wall developed this to classify the conjugacy class of $g$. We have (see Theorem 1.3.1~\cite{Wa2}): 
\begin{proposition}
	Let $g,h \in O(n,k)$. Then $g$ is conjugate to $h$ in $O(n,k)$ if and only if 
	$(V_g, [,]_g)\approx (V_h, [,]_h)$.
\end{proposition}
\noindent
Now the residual space $V_{g}$ is equipped with two bilinear forms: 
\begin{enumerate}
	\item The Wall's form $(V_{g}, [,]_{g})$, for this we use the notation $V_g$.
	\item Restriction of the usual form $B$ on $V_g$ is, denoted by $(V_{g}, B)$.
\end{enumerate}
\subsection{Spinor norm using Wall's theory}
We will now define the spinor norm using Wall's theory, which will be useful for our purpose.
\begin{definition}
	The \emph{spinor norm} is a group homomorphism  
	$\Theta _{W}:O(n,k)\rightarrow k^{\times}/k^{\times2}$ defined by  
	$\Theta_{W}(g)=(dV_{g})k^{\times2}$, where $V_g$ is defined as above.
\end{definition}
Let $\sigma_{u}$ be a reflection in $O(n,k)$.
Then the residual space is $V_{\sigma_{u}}=\langle u \rangle$, 
therefore $dV_{\sigma_{u}}=\mathrm{det}([u, u]_{\sigma_{u}})=Q(u)$. 
Hence $\Theta _{W}(\sigma_{u})=Q(u)k^{\times2}$, 
which is same as the spinor norm computed in Section~\ref{classicalspinornorm}.
The following Proposition and its Corollary are due to A. J. Hahn~\cite{Ha}. 
We include the proof for the sake of completeness. 
\begin{proposition}\label{hahn}
	Let $g \in O(n,k)$ be regular with residual space $V_{g}$. 
	Then $\Theta_{W}(g)=(\mathrm{det}(1_{V}-g)|_{V_{g}})(d(V_g ,B))k^{\times2}$.
\end{proposition}
\begin{proof}
	If $V_{g}=\{0\}$, then $\mathrm{det}((1_{V}-g)|_{V_{g}})=1$, since $\mathrm{det}(g|_{\{0\}})=1$, 
	for any $g \in \mathrm{Hom}_{k}(V, V)$ 
	and $dV_{g}=k^{\times2}$. Hence, in this case the result follows immediately. 
	Suppose now $V_{g} \neq \{0\}$. Since $g$ is regular, $V_{g}\cap V_{g}^{\perp}=\{0\}$, and 
	$\mathrm{ker}(1_{V}-g)=V_{g}^{\perp}$. 
	So $\mathrm{ker}(1_{V}-g)|_{V_{g}}=V_{g}\cap \, \mathrm{ker}(1_{V}-g)=V_{g}\cap V_{g}^{\perp}=\{0\}$. 
	Hence $(1_{V}-g)|_{V_{g}} \in GL(V_g)$. 
	Therefore $[u, v]_g=B(u, (1_{V}-g)|_{V_{g}}^{-1}v)$ for all $u,v \in V_{g}$. 
	Fix any basis for $V_{g}$, say $\{e_{1}, e_{2},\ldots, e_{r}\}$. 
	Let $M, N$ be the matrices corresponding to the forms
	$[,]_g$ and $B$ respectively, and let $T$ be the matrix 
	corresponding to the linear transformation $(1_{V}-g)|_{V_{g}}$
	with respect to the above mentioned basis. 
	Then using $[u, v]_g=B(u, (1_{V}-g)|_{V_{g}}^{-1}v)$, we get $M=NT^{-1}$. 
	Then $\mathrm{det}M=\mathrm{det}N\mathrm{det}T(\mathrm{det}(T^{-1}))^2$. 
	Therefore $dV_{g}=d(V_{g},B)\mathrm{det}((1_{V}-g)|_{V_{g}})(\mathrm{det}(T^{-1}))^2$. 
	Hence $\Theta_{W}(g)=(\mathrm{det}(1_{V}-g)|_{V_{g}})(d(V_{g} ,B))k^{\times2}$.
\end{proof}
\begin{corollary}\label{spinorunipotent}
	Let $g\in O(n,k)$ be unipotent. Then $\Theta_{W}(g)=k^{\times2}$.
\end{corollary}
\begin{proof}
	The fixed space of $-g$ is $0$, since $-1$ is not an eigenvalue of $g$. 
	Hence, its residual space $V_{-g}=V$, which is non-degenerate. Hence $-g$ is regular. 
	As $g$ is unipotent, there is a basis for $V$ such that the matrix of $g$ is upper triangular with 
	diagonal entries equal to $1$. 
	Therefore by Proposition~\ref{hahn} we have,
	$\Theta_{W}(-g)=(2^{n}dV)k^{\times2}$. 
	Again $-1_{V}$ has no fixed points (except 0), hence $-1_{V}$ is regular with residual space $V$. 
	Therefore $\Theta_{W}(-1_{V})=(2^{n})(dV)k^{\times2}$. 
	Hence $\Theta_{W}(g)=\Theta_W(-g)\Theta_W(-1_{V})=(2^{n}dV)(2^{n}dV)k^{\times2}=k^{\times2}$.
\end{proof}
%%This will be used later in Chapter~\ref{chapter7} for computing the spinor norm.
\section{Springer-Steinberg Form}
Let us fix some notation and terminology. Let $k$ be a perfect field of ${\rm char}\, k \neq 2$ with an involution 
$\sigma$ such that the fixed field of $\sigma$ is $k_0$. 
Let $V$ be a vector space over $k$, equipped with a non-degenerate hermitian form $B$.
Let $T\in U(V,B)$ with minimal polynomial $f(x)$. We define a $k$-algebra $E^T:=\frac{k[x]}{<f(x)>}$. 
Clearly, $V$ is an $E^T$-module, denoted by $V^T$.
The $E^T$-module structure on $V^T$ determines $GL(n)$-conjugacy class of $T$ 
(see~\cite{As} and~\cite{SS} for more details). 
To determine conjugacy classes of $T$ within $U(V,B)$, 
Springer and Steinberg defined a hermitian form $H^T$ on $V^T$, 
called \emph{Springer-Steinberg form}, denoted by $(H^T, V^T)$
(see 2.6 in~\cite{SS} Chapter IV). Since $f(x)$ is self-U-reciprocal (see~\ref{self} for definition), 
there exists a unique involution 
$\alpha$ on $E^T$ such that $\alpha(x)=x^{-1}$ and $\alpha$ is an extension of $\sigma$ on scalars. 
Thus $(E^T, \alpha)$ is an algebra with involution. They prove that there exists a $k$-linear function 
$l^T\colon E^T \rightarrow k$ such that the symmetric bilinear form $\bar{l^T}\colon E^T \times E^T \rightarrow k$ 
given by $\bar{l^T}(a,b):=l^T(ab)$ is non-degenerate with $l^T(\alpha(e))=l^T(e)$ for all $e\in E^T$. 
It follows that there exists a hermitian form $H^T$ on $E^T$-module $V^T$ (with respect to $\alpha$) 
satisfies $B(eu, v) = l^T(eH^T(u,v))$ for all $e\in E^T$, and $u,v \in V^T$ (see p. 254, 2.5~\cite{SS}). 
Let $S,T \in U(V,B)$, then the following commutative diagrams 
clarify what we are talking about so far, which will also be useful 
in the following proposition (Proposition~\ref{conjugacyunitary}). 
\begin{center}
	\begin{tikzpicture}[scale=0.8]
	\node (A1) at (-1,3) {$V^S\times V^S$};
	\node (B1) at (2,3) {$E^S$};
	\node (B2) at (2,1) {$k$};
	\node (P1) at (4,3) {$V^T\times V^T$};
	\node (Q1) at (7,3) {$E^T$};
	\node (Q2) at (7,1) {$k$};
	\node (a1) at (9,3) {$E^S$};
	\node (b1) at (12,3) {$E^T$};
	\node (b2) at (12,1) {$k$};
	%\path[->,font=\scriptsize,>=angle 90]
	\draw[->](B1)to node [right]{$l^S$} (B2);
	\draw[->](A1) to node [below]{$B$} (B2);
	\draw[->](A1) to node [above]{$H^S$} (B1);
	\draw[->](Q1)to node [right]{$l^T$} (Q2);
	\draw[->](P1) to node [below]{$B$} (Q2);
	\draw[->](P1) to node [above]{$H^T$} (Q1);
	\draw[->](b1)to node [right]{$l^T$} (b2);
	\draw[->](a1) to node [below]{$l^S$} (b2);
	\draw[->](a1) to node [above]{$f$} (b1);
	\end{tikzpicture}
\end{center}
%\begin{center}
%\begin{tikzpicture}[scale=1.2]
%\node (A1) at (4,3) {$V^T\times V^T$};
%\node (B1) at (7,3) {$E^T$};
%\node (B2) at (7,1) {$k$};
%\path[->,font=\scriptsize,>=angle 90]
%\draw[->](B1)to node [right]{$l^T$} (B2);
%\draw[->](A1) to node [below]{$B$} (B2);
%\draw[->](A1) to node [above]{$H^T$} (B1);
%\end{tikzpicture}
%\end{center}
%\begin{center}
%\begin{tikzpicture}[scale=0.8]
%\node (A1) at (0,3) {$E^S$};
%\node (B1) at (3,3) {$E^T$};
%\node (B2) at (3,1) {$k$};
%\path[->,font=\scriptsize,>=angle 90]
%\draw[->](B1)to node [right]{$l^T$} (B2);
%\draw[->](A1) to node [below]{$l^S$} (B2);
%\draw[->](A1) to node [above]{$f$} (B1);
%\end{tikzpicture}
%\end{center}

\begin{definition}
	Let $(V_1,H_1)$ and $(V_2,H_2)$ be two hermitian spaces over 
	$E_1$ and $E_2$ respectively, where $E_1$ and $E_2$ are isomorphic 
	modules over $k$ and let $f\colon E_1 \rightarrow E_2$ be an isomorphism. 
	Then we say $(V_1,H_1)$ and $(V_2,H_2)$ are \emph{equivalent}, denoted as 
	$(V_1,H_1) \approx (V_2,H_2)$, if there exists a $k$-isomorphism 
	$\varphi \colon V_1 \rightarrow V_2$ such that 
	\begin{enumerate}
		\item $\varphi(ev)=f(e)\varphi(v)$ and 
		\item $H_2(\varphi(u), \varphi(v))=f(H_1(u,v))$
	\end{enumerate}
	for all $u,v\in V_1$ and all $e\in E_1$.
\end{definition}
\noindent
We need the following (see p. 255, 2.7 and 2.8~\cite{SS} Chapter IV): 
\begin{proposition}\label{conjugacyunitary}
	With the notation as above, let $S$ and $T \in U(V,B)$. Then,
	\begin{enumerate}
		\item the elements $S$ and $T$ are conjugate in $U(V,B)$ if and only if 
		$(V^S,H^S)$ and $(V^T,H^T)$ are equivalent.
		\item The centralizer of $T$ in $U(V,B)$ is $\mathcal{Z}_{U(V,B)}(T)=U(V^T, H^T)$.
	\end{enumerate}
\end{proposition}
\begin{proof}
	\begin{enumerate}
		\item Suppose $S$ and $T$ are conjugate in $U(V,B)$. Then there exists a $\varphi \in U(V,B)$ such that 
		$T=\varphi S \varphi^{-1}$. 
		Then $\varphi : V^S \rightarrow V^T$ is a $k$-isomorphism. Here also $f:E^S \rightarrow E^T$ is a 
		$k$-isomorphism such that $f(S)=T$. Now for $v\in V^S$, $\varphi(S^mv)=\varphi\circ S^m(v)=T^m\circ \varphi(v)=
		f(S^m)\varphi(v)$. It then follows that $\varphi(ev)=f(e)\varphi(v)$ for all $e \in E^S$ and for all $v\in V^S$.
		Let $u,v \in V^S$, then $l^S(f^{-1}(H^T(\varphi u, \varphi v)))=l^T(H^T(\varphi u, \varphi v))=
		B(\varphi u, \varphi v)=B(u,v)=l^S(H^S(u,v))$. 
		Therefore $H^T(\varphi u, \varphi v)=f(H^S(u,v))$. Hence $(V^S, H^S) \approx (V^T, H^T)$.
		
		Conversely, suppose that $(V^S, H^S)$ and $(V^T, H^T)$ are equivalent. 
		Then there exists a $k$-isomorphism $\varphi : V^S \rightarrow V^T$  
		such that $H^T(\varphi u, \varphi v)=f(H^S(u,v))$ for all $u,v \in V^S$, where $f:E^S \rightarrow 
		E^T$ is a $k$-isomorphism such that $f(S)=T$ and 
		$\varphi (Sv)=f(S)\varphi(v)$ for all $v \in V^S$ and $S\in E^S$.
		For $v\in V^S$, $\varphi S(v)=\varphi(Sv)=f(S)\varphi(v)=T\varphi(v)=(T\varphi)(v)$, then $\varphi S=T \varphi$, i.e.,
		$\varphi S \varphi^{-1}=T$. Now, look at $B(\varphi u, \varphi v)=l^T(H^T(\varphi u, \varphi v))=l^T(f(H^S(u,v)))=
		l^S(H^S(u,v))=B(u,v)$ for all $u,v \in V$, then $\varphi $ is an isometry. 
		Hence $S$ and $T$ are conjugate in $U(V,B)$.
		\item Enough to show an isometry $\varphi$ is in $\mathcal{Z}_{U(V,B)}(T)$ if and only if $\varphi$ preserves $H^T$.
		Let $\varphi\in U(V,B)$ such that $\varphi T=T \varphi $. Then we get $H^T(\varphi u, \varphi v)=H^T(u,v)$ for all 
		$u,v \in V^T$ (here we replace $S$ by $T$ and $f$ by identity in part (1)).
		Conversely, suppose $\varphi$ preserves $H^T$, then $B(\varphi u, \varphi v)=l^T(H^T(\varphi u,\varphi v))=
		l^T(H^T(u,v))=B(u,v)$. So $\varphi$ is an isometry. Also as $\varphi T(v)=T\varphi(v)$ for all $v\in V$, then 
		$\varphi T=T\varphi$. Therefore $\varphi \in \mathcal{Z}_{U(V,B)}(T)$. 
		Hence $\mathcal{Z}_{U(V,B)}(T)=U(V^T,H^T)$.
	\end{enumerate}
\end{proof}

We can decompose $E^T=E_{1}\oplus E_{2} \oplus \cdots \oplus E_{r}$, 
where $E_{i}$ are indecomposable subalgebras of $E^T$ with respect to $\alpha$, i.e.,
$E_i$ are not direct sums of non-trivial $\alpha$-stable subalgebras (see section 2.2 Chapter IV of~\cite{SS}). 
The restriction of $\alpha$ to $E_i$ is an involution on $E_i$ denoted by $\alpha_{i}$. 
Clearly, $E_{i}$'s, are of one of the following forms 
according to the decomposition of $f(x)$ (see Equation (\ref{polydecomposition})): 
\begin{itemize}
	\item $\frac{k[x]}{<p(x)^d>}$, where $p(x)$ is an irreducible self-U-reciprocal polynomial.
	\item $\frac{k[x]}{<q(x)^d>} \oplus \frac{k[x]}{<\tilde{q}(x)^d>}$, where $q(x)$ is irreducible and not self-U-reciprocal.
\end{itemize}
In the second case, the two components $\frac{k[x]}{<q(x)^d>}$ and $\frac{k[x]}{<\tilde{q}(x)^d>}$ are isomorphic 
local rings (For $d=1$, let $\lambda$ be a root of $q(x)$ then $\bar{\lambda}^{-1}$ is a root 
of $\tilde{q}(x)$. Therefore $\frac{k[x]}{<q(x)>}\cong k(\lambda)\cong k(\bar{\lambda}^{-1})\cong 
\frac{k[x]}{<\tilde{q}(x)>}$. Let $\varphi$ be an isomorphism sending $x\mapsto \varphi(x)$. Now for $d>1$, 
define $\psi : \frac{k[x]}{<q(x)^d>}\rightarrow \frac{k[x]}{<\tilde{q}(x)^d>}$ via 
$\psi(f(x)+<q(x)^d>):=f(\varphi(x))+<\tilde{q}(x)^d>$, which is a ring homomorphism. Similarly, 
we can define a map other way. Hence the isomorphism). 
The restriction of $\alpha$ is given by $\alpha(a,b)=(b,a)$ via the isomorphism. 
Using Wall's approximation theorem (Corollary~\ref{Wall's}) it's easy to see that all hermitian forms over 
such rings are equivalent. Thus to determine equivalence of $H^T$, 
we need to look at modules over rings of the first kind.
\subsection{Wall's approximation theorem}
We recall a theorem of Wall (see Theorem 2.2.1~\cite{Wa2}), 
which will be useful for further analysis. Also, see Asai (Proposition 2.5~\cite{As})
for more details.
Let $R$ be a commutative ring with $1$, 
$\mathcal{J}$ be its Jacobson radical, and $\alpha$ be an involution on $R$. 
Let $(V,B)$ be a non-degenerate hermitian space of rank $n$ over $R$. 
We define $\underline {V}:= \frac{V}{\mathcal{J}V}$ a module over $\underline {R}:=\frac{R}{\mathcal{J}}$. 
Now $B$ induces a hermitian form $\underline {B}$ on $\underline{V}$ with respect to the involution 
$\underline{\alpha}$ of $\underline{R}$ induced by $\alpha$. Then we have 
(Theorem 2.2.1~\cite{Wa2}): 
\begin{theorem}[Wall's approximation theorem]\label{wall'sapproximationthm}
	With the notation as above,
	\begin{enumerate}
		\item any non-degenerate hermitian form over $\underline{R}$ 
		is induced by some non-degenerate hermitian form over $R$.
		\item Let $(V_{1}, B_{1})$ and $ (V_{2}, B_{2})$ be non-degenerate hermitian spaces over $R$, and 
		correspondingly, $(\underline{V_{1}}, \underline{B_{1}})$ and $(\underline{V_{2}}, \underline{B_{2}})$ be 
		non-degenerate hermitian spaces over $\underline{R}$. 
		Then $(V_{1}, B_{1})$ is equivalent to $(V_{2}, B_{2})$ if and only if $(\underline{V_{1}}, \underline{B_{1}})$ 
		is equivalent to $(\underline{V_{2}}, \underline{B_{2}})$.
	\end{enumerate}
\end{theorem}
\noindent 
For our purpose, we need the following (see also p. 256~\cite{SS}),
\begin{corollary}\label{Wall's}
	Let $V$ be a module over $R=\frac{k[x]}{<q(x)^d>} \oplus \frac{k[x]}{<\tilde{q}(x)^d>}$, and  
	$H_{1}$ and $H_{2}$ be two non-degenerate hermitian forms on $V$ with respect to the involution on $R$ 
	given by $\overline{(b,a)} = (a,b)$. Then $H_{1}$ and $H_{2}$ are equivalent.
\end{corollary}
\begin{proof}
	We use Wall's approximation theorem (Theorem~\ref{wall'sapproximationthm}). 
	Here the Jacobson radical of $R$ is 
	$\mathcal{J}=\frac{<q(x)>}{<q(x)^d>} \oplus \frac{<\tilde q(x)>}{<\tilde q(x)^d>}$. 
	Then $\underline {R}\cong \frac{k[x]}{<q(x)>} \oplus \frac{k[x]}{<\tilde q(x)>}\cong K \oplus K$,  
	where $K\cong \frac{k[x]}{<q(x)>} \cong \frac{k[x]}{<\tilde q(x)>}$ is a finite extension 
	of $k$ (thus separable). Now we have hermitian forms 
	$\underline{H_{i}}\colon  \underline{V}\times \underline{V}\rightarrow \underline{R}$ defined by 
	$\underline{H_{i}}(u+\mathcal{J}V, v+\mathcal{J}V)= H_{i}(u,v)\mathcal{J}$ for all $u,v \in V$. 
	Thus it is enough to show that 
	$\underline {H_{1}} $ is equivalent to $\underline{H_{2}}$ on $K \oplus K$-module $\underline{V}$. 
	The norm map $N \colon (K \oplus K)^{\times} \rightarrow K^{\times}$ is 
	$N(a,b)=\overline{(a,b)}(a,b)=(b,a)(a,b)=(ab,ab)$ (for definition of the norm map see p. 3~\cite{Kn}). 
	Clearly, this norm map is surjective. 
	Thus $\frac{K^{\times}}{Im(N)}$ is trivial. 
	Hence the hermitian form is unique up to equivalence in this case (see p. 87, Theorem 10.2~\cite{Gv}).
\end{proof}
%%%%%%%%%%%%%%%%%%%%%%%%%%%%%%%%%%%%%%%%%%%%%%%%%%%%%%%%%%%%%%%%%%
\chapter{Conjugacy Classes and $z$-classes}\label{chapter5}
The results in this chapter are part of~\cite{BS}.
To study the $z$-classes, it is important to understand the conjugacy classes 
because $z$-classes are union of conjugacy classes. 
The problem of classifying conjugacy classes in classical groups 
has been studied by many mathematicians, and there is a known 
substantial amount of results. See, for example, Asai, Macdonald, Milnor, Springer-Steinberg, Wall, 
Williamson~\cite{As,Ma,Mi,SS,Wa2,Wi}.
When the field is finite, Wall~\cite{Wa2} gave an explicit description of all the conjugacy classes in the 
unitary, symplectic and orthogonal groups, and also the order of centralizers.
For some recent accounts in this direction, especially with the applications in mind, see 
Thiem-Vinroot~\cite{TV}, and Burness and Giudici~\cite{BG} etc. 
The conjugacy classes in $GL(n,k)$ are given by the canonical form theory, and with 
the unitary group being its subgroup, one needs to begin there. 
We begin with recalling the notation involved in the description of conjugacy classes and $z$-classes. 
In Section 5.1 we define certain kinds of polynomials, which will be used 
in Section 5.2 to decompose the space with respect to a unitary linear transformation.
This decomposition may be thought of as a reduction step, which will be used in 
Chapter~\ref{chapter8} to prove one of the main theorems of this thesis. 
In Section 5.3 we describe $z$-classes in orthogonal and symplectic groups (for more details see~\cite{GK}).
\section{Self-$U$-reciprocal Polynomials}\label{self}
Let $k$ be a field with an involution given by $\bar{a}=a$ for all $a \in k$.
Let $f(x)=\displaystyle \sum_{i=0}^n a_{i}x^i \in k[x]$. 
We extend the involution on $k$ to that of $k[x]$ by $\overline{f(x)}:=\displaystyle \sum_{i=0}^n \bar{a_{i}} x^i$. 
Let $f(x)$ be a polynomial with $f(0)\neq 0$. The corresponding $U$-reciprocal polynomial of $f(x)$ is defined by 
$$\tilde{f}(x):=\overline{f(0)^{-1}}\ x^n \  \overline{f(x^{-1})}.$$
A monic polynomial $f(x)$ with a non-zero constant term is said to be \emph{self-$U$-reciprocal} if 
$f(x)=\tilde{f}(x)$.
In terms of roots, it means that for a self-U-reciprocal polynomial, 
whenever $\lambda$ is a root,  $\bar{\lambda}^{-1}$ is also a root with the same multiplicity. 
Note that $f(x)=\tilde{\tilde{f}}(x)$, and if $f(x)=f_{1}(x)f_{2}(x)$ 
then $\tilde{f}(x)=\tilde{f_{1}}(x)\tilde{f_{2}}(x)$ provided $f(x)$ is a monic polynomial. 
Also, $f(x)$ is irreducible if and only if $\tilde{f}(x)$ is 
irreducible. In the case of $f(x)=(x-\lambda)^n$, the polynomial $f(x)$ is self-U-reciprocal if and only if 
$\lambda \bar{\lambda}=1$. 
A slightly more general polynomial, called self-dual polynomial will be defined in Section 5.3. 
%which has been used to study $z$-classes in symplectic and orthogonal groups.
Over a finite field, we have the following result due to Ennola (Lemma 2~\cite{En}):
\begin{proposition}\label{ennolaodd}
	Let $f(x)$ be a monic, irreducible, self-U-reciprocal polynomial over a finite field $\mathbb F_{q^2}$. 
	Then the degree of $f(x)$ is odd.    
\end{proposition}
\begin{proof}
	Here the involution of $\mathbb{F}_{q^2}$ is given by $\bar{a}=a^q$. 
	Let $\mathrm{deg}\, f(x)=n$. Let $\alpha$ be a root of $f(x)$ in its splitting field $L$ over $\mathbb{F}_{q^2}$.
	Let $\sigma$ be the Frobenius automorphism of $L$ given by $\sigma(a)=a^{q^2}$. Then $[L:\mathbb{F}_{q^2}]=\mathrm{deg}\,
	f(x)=n=\mathrm{order}\,(\sigma)$. Since $f(x)$ is self-U-reciprocal, so if $\alpha$ is a root of $f(x)$, then 
	$\alpha ^{-q}$ is also a root of $f(x)$ with the same multiplicity. Therefore there is an automorphism $\tau$ of $L$ 
	over $\mathbb{F}_{q^2}$ such that $\tau (\alpha)=\alpha^{-q}$. Then $\tau^2(\alpha)=\alpha^{q^2}=\sigma(\alpha)$, so 
	$\tau^2=\sigma$ since $L=\mathbb{F}_{q^2}(\alpha)$. 
	Now $\tau \in \langle \sigma \rangle \cong \mathbb{Z}/{n\mathbb{Z}} \cong \mathrm{Gal}(L/\mathbb{F}_{q^2})$, 
	so $\tau=\sigma^t$ for some $t$.
	Therefore $\sigma=\tau^2=\sigma^{2t}$, so $\sigma^{2t-1}=1$. 
	Hence $n$ is a divisor of $2t-1$, so $n$ is odd.
\end{proof}
\begin{lemma}
	Let $T\in GL(n, k)$, and suppose $f(x)$ is the minimal polynomial of $T$. Then the minimal polynomial
	of $\bar{T}^{-1}$ is $\tilde{f}(x)$.
\end{lemma}
\begin{proof}
	Since $\tilde{f}(x)=\overline{f(0)}^{-1}x^d \overline{f(x^{-1})}$, and $f(T)=0$, then 
	$\tilde{f}(\bar{T}^{-1})=\overline{f(0)}^{-1}(\bar{T}^{-1})^d \overline{f((\bar{T}^{-1})^{-1})}=
	\overline{f(0)}^{-1}\bar{T}^{-d}\overline{f(\bar{T})}=\overline{f(0)}^{-1}\bar{T}^{-d}\overline{f(T)} =0$.
	Thus we conclude that $\tilde{f}(x)$ is the minimal polynomial of $\bar{T}^{-1}$.
\end{proof}
\begin{remark}
	If $g \in U(V,B)$, then $\tr \bar{g} \beta g=\beta$. So $\beta g \beta^{-1}=\tr \bar{g}^{-1}$, which is conjugate 
	to $\bar{g}^{-1}$, as $g$ is conjugate to its transpose in $GL(n,k)$. 
	Hence the minimal polynomials of $g$ and $\bar{g}^{-1}$
	are same, i.e., $f(x)=\tilde{f}(x)$.
\end{remark}
If $T \in U(V,B)$ then its minimal polynomial $f(x)$ is monic with a non zero constant term, and is self-U-reciprocal. 
We can write it as follows:
\begin{equation}\label{polydecomposition}
f(x)=\prod_{i=1}^{k_{1}}p_{i}(x)^{r_{i}}\prod_{j=1}^{k_{2}}(q_{j}(x)\tilde{q_{j}}(x))^{s_{j}}, 
\end{equation}
where $p_{i}(x)$ and $q_{j}(x)$ are irreducible, and $p_{i}(x)$ is self-$U$-reciprocal but 
$q_{j}(x)$ is not self-U-reciprocal for all $i,j$.
\section{Space Decomposition with Respect to a Unitary Transformation}
Let $T \in U(V, B)$, and $f(x)\in k[x]$ satisfying $f(0)\neq 0$. Then,
\begin{lemma}\label{bar}
	For any $u,v \in V$, we have 
	$B(u, f(T)v)=B(\overline{f(T^{-1})}u, v)$.
\end{lemma}
\begin{proof}
	Let $f(x)=\displaystyle \sum_{i=0}^n a_{i}x^i$, then $f(T)=\displaystyle \sum_{i=0}^n a_{i}T^i$ .
	Observe that $B(T^{-1}u,v)=B(u, Tv)$ for all $u,v \in V$. Now 
	$B(u,\displaystyle \sum_{i=0}^n a_{i}T^iv)=\displaystyle \sum_{i=0}^n a_{i} B(u, T^i v)
	=\displaystyle \sum_{i=0}^n a_{i}B(T^{-i}u,v)=
	B(\displaystyle \sum_{i=0}^n\bar{a_{i}}T^{-i}u, v)$ for all $u,v \in V$.
	Hence $B(u, f(T)v)=B(\overline{f(T^{-1})}u, v)$.
\end{proof}

\begin{lemma}\label{orthogonalsum}
	The subspaces $\mathrm{Im}(f(T))$ and 
	$\mathrm{ker}(\tilde{f}(T))$ are mutually orthogonal.
\end{lemma}
\begin{proof}
	Let $u\in \mathrm{ker}(\tilde{f}(T))$ and $v\in \mathrm{Im}(f(T))$, therefore $\tilde{f}(T)u=0$ and $v=f(T)w$ 
	for some $w \in V$. Now 
	$B(u,v)=B(u, f(T)w)=B(\overline{f(T^{-1})}u,w)=B(T^d\overline{f(T^{-1})}u,T^dw)=
	f(0) B(\overline{f(0)}^{-1}T^d\overline{f(T^{-1})}u,T^dw)=f(0)B(\tilde{f}(T)u,T^dw)=f(0)B(0, T^dw)=0$.
	Hence $\mathrm{Im}(f(T))\perp \mathrm{ker}(\tilde{f}(T))$.
\end{proof}
Let $T \in U(V, B)$ with minimal polynomial $f(x)$. Write $f(x)=\prod_{i}f_{i}(x)^{m_{i}}$ as in  
Equation (\ref{polydecomposition}), where $f_i(x)=p_i(x)$ or $q_i(x)\tilde{q_i}(x)$. 
Then,
\begin{proposition}\label{spacedecomposition}
	The direct sum decomposition $V=\bigoplus \limits_{i}\mathrm{ker}(f_{i}(T)^{m_{i}})$ is a decomposition into 
	non-degenerate mutually orthogonal $T$-invariant subspaces.
\end{proposition}
\begin{proof}
	Let $v\in \mathrm{ker}(f_i(T)^{m_i})\cap \mathrm{ker}(f_j(T)^{m_j})$ for some $i\neq j$, then 
	$f_i(T)^{m_i}(v)=0=f_j(T)^{m_j}(v)$. Since $f_i(x)^{m_i}$ are pairwise relatively prime, then 
	$a_i(x)f_i(x)^{m_i}+a_j(x)f_j(x)^{m_j}=1$ for some $a_i(x), a_j(x)\in k[x]$. 
	So $a_i(T)f_i(T)^{m_i}(v)+a_j(T)f_j(T)^{m_j}(v)=v$, therefore $v=0$. Hence 
	the sum $V=\bigoplus \limits_{i}\mathrm{ker}(f_{i}(T)^{m_{i}})$ is a direct sum. 
	Clearly, these subspaces are $T$-invariant. Observe that $\mathrm{Im}(f_{j}(T)^{m_{j}})=
	\bigoplus \limits_{i\neq j}\mathrm{ker}(f_{i}(T)^{m_{i}})$, and 
	$\mathrm{ker}(f_{i}(T)^{m_{i}})=\mathrm{ker}(\tilde{f_{i}}(T)^{m_{i}})$, since $f_{i}(x)=\tilde{f_{i}}(x)$ 
	for all $i$. By Lemma~\ref{orthogonalsum}, 
	we have $\mathrm{Im}(f_{j}(T)^{m_{j}})\perp \mathrm{ker}(\tilde{f_{j}}(T)^{m_{j}})$. So we get 
	$\bigoplus \limits_{i\neq j}\mathrm{ker}(f_{i}(T)^{m_{i}}) \perp \mathrm{ker}(f_{j}(T)^{m_{j}})$ for all $j$.
	Hence in the sum $V=\bigoplus \limits_{i}\mathrm{ker}(f_{i}(T)^{m_{i}})$, the subspaces are mutually orthogonal. 
	Also mutual orthogonality implies that the restriction of the form on each subspaces are non-degenerate.
\end{proof}
\noindent 
This decomposition helps us reduce the questions about conjugacy classes and 
$z$-classes of a unitary transformation to the unitary transformations with minimal polynomial of 
one of the following two kinds:
\begin{description}\label{type}
	\item[\textbf{Type 1.}] $p(x)^{m}$, where $p(x)$ is monic irreducible 
	self-U-reciprocal polynomial with a non-zero constant term, 
	\item[\textbf{Type 2.}] $(q(x)\tilde{q}(x))^{m}$, where $q(x)$ is monic, irreducible, 
	not self-U-reciprocal with a non-zero constant term.
\end{description}
Thus Proposition~\ref{spacedecomposition} gives us a \emph{primary decomposition} of 
$V$ into $T$-invariant $B$ non-degenerate subspaces
\begin{equation}\label{primary}
V=\left(\bigoplus_{i=1}^{k_{1}}V_{i}\right)\bigoplus \left( \bigoplus_{j=1}^{k_2}V_{j} \right), 
\end{equation}
where $V_{i}=\mathrm{ker}(p_i(T)^{r_i})$ corresponds 
to the polynomials of Type 1, and $V_{j}=W_j + W_{j}^{*}$ 
corresponds to the polynomials of Type 2, where $W_j=\mathrm{ker}(q_{j}(T)^{s_j})$ and 
$W_{j}^{*}=\mathrm{ker}(\tilde{q_{j}}(T)^{s_j})$.
Denote the restriction of $T$ to each $V_r$ by $T_{r}$. 
Then the minimal polynomial of $T_{r}$ is one of the two types. It turns out that the centralizer of $T$ in $U(V,B)$ is 
$$\mathcal Z_{U(V,B)}(T)=\prod_{r} \mathcal Z_{U(V_r, B_r)}(T_{r}),$$
where $B_r$ is a hermitian form obtained by restricting $B$ to $V_r$. 
Thus the conjugacy class and the $z$-class of $T$ is determined by the restriction of $T$ to 
each of the primary subspaces. Hence it is enough to determine the conjugacy class and the $z$-class of 
$T\in U(V,B)$, which has the minimal polynomial of one of the types in~\ref{type}. 
\section{$z$-classes in Orthogonal and Symplectic Groups}
Let $V$ be an $n$-dimensional vector space over $k$ with the property FE, 
equipped with a non-degenerate symmetric or skew-symmetric bilinear form $B$.
The $z$-classes of orthogonal groups $O(V,B)$ and symplectic groups $Sp(V,B)$ have 
been discussed by Gongopadhyay and Ravi S. Kulkarni in~\cite{GK} (see Theorem~\ref{gongopadhyay}). 
We will be very brief in this section to parametrize the $z$-classes in orthogonal and symplectic groups. 
Let $f(x)=\displaystyle \sum_{i=0}^n a_{i}x^i$ be a polynomial in $k[x]$ of degree $n$ such 
that $0,1$ and $-1$ are not its roots. 
The corresponding dual polynomial of $f(x)$ is defined by 
$$f^{*}(x):=f(0)^{-1} x^n f(x^{-1}).$$
A monic polynomial $f(x)$ with $0,1,-1$ are not its roots is said to be \emph{self-dual} if 
$f(x)=f^{*}(x)$.
In terms of roots, it means that for a self-dual polynomial, 
whenever $\lambda$ is a root,  $\lambda^{-1}$ is also a root with the same multiplicity. 
Suppose $T\in O(V,B)$ or $Sp(V,B)$ with the minimal polynomial $m_{T}(x)$. 
Thus an irreducible factor says $p(x)$, of the minimal polynomial, can be one of the following three types:
\begin{itemize}
	\item $x+1$ or $x-1$.
	\item $p(x)$ is self-dual. 
	\item $p(x)$ is not self-dual. In this case, there is an irreducible factor $p^{*}(x)$ will occur 
	in the minimal polynomial.
\end{itemize}
If $T \in O(V,B)$ or $Sp(V,B)$ then its minimal polynomial $m_{T}(x)$ 
is monic with a non-zero constant term, and is self-dual. 
We can write it as follows
\begin{equation}
m_{T}(x)=(x+1)^e (x-1)^f \prod_{i=1}^{k_{1}}p_{i}(x)^{r_{i}}\prod_{j=1}^{k_{2}}(q_{j}(x)q_{j}^{*}(x))^{s_{j}}, 
\end{equation}
where $p_{i}(x)$ and $q_{j}(x)$ are irreducible, and $p_{i}(x)$ is self-dual but 
$q_{j}(x)$ is not self-dual for all $i,j$.
Thus Proposition~\ref{spacedecomposition} gives us a \emph{primary decomposition} of 
$V$ into $T$-invariant $B$ non-degenerate subspaces
\begin{equation}\label{primary}
V=\left(V_{1}\bigoplus V_{-1}\right)\bigoplus \left(\bigoplus_{i=1}^{k_{1}}V_{i}\right)
\bigoplus \left( \bigoplus_{j=1}^{k_2}V_{j} \right), 
\end{equation}
where  $V_{-1}=\mathrm{ker}(T+I)^e, V_{1}=\mathrm{ker}(T-I)^f$, and $V_{i}=\mathrm{ker}(p_i(T)^{r_i})$ corresponds 
to the self-dual polynomials, and $V_{j}=W_j + W_{j}^{*}$ 
corresponds to the not self-dual polynomials, where $W_j=\mathrm{ker}(q_{j}(T)^{s_j})$ and 
$W_{j}^{*}=\mathrm{ker}(q_{j}^{*}(T)^{s_j})$.
Denote the restriction of $T$ to each $V_r$ by $T_{r}$ so $T=\bigoplus_{r}T_r$. 
Then the minimal polynomial of $T_{r}$ is one of the three types. 
It turns out that the centralizer of $T$ in $O(V,B)$ or $Sp(V,B)$ is 
$$\mathcal Z (T)=\prod_{r} \mathcal Z (T_{r}).$$
%where $B_r$ is a symmetric or skew-symmetric form obtained by restricting $B$ to $V_r$. 
Thus the $z$-class of $T$ is determined by the restriction of $T$ to 
each of the primary subspaces. Then it has been proved that there are only finitely many $z$-classes 
of semisimple and unipotent elements in orthogonal and symplectic groups respectively. Thus using 
Jordan decomposition (Theorem~\ref{jordandecomposition}), there are only finitely many $z$-classes 
in orthogonal groups $O(V,B)$ and symplectic groups $Sp(V,B)$. For more details see p. 257 in~\cite{GK}.
%%%%%%%%%%%%%%%%%%%%%%%%%%%%%%%%%%%%%%%%%%%%%%%%%%%%%%%%%%%%%%%%%%%
\chapter{Gaussian Elimination}\label{chapter6}
The results in this chapter are part of~\cite{BMS}. We improved the results on the 
symplectic and split orthogonal similitude groups. 
This chapter is one of the main chapters of this thesis. For 
instance, in Chapter~\ref{chapter7} we use Gaussian elimination to compute 
the spinor norm as well as similitude characters.
In dealing with constructive group recognition project, one needs to solve the word 
problem in some generating set. Thus, the main objective of this chapter 
is to develop a similar algorithm for symplectic and 
split orthogonal similitude groups to solve the word problem. 
In Section 6.1 we describe the classical Gaussian elimination algorithm for 
general linear groups. In Section 6.2 we define elementary operations for similitude 
groups, and describe the Gaussian elimination in similitude groups in Section 6.3. In 
Section 6.4 we record a result~\cite{MS} on the Gaussian elimination in the split 
unitary groups.
\section{Gaussian Elimination in General Linear Groups}
As we know, in the general linear group 
$GL(n,k)$, the word problem has an efficient solution in elementary 
transvections (or elementary matrices) - the Gaussian elimination. 
One observes that the effect of multiplying by elementary transvections 
on a matrix from left or right is either a row or column operation 
respectively. We have the following classical Gaussian elimination algorithm for $GL(n,k)$:
\begin{theorem}\label{classicalgauss}
	Every element $g\in M(n,k)$ can be written as a product 
	of elementary matrices and a 
	diagonal matrix, the diagonal matrix is of the form
	$\mathrm{diag}(1,\ldots,1,\mathrm{det}(g))$ if $\mathrm{det}(g)\neq 0$; else
	$\mathrm{diag}(1,\ldots,1,0,\ldots,0)$.
\end{theorem}
Using the above Theorem~\ref{classicalgauss} one can solve the word problem in $SL(n,k)$, which 
can be stated as follows: 
\begin{corollary}
	Every element of the special linear group $SL(n,k)$ can be written 
	as a product of elementary transvections (or elementary matrices).
\end{corollary}
Let $B$ be a subgroup of upper triangular matrices and $W$ be the 
subgroup of permutation matrices in $GL(n,k)$ respectively. In this case 
$W\cong S_n$, symmetric group on $n$ letters. Then we have the following (see p.108~\cite{Ca1}): 
\begin{theorem}[Bruhat decomposition]\label{bruhatdecomposition}
	With the notation as above,  
	$$GL(n,k)=BWB=\bigsqcup_{w\in W} BwB.$$
\end{theorem}
So the above Theorem~\ref{bruhatdecomposition} says that 
any element $g\in GL(n,k)$ can be written as $g=b_1wb_2$ 
for some $b_1, b_2 \in B$, and $w\in W$ (which is unique). 
Therefore $w=b_1^{-1}gb_2^{-1}$. Thus, any invertible matrix can be transformed 
into a permutation matrix by a series of row and column operations. 
\section{Elementary Operations}\label{elementaryoperations}
Elementary operations can be thought of as usual row-column 
operations for matrices. 
We already described the elementary matrices in Section~\ref{elementarymatrices} for the 
symplectic and split orthogonal similitude groups. Then multiplications of those 
elementary matrices on the left and right to an element of the similitude groups, for 
example, symplectic and split orthogonal similitude groups, are elementary operations,  
which we are going to describe below case by case. 
The Gaussian elimination algorithm is slightly different for matrices of even and odd size. 
We first describe it for matrices of even size and then for matrices of the odd size.
\subsection{Elementary operations for $GSp(2l,k)\; (l\geq 2)$}
Let $g=\begin{pmatrix}A&B\\C&D\end{pmatrix}$ be a $2l\times 2l$ 
matrix written in block form of size $l\times l$. 
Then the row and column operations are as follows:
\begin{eqnarray*}
	ER1:&\begin{pmatrix}R&0\\0 & \tr {R}^{-1}\end{pmatrix}\begin{pmatrix}A&B\\C&D
	\end{pmatrix}&=\begin{pmatrix}RA&RB \\ \tr {R}^{-1}C & \tr {R}^{-1} D\end{pmatrix}\\
	EC1:&\begin{pmatrix}A&B\\C&D\end{pmatrix}\begin{pmatrix}R&0\\0 & \tr {R}^{-1}
	\end{pmatrix}&=\begin{pmatrix} AR&B\tr {R}^{-1} \\ CR & D\tr {R}^{-1}\end{pmatrix}
\end{eqnarray*}
\begin{eqnarray*}
	ER2: & \begin{pmatrix}I&R\\ 0&I\end{pmatrix}\begin{pmatrix}A&B\\C&D\end{pmatrix}&=
	\begin{pmatrix}A+RC&B+RD \\ C & D\end{pmatrix} \\
	EC2: &\begin{pmatrix}A&B\\C&D\end{pmatrix}\begin{pmatrix}I&R\\0 &I \end{pmatrix}&=
	\begin{pmatrix} A&AR+B \\ C & CR+ D\end{pmatrix}
\end{eqnarray*}
\begin{eqnarray*}
	ER3: &\begin{pmatrix}I&0\\R &I\end{pmatrix}\begin{pmatrix}A&B\\C&D\end{pmatrix}&=
	\begin{pmatrix}A&B \\ RA+ C & RB+D\end{pmatrix}\\ 
	EC3: &\begin{pmatrix}A&B\\C&D\end{pmatrix}\begin{pmatrix}I&0\\ R&I \end{pmatrix}&=
	\begin{pmatrix} A+BR&B \\ C+DR & D\end{pmatrix}.
\end{eqnarray*}
\subsection{Elementary operations for $GO(2l,k)\; (l\geq 2)$}
Let $g=\begin{pmatrix}
A & B \\ C & D
\end{pmatrix}
$ be a $2l \times 2l$ matrix written in block form of size $l\times l$. 
Then the row and column operations are as follows:
\begin{eqnarray*}
	ER1:&\begin{pmatrix}R&0\\ 0& \tr {R}^{-1}\end{pmatrix}\begin{pmatrix}
		A&B\\C&D\end{pmatrix}&=\begin{pmatrix}RA&RB \\ \tr {R}^{-1}C & \tr {R}^{-1} D\end{pmatrix}\\
	EC1:&\begin{pmatrix}A&B\\C&D\end{pmatrix}\begin{pmatrix}R&0\\ 0& \tr {R}^{-1}
	\end{pmatrix}&=\begin{pmatrix} AR&B\tr {R}^{-1} \\ CR & D\tr {R}^{-1}\end{pmatrix}
\end{eqnarray*}
\begin{eqnarray*}
	ER2: & \begin{pmatrix}I&R\\ 0&I\end{pmatrix}\begin{pmatrix}A&B\\C&D\end{pmatrix}
	&=\begin{pmatrix}A+RC&B+RD \\ C & D\end{pmatrix} \\
	EC2: &\begin{pmatrix}A&B\\C&D\end{pmatrix}\begin{pmatrix}I&R\\ 0&I \end{pmatrix}
	&=\begin{pmatrix} A&AR+B \\ C & CR+ D\end{pmatrix}
\end{eqnarray*}
\begin{eqnarray*}
	ER3: &\begin{pmatrix}I&0\\R &I\end{pmatrix}\begin{pmatrix}A&B\\C&D\end{pmatrix}
	&=\begin{pmatrix}A&B \\ RA+ C & RB+D\end{pmatrix}\\ 
	EC3: &\begin{pmatrix}A&B\\C&D\end{pmatrix}\begin{pmatrix}I&0\\ R&I \end{pmatrix}
	&=\begin{pmatrix} A+BR&B \\ C+DR & D\end{pmatrix}.
\end{eqnarray*}
\subsection{Elementary operations for $GO(2l+1,k)\; (l\geq 2)$}
Let $g=\begin{pmatrix}\alpha&X&Y\\ E& A&B\\F&C&D\end{pmatrix}$ be a $(2l+1)\times (2l+1)$ 
matrix, where $A,B,C,D$ are $l\times l$ matrices, and $X=(X_1 X_2 \cdots X_l)$ and  
$Y=(Y_1 Y_2 \cdots Y_l)$ are $1\times l$ matrices, 
and $E=\tra(E_1 E_2 \cdots E_l)$ and $F=\tra(F_1 F_2 \cdots F_l)$ are 
$l\times 1$ matrices. 
Let $\alpha\in k$. Then the row and column operations are as follows:
\begin{eqnarray*}
	ER1:&\begin{pmatrix}1&0&0\\ 0&R&0\\ 0&0& \tr {R}^{-1}\end{pmatrix}\begin{pmatrix}
		\alpha & X&Y\\E& A&B\\F& C&D\end{pmatrix}&=\begin{pmatrix}\alpha &X&Y\\ RE&RA&RB 
		\\ \tr {R}^{-1}F&\tr {R}^{-1}C & \tr {R}^{-1} D\end{pmatrix}\\
	EC1:&\begin{pmatrix}\alpha & X&Y\\ E&A&B\\F&C&D\end{pmatrix}\begin{pmatrix}1&0&0 \\ 
		0&R&0\\ 0&0& \tr {R}^{-1}\end{pmatrix}&=\begin{pmatrix} \alpha& XR & Y\tr {R}^{-1}\\ 
		E&AR&B\tr {R}^{-1} \\ F& CR & D\tr {R}^{-1}\end{pmatrix}.
\end{eqnarray*}
\begin{eqnarray*}
	ER2: & \begin{pmatrix}1&0&0\\ 0&I&R\\ 0&0&I\end{pmatrix}\begin{pmatrix}\alpha &X&Y\\ 
		E&A&B\\F&C&D\end{pmatrix}&=\begin{pmatrix}\alpha &X&Y\\ E+RF& A+RC&B+RD \\ F& C & D\end{pmatrix} \\
	EC2:&\begin{pmatrix}\alpha& X&Y\\ E& A&B\\F& C&D\end{pmatrix}\begin{pmatrix}1&0&0\\ 
		0&I&R\\ 0&0&I \end{pmatrix}&=\begin{pmatrix} \alpha& X& XR+Y\\ E&A&AR+B \\ F& C & CR+ D\end{pmatrix}.
\end{eqnarray*}
\begin{eqnarray*}
	ER3: &\begin{pmatrix}1&0&0\\ 0&I&0\\0&R &I\end{pmatrix}\begin{pmatrix}\alpha & X&Y\\ E&A&B\\F&C&D
	\end{pmatrix}&=\begin{pmatrix}\alpha&X&Y\\ E&A&B \\ RE+F&RA+ C & RB+D\end{pmatrix}\\ 
	EC3:&\begin{pmatrix}\alpha& X&Y\\ E&A&B\\F&C&D\end{pmatrix}\begin{pmatrix}1&0&0\\ 0&I&0\\ 0&R&I 
	\end{pmatrix}&=\begin{pmatrix}\alpha& X+YR&Y\\E& A+BR&B \\ F&C+DR & D\end{pmatrix}.
\end{eqnarray*}
For $E4$ we only write the equations which we need later. 
%For $1\leq i\leq l$, we get the following:
\begin{itemize}
	\item Let the matrix $g$ have $C=\mathrm{diag}(d_1,\ldots,d_l)$.
	$$ER4:\;[(I+t(2e_{i,0}-e_{0,-i})-t^2e_{i,-i})g]_{0,i}= X_i-td_i $$
	%$$EC4:\;[g(I+te_{0,-i}-2te_{i,0}-t^2e_{i,-i})]_{-i,0}= F_i-2td_i.$$
	\item Let the matrix $g$ have $A=\mathrm{diag}(d_1,\ldots,d_l)$.
	$$ER4:\;[(I+t(-2e_{-i,0}+e_{0,i})-t^2e_{-i,i})g]_{0,i}= X_i+td_i $$
	$$EC4:\;[g(I+t(2e_{i,0}-e_{0,-i})-t^2e_{i,-i})]_{i,0}= E_i+2 td_i, $$ 
	where $1\leq i \leq l.$
\end{itemize}
\section[Gaussian Elimination in Similitude Groups]{Gaussian Elimination in Symplectic and Orthogonal Similitude Groups}
\subsection{Some useful lemmas}
To justify the steps of the Gaussian elimination algorithm we need several lemmas.
So this subsection is devoted to prove these lemmas.
\begin{lemma}\label{lemma1}
	Let $Y=\mathrm{diag}(1,\ldots,1,\lambda,\ldots,\lambda)$ be of size $l$
	with the number of $1$s equal to $m<l$. Let $X$ be a matrix of size $l$ such that 
	$YX$ is symmetric (resp. skew-symmetric) then $X$ is of the form 
	$ \begin{pmatrix}X_{11}&X_{12}\\X_{21}&X_{22}\end{pmatrix}$, 
	where $X_{11}$ is an $m\times m$ symmetric (resp. skew-symmetric), and 
	$X_{12}=\lambda \tr X_{21}$ (resp. $X_{12}=-\lambda \tr X_{21}$).
	Furthermore, if $\lambda \neq 0$ then $X_{22}$ is symmetric (resp. skew-symmetric).
\end{lemma}
\begin{proof}
	First, observe that the matrix $YX=\begin{pmatrix}X_{11}&X_{12}\\\lambda X_{21}&\lambda X_{22}\end{pmatrix}$.
	Since the matrix $YX$ is symmetric (resp. skew-symmetric), then $X_{11}$ is symmetric (resp. skew-symmetric),
	and $X_{12}=\lambda \tr X_{21}$ (resp. $X_{12}=-\lambda \tr X_{21}$).
	Also if $\lambda \neq 0$ then $X_{22}$ is symmetric (resp. skew-symmetric).
\end{proof}
\begin{corollary}\label{lemma2}
	Let $g=\begin{pmatrix} A&B\\C&D \end{pmatrix}$ be either in $GSp(2l,k)$ or $GO(2l,k)$.
	\begin{enumerate}
		\item If $A$ is a diagonal matrix $\mathrm{diag}(1,\ldots,1,\lambda), \lambda \in k^{\times}$,  
		then the matrix $C$ is of the form $\begin{pmatrix}C_{11}&\pm \lambda \tr C_{21} 
		\\C_{21} & c_{ll}\end{pmatrix}$, 
		where $C_{11}$ is an $(l-1)\times(l-1)$ symmetric if $g \in GSp(2l,k)$,  
		and $C_{11}$ is skew-symmetric with $c_{ll}=0$ if $g \in GO(2l,k)$.
		\item If $A$ is a diagonal matrix $\mathrm{diag}(\underbrace{1,\ldots,1}_{m},\underbrace{0,\ldots,0}_{l-m})$, 
		then the matrix $C$ is of the form 
		$\begin{pmatrix} C_{11} &0\\C_{21}&C_{22}\end{pmatrix}$, where $C_{11}$ is an $m\times m$
		symmetric matrix if $g\in GSp(2l,k)$, and is skew-symmetric if $g\in GO(2l,k)$.
	\end{enumerate}
\end{corollary}
\begin{proof}
	We use the condition that $g$ satisfies $\tr g \beta g=\mu(g)\beta$, and 
	$AC$ is symmetric (using $\tr A=A$, as $A$ is diagonal),  
	when $g\in GSp(2l,k)$, and $AC$ is skew-symmetric, when 
	$g\in GO(2l,k)$. Then Lemma~\ref{lemma1} gives the required 
	form for $C$.
\end{proof}
\begin{corollary}\label{lemma3}
	Let $g=\begin{pmatrix}A&B\\0&\mu(g)A^{-1}\end{pmatrix} \in GSp(2l,k)$ or $ GO(2l,k)$, where 
	$A=\mathrm{diag}(1,\ldots,1,\lambda)$, then the matrix $B$ is of the form 
	$\begin{pmatrix} B_{11}&\pm \lambda^{-1}\tr B_{21}\\B_{21}&b_{ll}\end{pmatrix}$, 
	where $B_{11}$ is a symmetric matrix of size $l-1$ if $g\in GSp(2l,k)$, and 
	skew-symmetric with $b_{ll}=0$ if $g\in GO(2l,k)$.
\end{corollary}
\begin{proof}
	We use the condition that $g$ satisfies $\tr g \beta g=\mu(g)\beta$ and 
	$\tr A=A$ to get $A^{-1}B$ is symmetric if $g\in GSp(2l,k)$, and skew-symmetric 
	if $g \in GO(2l,k)$. Again Lemma~\ref{lemma1} gives the required form 
	for $B$.
\end{proof}
\begin{lemma}\label{lemma4}
	Let $g=\begin{pmatrix} A&B\\0&D\end{pmatrix}\in GL(2l,k)$. Then,
	\begin{enumerate}
		\item  $g \in GSp(2l,k)$ if and only if 
		$D=\mu(g)\tr A^{-1}$ and $\tr (A^{-1}B)=(A^{-1}B)$, and 
		\item  $g \in GO(2l,k)$ if and only if 
		$D=\mu(g)\tr A^{-1}$ and $\tr (A^{-1}B)=-(A^{-1}B)$.
	\end{enumerate}
\end{lemma}
\begin{proof}
	\begin{enumerate}
		\item Let $g\in GSp(2l,k)$ then $g$ satisfies $\tr g \beta g=\mu(g)\beta$.
		Then this implies $D=\mu(g)\tr A^{-1}$ and $\tr (A^{-1}B)=(A^{-1}B)$. 
		
		Conversely, if $g$ satisfies the given condition then clearly $g\in GSp(2l,k)$.
		\item This follows by similar computation.
	\end{enumerate}
\end{proof}
\begin{lemma}\label{lemma5}
	Let $Y=\mathrm{diag}(1,\ldots,1,\lambda)$ be of size 
	$l$, where $\lambda \in k^{\times}$ and $X=(x_{ij})$ be 
	a matrix such that $YX$ is symmetric (resp. skew-symmetric).
	Then $X=(R_{1}+R_{2}+\ldots)Y$, where each $R_{m}$ is of the 
	form $t(e_{i,j}+e_{j,i})$ for some $i<j$ or of the form $te_{i,i}$
	for some $i$ (resp. each $R_{m}$ is of the form 
	$t(e_{i,j}-e_{j,i})$ for some $i<j$).
\end{lemma}
\begin{proof}
	Since the matrix $YX$ is symmetric (resp. skew-symmetric), 
	then the matrix $X$ is of the form 
	$\begin{pmatrix}X_{11}&X_{12}\\X_{21}&x_{ll}\end{pmatrix}$, 
	where $X_{11}$ is symmetric (resp. skew-symmetric), 
	$X_{12}=\lambda \tr X_{21}$ (resp. $X_{12}=-\lambda \tr X_{21}$) 
	and $X_{21}$ is a row of size $l-1$. 
	Clearly, $X$ is a sum of the matrices of the form 
	$R_m Y$.    
\end{proof}
\begin{lemma}\label{lemma6}
	For $1\leq i\leq l$, 
	\begin{enumerate}
		\item The element $w_{i,-i}=I+e_{i,-i}-e_{-i,i}-e_{i,i}-e_{-i,-i} \in GSp(2l,k)$ 
		is a product of elementary matrices.
		\item The element $w_{i,-i}=I-e_{i,-i}-e_{-i,i}-e_{i,i}-e_{-i,-i} \in GO(2l,k)$ 
		is a product of elementary matrices.
		\item The element $w_{i,-i}=I-2e_{0,0}-e_{i,-i}-e_{-i,i}-e_{i,i}-e_{-i,-i} \in GO(2l+1,k)$ 
		is a product of elementary matrices.     
	\end{enumerate}
\end{lemma}
\begin{proof}
	\begin{enumerate}
		\item We have $w_{i,-i}=x_{i,-i}(1)x_{-i,i}(-1)x_{i,-i}(1)$.
		\item We produce these elements inductively. 
		First we get $w_{i,-j}=(I+e_{i,-j}-e_{j,-i})(I+e_{-i,j}-e_{-j,i})(I+e_{i,-j}-e_{j,-i})
		=x_{i,-j}(1)x_{-i,j}(1)x_{i,-j}(1)$, and 
		$w_{i,j}=(I+e_{i,j}-e_{-j,-i})(I-e_{j,i}+e_{-i,-j})(I+e_{i,j}-e_{-j,-i})
		=x_{i,j}(1)x_{j,i}(-1)x_{i,j}(1)$. 
		Set $w_{l}:=w_{l,-l}=I-e_{l,l}-e_{-l,-l}-e_{l,-l}-e_{-l,l}$. 
		Then compute $w_{l}w_{l,l-1}w_{l,-(l-1)}=w_{(l-1),-(l-1)}$.
		So inductively we get $w_{i,-i}$ is a product of elementary matrices.
		\item We have $w_{i,-i}=x_{0,i}(-1)x_{i,0}(1)x_{0,i}(-1)$.
	\end{enumerate}
\end{proof}
\begin{lemma}\label{lemma7}
	The element $\mathrm{diag}(1,\ldots,1,\lambda,1,\ldots,1,\lambda^{-1})\in GSp(2l,k)$ 
	is a product of elementary matrices.
\end{lemma}
\begin{proof}
	First we compute 
	\begin{align*}
	w_{l,-l}(t)&=(I+te_{l,-l})(I-t^{-1}e_{-l,l})(I+e_{l,-l}) \\
	&=I-e_{l,l}-e_{-l,-l}+te_{l,-l}-t^{-1}e_{-l,l} \\
	&=x_{l,-l}(t)x_{-l,l}(-t^{-1})x_{l,-l}(t).
	\end{align*}                             
	Then compute 
	\begin{align*}
	h_{l}(\lambda)&=w_{l,-l}(\lambda)w_{l,-l}(-1) \\
	&=I-e_{l,l}-e_{-l,-l}+\lambda e_{l,l}+\lambda^{-1}e_{-l,-l}, 
	\end{align*}                      
	which is the required element.
\end{proof}
\begin{lemma}\label{lemma8}
	Let $g=\begin{pmatrix} \alpha &X&Y\\E&A&B\\F&C&D\end{pmatrix} \in GO(2l+1,k)$. Then, 
	\begin{enumerate}
		\item If $A=\mathrm{diag}(1,\ldots,1,\lambda)$ and $X=0$, then 
		$C$ is of the form $\begin{pmatrix}C_{11}&-\lambda \tr C_{21}\\C_{21}&0 \end{pmatrix}$ 
		with $C_{11}$ skew-symmetric.
		\item If $A=\mathrm{diag}(\underbrace{1,\ldots,1}_{m},\underbrace{0,\ldots,0}_{l-m})$, 
		and $X$ with its first $m$ entries $0$, then 
		$C$ is of the form $\begin{pmatrix}C_{11}&0\\C_{21}&C_{22} \end{pmatrix}$ 
		with $C_{11}$ skew-symmetric.
	\end{enumerate}
\end{lemma}
\begin{proof}
	We use the equation $\tr g\beta g=\mu(g)\beta$, and get 
	$2\tr XX+\tr AC+\tr CA=0$. In the first case, $AC$ is skew-symmetric 
	(using $X=0$ and $\tr A=A$). Then Lemma~\ref{lemma1} and Corollary~\ref{lemma2} 
	give the required form for $C$. In the second case, we note that $\tr XX$ has 
	top-left and top-right blocks $0$, and get the required form for $C$.
\end{proof}
\begin{lemma}\label{lemma9}
	Let $g=\begin{pmatrix}\alpha &X&Y\\E&A&B\\F&0&D \end{pmatrix} \in GO(2l+1,k)$, then
	$X=0$, and $D=\mu(g)\tr A^{-1}$.
\end{lemma}
\begin{proof}
	We compute $\tr g\beta g=\mu(g)\beta$, and get $2\tr XX=0$ and 
	$2\tr XY+\tr AD=\mu(g)I$. Hence $X=0$, and $D=\mu(g)\tr A^{-1}$.
\end{proof}
\begin{lemma}\label{lemma10}
	Let $g=\begin{pmatrix}\alpha &0&Y\\0&A&B\\F&0&D \end{pmatrix}$, 
	with $A$ an invertible diagonal matrix. Then 
	$g\in GO(2l+1,k)$ if and only if $\alpha^2=\mu(g), F=0=Y, D=\mu(g)A^{-1}$ and 
	$\tr DB+\tr BD=0$, where $\mu(g) \in k^{\times}$ is similitude of $g$.
\end{lemma}
\begin{proof}
	Let $g\in GO(2l+1,k)$ then we have $\tr g\beta g=\mu(g)\beta$. So we get 
	$\alpha^2=\mu(g), F=0=Y, D=\mu(g)A^{-1}$ and $\tr DB+\tr BD=0$.
	
	Conversely, if $g$ satisfies the given condition, then $g \in GO(2l+1,k)$.
\end{proof}
\subsection{Gaussian elimination for $GSp(2l,k)$ and $GO(2l,k)$}\label{gausseven}
The algorithm is as follows:\\
Step $1$: 
\begin{enumerate}[]
	\item \textbf{Input}: A matrix $g=\begin{pmatrix}A&B\\C&D \end{pmatrix} \in GSp(2l,k)$ or $GO(2l,k)$.\\
	\item \textbf{Output}: The matrix $g_1=\begin{pmatrix}A_1&B_1\\C_1&D_1 \end{pmatrix}$ 
	is one of the following kind:\\
	\begin{enumerate}[a:]
		\item The matrix $A_1$ is a diagonal matrix 
		$\mathrm{diag}(1,\ldots,1,\lambda)$ with $\lambda \neq 0$, 
		and $C_1=\begin{pmatrix}C_{11}&C_{12}\\C_{21}&c_{ll}\end{pmatrix}$,  
		where $C_{11}$ is symmetric, when $g\in GSp(2l,k)$, and 
		skew-symmetric, when $g\in GO(2l,k)$, and is of size $l-1$. 
		Furthermore, $C_{12}=\lambda \tr C_{21}$, when $g \in GSp(2l,k)$, and 
		$C_{12}=-\lambda \tr C_{21}, c_{ll}=0$, when $g \in GO(2l,k)$.
		\item The matrix $A_1$ is a diagonal matrix $\mathrm{diag}
		(\underbrace{1,\ldots,1}_{m},\underbrace{0,\ldots,0}_{l-m})$, and 
		$C_1=\begin{pmatrix} C_{11}&0\\C_{21}&C_{22}\end{pmatrix}$, where 
		$C_{11}$ is an $m\times m$ symmetric, when $g\in GSp(2l,k)$ and 
		skew-symmetric, when $g\in GO(2l,k)$.
	\end{enumerate}
	\item \textbf{Justification}: Observe the effect of ER$1$ and EC$1$ on the block $A$. 
	This amounts to the classical Gaussian elimination 
	(see Theorem~\ref{classicalgauss}) on a $l\times l$ matrix $A$. 
	Thus we can reduce $A$ to a diagonal matrix, and Corollary~\ref{lemma2} makes sure 
	that $C$ has the required form. 
\end{enumerate}     
Step $2$:       
\begin{enumerate}[]
	\item \textbf{Input}: matrix $g_1=\begin{pmatrix}A_1&B_1\\C_1&D_1\end{pmatrix}$.
	\item \textbf{Output}: matrix $g_2=\begin{pmatrix}A_2&B_2\\0&\mu(g) \tr A_{2}^{-1}\end{pmatrix}
	;A_2=\mathrm{diag}(1,\ldots,1,\lambda)$.
	\item \textbf{Justification}: Observe the effect of ER$3$. It changes $C_1$ by 
	$RA_1+C_1$. Using Lemma~\ref{lemma5} we can make the matrix $C_1$ the zero matrix 
	in the first case, and $C_{11}$ the zero matrix in the second case. Further, in the second 
	case, we make use of Lemma~\ref{lemma6} to interchange the rows, so that we get a zero matrix 
	in place of $C_1$. If required, use ER$1$ and EC$1$ to make $A_1$ a diagonal matrix.
	Lemma~\ref{lemma4} ensures that $D_1$ becomes $\mu(g) \tr A_2^{-1}$.
\end{enumerate}
Step $3$: 
\begin{enumerate}[]
	\item \textbf{Input}: matrix $g_2=\begin{pmatrix}A_2&B_2\\0&\mu(g)\tr A_2^{-1}\end{pmatrix}
	;A_2=\mathrm{diag}(1,\ldots,1,\lambda)$.
	\item \textbf{Output}: matrix $g_3=\mathrm{diag}(1,\ldots,1,\lambda,\mu(g),\ldots,\mu(g),\mu(g) \lambda^{-1})$.
	\item \textbf{Justification}: Using Corollary~\ref{lemma3} we see that the matrix 
	$B_2$ has a certain form. We can use ER$2$ to make the matrix $B_2$ a zero matrix 
	because of Lemma~\ref{lemma5}.
\end{enumerate}
The algorithm terminates here for $GO(2l,k)$. However for $GSp(2l,k)$ there is one more step.\\
Step $4$:
\begin{enumerate}[]
	\item \textbf{Input}: matrix $g_3=\mathrm{diag}(1,\ldots,1,\lambda,\mu(g),\ldots,\mu(g),\mu(g) \lambda^{-1})$.
	\item \textbf{Output}: matrix $g_4=\mathrm{diag}(1,\ldots,1,\mu(g),\ldots,\mu(g))$, where $\mu(g)\in k^{\times}$.
	\item \textbf{Justification}: Using Lemma~\ref{lemma7}.
\end{enumerate}
\subsection{Gaussian elimination for $GO(2l+1,k)$}\label{gaussodd}
The algorithm is as follows:\\
Step $1$:
\begin{enumerate}[]
	\item \textbf{Input}: A matrix $g=\begin{pmatrix}\alpha &X&Y\\E&A&B\\F&C&D\end{pmatrix} \in GO(2l+1,k)$.
	\item \textbf{Output}: The matrix $g_1=\begin{pmatrix}\alpha_1 &X_1&Y_1\\E_1&A_1&B_1\\F_1&C_1&D_1\end{pmatrix}$ 
	is one of the following kind:
	\begin{enumerate}[a:]
		\item The matrix $A_1$ is a diagonal matrix $\mathrm{diag}(1,\ldots,1,\lambda)$ with $\lambda \neq 0$.
		\item The matrix $A_1$ is a diagonal matrix $\mathrm{diag}(\underbrace{1,\ldots,1}_{m},\underbrace{0,\ldots,0}_{l-m}) (m<l)$.
	\end{enumerate}
	\item \textbf{Justification}: Using ER$1$ and EC$1$ we do the 
	classical Gaussian elimination (see Theorem~\ref{classicalgauss}) on a 
	$l\times l$ matrix $A$.
\end{enumerate}
Step $2$:
\begin{enumerate}[]
	\item \textbf{Input}: matrix $g_1=\begin{pmatrix}\alpha_1 &X_1&Y_1\\E_1&A_1&B_1\\F_1&C_1&D_1\end{pmatrix}$.
	\item \textbf{Output}: matrix $g_2=\begin{pmatrix}\alpha_2 &X_2&Y_2\\E_2&A_2&B_2\\F_2&C_2&D_2\end{pmatrix}$ 
	is one of the following kind:
	\begin{enumerate}[a:]
		\item The matrix $A_2$ is $\mathrm{diag}(1,\ldots,1,\lambda)$ with $\lambda \neq 0, X_2=0=E_2$, and 
		$C_2=\begin{pmatrix}C_{11}&-\lambda \tr C_{21}\\C_{21}&0\end{pmatrix}$, where $C_{11}$ is skew-symmetric of size 
		$l-1$.
		\item The matrix $A_2$ is $\mathrm{diag}(\underbrace{1,\ldots,1}_{m},\underbrace{0,\ldots,0}_{l-m}) (m<l)$; 
		$X_2, E_2$ have first $m$ entries $0$, and $C_2=\begin{pmatrix}C_{11}&0\\C_{21}&C_{22}\end{pmatrix}$, 
		where $C_{11}$ is an $m\times m$ skew-symmetric matrix.
	\end{enumerate}
	
	\item \textbf{Justification}: Once we have $A_1$ in diagonal form, we use ER$4$ and EC$4$ to change 
	$X_1$ and $E_1$ to the required form. Then Lemma~\ref{lemma8} makes sure that $C_1$ has the required form.
\end{enumerate}
Step $3$:
\begin{enumerate}[]
	\item \textbf{Input}: matrix $g_2=\begin{pmatrix}\alpha_2 &X_2&Y_2\\E_2&A_2&B_2\\F_2&C_2&D_2\end{pmatrix}$.
	\item \textbf{Output}: 
	\begin{enumerate}[a:]
		\item matrix $g_3=\begin{pmatrix}\alpha_3 &0&Y_3\\0&A_3&B_3\\F_3&0&D_3\end{pmatrix};A_3=
		\mathrm{diag}(1,\ldots,1,\lambda)$.
		\item matrix $g_3=\begin{pmatrix}\alpha_3 &X_3&Y_3\\E_3&A_3&B_3\\F_3&C_3&D_3\end{pmatrix};
		A_3=\mathrm{diag}(\underbrace{1,\ldots,1}_{m},\underbrace{0,\ldots,0}_{l-m})$; $X_3, E_3$ have first 
		$m$ entries $0$, and $C_3=\begin{pmatrix}0&0\\C_{21}&C_{22}\end{pmatrix}$.
	\end{enumerate}
	\item \textbf{Justification}: Observe the effect of ER$3$, and Lemma~\ref{lemma5} ensures the 
	required form.
\end{enumerate}
Step $4$:
\begin{enumerate}[]
	\item \textbf{Input}: matrix $g_3=\begin{pmatrix}\alpha_3 &X_3&Y_3\\E_3&A_3&B_3\\F_3&C_3&D_3\end{pmatrix}$
	\item \textbf{Output}: matrix $g_4=\begin{pmatrix}\alpha_4&0&0\\0&A_4&B_4\\0&0&\mu(g)A_4^{-1}\end{pmatrix}$ with 
	$A_4=\mathrm{diag}(1,\ldots,1,\lambda), \alpha_4^2=\mu(g)$, and 
	$B_4A_4+A_4\tr B_4=0$.
	
	\item \textbf{Justification}: In the first case, Lemma~\ref{lemma10} ensures the required form. 
	In the second case, we interchange $i$ with $-i$ for $m+1\leq i \leq l$. This will make $C_3=0$.
	Then, if needed, we use ER$1$ and EC$1$ on $A_3$ to make it diagonal. Then Lemma~\ref{lemma9} 
	ensures that $A_3$ has full rank. Further, we can use ER$4$ and EC$4$ to make $X_3=0=E_3$. 
	Lemma~\ref{lemma10} gives the required form.
\end{enumerate}
Step $5$:
\begin{enumerate}[]
	\item \textbf{Input}: matrix $g_4=\begin{pmatrix}\alpha_4&0&0\\0&A_4&B_4\\0&0&\mu(g)A_4^{-1}\end{pmatrix};
	A_4=\mathrm{diag}(1,\ldots,1,\lambda), \alpha_4^2=\mu(g)$.
	\item \textbf{Output}: matrix $g_5=\mathrm{diag}(\alpha_5,1,\ldots,1,\lambda,\mu(g),\ldots,\mu(g),\mu(g)\lambda^{-1})$
	with $\alpha_5^2=\mu(g)$.
	\item \textbf{Justification}: Lemma~\ref{lemma10} ensures that $B_4$ is of a certain kind. We 
	can use ER$2$ to make $B_4=0$. 
\end{enumerate}
\noindent
Thus the main result of this chapter is the following theorem:
\begin{theorem}\label{maintheorem1}
	Every element of symplectic similitude group $GSp(2l,k)$ or split  
	orthogonal similitude group $GO(n,k)$ (here $n=2l$ or $2l+1$), 
	can be written as a product of elementary matrices and 
	a diagonal matrix. Furthermore, the diagonal matrix is of the following form:
	\begin{enumerate}
		\item In $GSp(2l,k)$, 
		$\mathrm{diag}(\underbrace{1,\ldots,1}_{l},\underbrace{\mu(g),\ldots, \mu(g)}_{l})$, where $\mu(g) \in k^{\times}$.
		\item In $GO(2l,k)$, 
		$\mathrm{diag}(\underbrace{1,\ldots,1,\lambda}_{l},\underbrace{\mu(g),\ldots,\mu(g), \mu(g)\lambda^{-1}}_{l})$, 
		where $\lambda, \mu(g) \in k^{\times}$.
		\item In $GO(2l+1,k)$, 
		$\mathrm{diag}(\alpha(g),\underbrace{1,\ldots,1,\lambda}_{l},\underbrace{\mu(g),\ldots, \mu(g), \mu(g)\lambda^{-1}}_{l})$, 
		where $\alpha(g)^2=\mu(g)$ and $\mu(g), \lambda \in k^{\times}$.
	\end{enumerate}
\end{theorem}
\begin{proof}
	This follows from the above algorithms~\ref{gausseven} and~\ref{gaussodd}.
\end{proof}
%Then there is a Gaussian elimination algorithm given in~\cite{BMS} which reduces, 
%by multiplying elementary matrices from left and right (that is row and column operations), an element of orthogonal 
%group to a diagonal element. 
\noindent
This gives us following:
\begin{corollary}\label{gaussianelimination}
	Every element $g\in O(n,k)$ (here $n=2l$ or $2l+1$) 
	can be written as a product of elementary matrices and a diagonal matrix. 
	Furthermore, the diagonal matrix is 
	$\mathrm{diag}(\underbrace{1, \ldots ,1, \lambda}_{l\; \text{or}\; l+1}, \underbrace{1, \ldots , 1, \lambda^{-1}}_{l}), 
	\lambda \in k^{\times}$.
\end{corollary}
\begin{proof}
	As $g\in O(n,k)$ so $\mu(g)=1$. In the odd dimensional 
	orthogonal group, $\alpha=\pm 1$. In this situation, if needed we use 
	Lemma~\ref{lemma6} to make the first diagonal entry $1$.
	Hence this follows from Theorem~\ref{maintheorem1}.
\end{proof}
\begin{corollary}\label{gaussiansymplectic}
	Every element of the symplectic group $Sp(n,k)$ can be written as a product 
	of elementary matrices. 
\end{corollary}
\begin{proof}
	This follows from Theorem~\ref{maintheorem1}, as $\mu(g)=1$.
\end{proof}
\begin{remark}
	Corollary~\ref{gaussianelimination} and Corollary~\ref{gaussiansymplectic} solve the word 
	problem in orthogonal groups $O(n,k)$ and symplectic groups $Sp(n,k)$.
\end{remark}
\section{Gaussian Elimination in Unitary Groups}
A similar algorithm has been developed in~\cite{MS}. One 
can define elementary matrices and elementary operations for split unitary groups, 
%(see (4) in Example~\ref{splitu}), 
similar to that of symplectic and split orthogonal groups. Using those elementary 
matrices and elementary operations, Mahalanobis and Singh 
solved the word problem in split unitary groups. They proved (Theorem A~\cite{MS}):
\begin{theorem}\label{gaussianunitary}
	Every element of the split unitary group $U(n,k_0)$ (here $n=2l$ or $2l+1$) can be 
	written as a product of elementary matrices and a diagonal matrix. Furthermore, 
	the diagonal matrix is of the following form: 
	\begin{enumerate}
		\item In $U(2l,k_0)$,  
		$\mathrm{diag}(\underbrace{1,\ldots,1,\lambda}_{l},\underbrace{1,\ldots,1,\bar{\lambda}^{-1}}_{l})$, 
		where $\lambda \in k^{\times}$.
		\item In $U(2l+1,k_0)$,  
		$\mathrm{diag}(\alpha,\underbrace{1,\ldots,1,\lambda}_{l},\underbrace{1,\ldots,1,\bar{\lambda}^{-1}}_{l})$, 
		where $\lambda, \alpha \in k^{\times}$ with $\alpha \bar{\alpha}=1$.
	\end{enumerate}
\end{theorem}
%%%%%%%%%%%%%%%%%%%%%%%%%%%%%%%%%%%%%%%%%%%%%%%%%%%%%%%%%%%%%%%%%%%
\chapter{Computing Spinor Norm and Similitude}\label{chapter7}
This chapter reports the work done in~\cite{BMS}.
In this chapter, we show how we can use Gaussian elimination developed in Chapter~\ref{chapter6} 
to compute the spinor norm for split orthogonal groups.
Also in this chapter, we compute similitude character for split groups 
using the Gaussian elimination algorithm. In this chapter, we make use of Wall's theory 
developed in Chapter~\ref{chapter4}.

\section{}
To compute the spinor norm, we will use the following lemma.
\begin{lemma}\label{techlemma}
	With the notation as earlier for the group $O(n,k)$ (here $n=2l$ or $2l+1$), we have,
	\begin{enumerate}
		\item $\Theta(x_{i,j}(t))=\Theta(x_{i,-j}(t))=\Theta(x_{-i,j}(t))=\Theta(x_{i,0}(t))=\Theta(x_{0,i}(t))=k^{\times2}$.
		\item $\Theta(w_{l})=k^{\times2}$.
		\item $\Theta(\mathrm{diag}(\underbrace{1,\ldots,1,\lambda}_{l \; \text{or} \; l+1},
		\underbrace{1,\ldots,1,{\lambda}^{-1}}_{l}))=\lambda k^{\times2}$.
	\end{enumerate}
\end{lemma}
\begin{proof}
	\begin{enumerate}
		\item This follows from Corollary~\ref{spinorunipotent}, since the given elements are all unipotent.
		\item Observe that $w_{l}$ is a reflection along $e_{l}+e_{-l}$, and $Q(e_{l}+e_{-l})=1$, hence
		$\Theta (w_l)=Q(e_l+e_{-l})k^{\times 2}=k^{\times 2}$.
		\item First observe that $\mathrm{diag}(\underbrace{1,\ldots,1,\lambda}_{l \; \text{or} \; l+1},
		\underbrace{1,\ldots,1,{\lambda}^{-1}}_{l})=
		\sigma_{e_{l}+e_{-l}} \sigma_{e_{l}+\lambda e_{-l}}$.
		Since $\{e_l,e_{-l}\}$ is a hyperbolic pair (see in Section 2.2) 
		then $Q(e_{l}+e_{-l})=1$, and $ Q(e_{l}+\lambda e_{-l})=\lambda $ . 
		Hence 
		$$\Theta(\mathrm{diag}(\underbrace{1,\ldots,1,\lambda}_{l \; \text{or} \; l+1},
		\underbrace{1,\ldots,1,{\lambda}^{-1}}_{l}))=Q(e_{l}+e_{-l})Q(e_{l}
		+\lambda e_{-l})k^{\times2}=\lambda k^{\times2}.$$
	\end{enumerate}
\end{proof}
\noindent 
The main result is the following:
\begin{theorem}[Spinor norm]\label{maintheorem2}
	Let $g\in O(n,k)$ (here $n=2l$ or $2l+1$). Suppose Gaussian elimination reduces $g$ to 
	$\mathrm{diag}(\underbrace{1,\ldots,1,\lambda}_{l \; \text{or} \; l+1},\underbrace{1,\ldots,1,\lambda^{-1}}_{l})$, 
	where $\lambda \in k^{\times}$. 
	Then the spinor norm $\Theta(g)=\lambda k^{\times 2}$.
\end{theorem}
\begin{proof}
	Let $g\in O(n,k)$. We write 
	$g$ as a product of elementary matrices and a diagonal matrix of the form 
	$\mathrm{diag}(\underbrace{1,\ldots,1,\lambda}_{l \; \text{or} \; l+1},\underbrace{1,\ldots,1,\lambda^{-1}}_{l})$, 
	following Corollary~\ref{gaussianelimination}.
	Again from Lemma~\ref{techlemma}, we get the spinor norm for the elementary matrices and the 
	diagonal matrix. Hence $\Theta(g)=\lambda k^{\times 2}$.
\end{proof}
\begin{remark}
	The Gaussian elimination algorithm also gives us how to compute the 
	similitude character $\mu(g)$ of the symplectic and split orthogonal 
	similitude groups (see Theorem~\ref{maintheorem1}).
\end{remark}
%%%%%%%%%%%%%%%%%%%%%%%%%%%%%%%%%%%%%%%%%%%%%%%%%%%%%%%%%%%%%%%%%%%
\chapter{Finiteness of $z$-classes}\label{chapter8}
The results in this chapter are part of~\cite{BS}.
This chapter is devoted to the study of $z$-classes in unitary groups. 
%Let $V$ be a finite dimensional vector space over a field $k$ with 
%the property FE, equipped with a non-degenerate hermitian form $B$.
A unitary group is an algebraic group defined over $k_0$. 
Since we are working with perfect fields, an element $T\in U(V,B)$ 
has a Jordan decomposition, $T=T_{s}T_{u}=T_uT_s$, where 
$T_{s}$ is semisimple and $T_{u}$ is unipotent (see Theorem~\ref{jordandecomposition}). 
Further one can use this to compute the centralizer $\mathcal Z_{U(V,B)}(T)=
\mathcal Z_{U(V,B)}(T_{s})\cap \mathcal Z_{U(V,B)}(T_{u})$. 
So the Jordan decomposition helps us reduce the study of conjugacy 
and computation of the centralizer of an element to the study 
of that of its semisimple and unipotent parts. 
In Section 8.1 we study the $z$-classes for unipotent elements. In Section 8.2 
we explore the $z$-classes for semisimple elements, and then we prove our main theorem,  
which states that the number of $z$-classes in any unitary group is finite if $k_0$ has the property FE. 
The preliminaries for this chapter have been discussed in Chapters~\ref{chapter2},~\ref{chapter4} and~\ref{chapter5}. 
\section{Unipotent $z$-classes}
We look at a special case when the minimal polynomial is $p(x)^d$, where $p(x)$ is an irreducible, 
self-$U$-reciprocal polynomial. This includes unipotent elements. The rational canonical form theory 
gives a decomposition of 
$$V= \displaystyle \bigoplus_{i=1}^r V_{d_{i}}$$
with $1\leq d_{1}\leq d_{2} \leq \ldots \leq d_{r}=d$, 
and each $V_{d_{i}}$ is a free module over the $k$-algebra $\frac{k[x]}{<p(x)^{d_{i}}>}$ 
(see 2.14 Chapter IV~\cite{SS}). Thus,
\begin{proposition}\label{conjugacycriterion}
	Let $S$ and $T$ be in $U(V,B)$. Suppose the minimal polynomial of both $S$ and $T$ are equal, and it equals $p(x)^d$, 
	where $p(x)$ is irreducible self-$U$-reciprocal. Then $S$ and $T$ are conjugate in $U(V,B)$ if and only if 
	\begin{enumerate}
		\item the elementary divisors $p(x)^{d_i}$ of $S$ and $T$ are 
		same for $1\leq d_{1}\leq d_{2} \leq \ldots \leq d_{r} =d$, and
		\item the sequence of hermitian spaces, 
		$\left\{(V_{d_{1}}^S, H_{d_{1}}^S), \ldots,(V_{d_{r}}^S,H_{d_{r}}^S)\right\}$ 
		corresponding to $S$, and  $\left\{(V_{d_{1}}^T, H_{d_{1}}^T),\ldots,(V_{d_{r}}^T,H_{d_{r}}^T)\right\}$ 
		corresponding to $T$ are equivalent. Here $H_{d_{i}}^S$ and $H_{d_{i}}^T$ take values in the 
		cyclic $k$-algebra $\frac{k[x]}{<p(x)^{d_{i}}>}$.
	\end{enumerate}
	Moreover, the centralizer of $T$, in this case, is the direct product 
	$\mathcal{Z}_{U(V,B)}(T)=\displaystyle \prod_{i=1}^r U(V_{d_{i}}^T, H_{d_{i}}^T)$.
\end{proposition}
\begin{proof}
	Suppose $S$ and $T$ are conjugate in $U(V,B)$.  
	Since they are conjugate they have the same set of elementary divisors which proves (1), and (2) follows 
	from Proposition~\ref{conjugacyunitary}.
	
	Conversely, the elementary divisors of $S$ and $T$ determine the orthogonal decomposition of $V$ as follows:
	\begin{align}
	V &= V_{d_{1}}^S \oplus \cdots \oplus V_{d_{r}}^S \\
	V &= V_{d_{1}}^T \oplus \cdots \oplus V_{d_{r}}^T, 
	\end{align}
	where $1\leq d_{1}\leq d_{2} \leq \ldots \leq d_{r}=d$, and for each $i$, 
	$V_{d_{i}}^S$ and $V_{d_{i}}^T$ are free as $E_{d_{i}}^S$ 
	and $E_{d_{i}}^T$-module respectively.
	Since $E_{d_{i}}^S$ and $E_{d_{i}}^T$ are isomorphic as $k$-modules. We may write $E_{d_{i}}:=
	E_{d_{i}}^S \cong E_{d_{i}}^T \cong \frac{k[x]}{<p(x)^{d_{i}}>}$. Also by (2) we have 
	$(V_{d_{i}}^S, H_{d_{i}}^S)\approx (V_{d_{i}}^T, H_{d_{i}}^T)$ for all $i=1,2,\ldots,r$.
	So by Proposition~\ref{conjugacyunitary}, we get 
	$S|_{V_{d_{i}}^S}$ is conjugate to $T|_{V_{d_{i}}^T}$ by $\varphi_{i}$, then 
	$\varphi=\varphi_{1} \oplus \cdots \oplus \varphi_{r}$ conjugates $S$ and $T$.
	
	Moreover, we have already seen that $\mathcal{Z}_{U(V,B)}(T)
	=\displaystyle \prod_{i=1}^r \mathcal{Z}_{U(V_i, B_i)}(T_{i})$.
	And by Proposition~\ref{conjugacyunitary}, we have 
	$\mathcal{Z}_{U(V_{d_i},B_i)}(T_{i})=U(V_{d_{i}}^T, H_{d_{i}}^T)$ for all $i$. 
	Hence $\mathcal{Z}_{U(V,B)}(T)=\displaystyle \prod_{i=1}^r U(V_{d_{i}}^T, H_{d_{i}}^T)$.
\end{proof}
\noindent
This gives us following:
\begin{corollary}\label{zunipotent}
	Let $k_0$ have the property FE. Then,
	\begin{enumerate}
		\item the number of conjugacy classes of unipotent elements in $U(V, B)$ is finite.
		\item The number of $z$-classes of unipotent elements in $U(V, B)$ is finite.
	\end{enumerate}
\end{corollary}
\begin{proof}
	\begin{enumerate}
		\item In view of Proposition~\ref{conjugacycriterion}, let the minimal polynomial be $(x-1)^d$.  
		Thus, we have $p(x)= x-1$. Then the conjugacy classes correspond to a sequence 
		$1\leq d_{1}\leq d_{2} \leq \ldots \leq d_{r} =d$, and hermitian spaces 
		$\{(V_{d_1}^T, H_{d_1}^T), \ldots, (V_{d_{r}}^T, H_{d_{r}}^T)\}$ up to equivalence. 
		Now $\underline {E}_{d_{i}}^T=\frac{k[T]}{<T-1>} \cong k$. 
		Then, by the Wall's approximation theorem (Theorem~\ref{wall'sapproximationthm}), 
		the number of non-equivalent hermitian 
		forms $(V, B)$ is exactly equal to the number of non-equivalent hermitian forms 
		$(\underline{V}, \underline{B})$. 
		Now $k_0$ has the property FE, so $k_{0}^{\times}/k_{0}^{\times 2}$ is finite (see Lemma~\ref{kfinite}). 
		Then there are only finitely many non-equivalent quadratic forms over $k_0$ (see p. 32, Corollary 4.3 in~\cite{Gv}).
		Hence, we know (see p. 267, Theorem~\cite{Ja}) that there are only finitely 
		many non-equivalent hermitian forms over $k$. 
		Thus $H_{d_i}^T$ has only finitely many choices for each $i$. Hence the result.
		\item Two elements are conjugate implies that they are also $z$-conjugate. Hence it follows 
		from the previous part.
	\end{enumerate}
\end{proof}
\section{Semisimple $z$-classes} Let $T\in U(V,B)$ be a semisimple element. 
First, we begin with a basic case.
\begin{lemma}\label{semisimplezclassl}
	Let $T\in U(V,B)$ be a semisimple element such that its minimal polynomial 
	is either $p(x)$, which is irreducible, self-$U$-reciprocal of degree $\geq 2$, or 
	$q(x)\tilde q(x)$, where $q(x)$ is irreducible not self-U-reciprocal. Let $E=\frac{k[x]}{<p(x)>}$ 
	in the first case and $\frac{k[x]}{<q(x)>}$ in the second case. 
	Then the $z$-class of $T$ is determined by the following:
	\begin{enumerate}
		\item the algebra $E$ over $k$, and
		\item the equivalence class of the $E$-valued hermitian form $H^T$ on $V^T$.
	\end{enumerate}
\end{lemma}
\begin{proof}
	Suppose $S, T \in U(V,B)$ are in the same $z$-class, then 
	$\mathcal Z_{U(V,B)}(S)=g\mathcal Z_{U(V,B)}(T)g^{-1}$ for some $g\in U(V,B)$. 
	We may replace $T$ by its conjugate $gTg^{-1}$, so we get 
	$\mathcal Z_{U(V,B)}(S)=\mathcal Z_{U(V,B)}(T)$, thus $U(V^S, H^S)=U(V^T, H^T)$. 
	Hence $(V^S, H^S)$ is equivalent to $(V^T,H^T)$. So, in particular, $E^S$ and $E^T$ 
	are isomorphic as $k$-algebras. The converse follows from Proposition~\ref{conjugacyunitary}.
\end{proof}

Now for the general case, let $T\in U(V,B)$ be a semisimple element with minimal polynomial
$$m_{T}(x)= \prod_{i=1}^{k_{1}}p_{i}(x) 
\prod_{j=1}^{k_{2}}\left(q_{j}(x)\tilde{q}_{j}(x)\right),$$ 
where the $p_{i}(x)$ are self-$U$-reciprocal polynomials of degree $d_i$, and 
$q_{j}(x)$ not self-$U$-reciprocal of degree $e_j$. 
Let the characteristic polynomial of $T$ be $$\chi_{T}(x)=\prod_{i=1}^{k_{1}}p_{i}(x)^{r_i} \prod_{j=1}^{k_{2}}
\left(q_{j}(x)\tilde{q}_{j}(x)\right)^{s_j}.$$
Let us write the primary decomposition of $V$ with respect to $m_T$ into $T$-invariant subspaces as 
\begin{equation}
V= \bigoplus_{i=1}^{k_{1}}V_{i}\bigoplus_{j=1}^{k_{2}}\left(W_{j}+W_{j}^*\right).
\end{equation}
Denote $E_{i}=\frac{k[x]}{<p_{i}(x)>}$ and $K_{j}=\frac{k[x]}{<q_{j}(x)>}$, the field extensions of $k$ 
of degree $d_{i}$ and $e_{j}$ respectively.
\begin{theorem}\label{semisimplezclass}
	With notation as above, let $T\in U(V,B)$ be a semisimple element. Then the $z$-class of $T$ is determined by the following:
	\begin{enumerate}
		\item a finite sequence of integers $(d_{1},\ldots, d_{k_{1}}; e_{1},\ldots, e_{k_{2}})$ 
		each $d_i, e_j\geq 0$ and $n = \displaystyle \sum_{i=1}^{k_{1}}d_{i}r_{i} +2 \sum_{j=1}^{k_{2}}e_{j}s_{j}$.
		\item Finite field extensions $E_{i}$ of $k$ of degree $d_{i}$ for $1\leq i \leq k_{1}$ and $K_{j}$ 
		of $k$ of degree $e_{j}$, for $1\leq j \leq k_{2}$, and 
		\item equivalence classes of $E_{i}$-valued hermitian forms $H_{i}$ of rank $r_i$, and 
		$K_{j}\times K_j$-valued hermitian forms $H_{j}^{'}$ of rank $s_j$.
	\end{enumerate}
	Further with this notation, $\mathcal Z_{U(V,B)}(T)\cong \displaystyle \prod_{i=1}^{k_1} U_{r_i}(H_i) \times 
	\prod_{j=1}^{k_2} GL_{s_j}(K_j)$.
\end{theorem}
\begin{proof}
	Follows from Lemma~\ref{semisimplezclassl} and Proposition~\ref{conjugacyunitary}. 
	Also, we use the fact that unitary group is a form of general linear group.
\end{proof}
\noindent
This gives us following:
\begin{corollary}\label{zsemisimple}
	Let $k_0$ have the property FE. Then the number of semisimple $z$-classes in $U(V,B)$ is finite.
\end{corollary}
\begin{proof}
	This follows if we show that there are only finitely many hermitian forms over $k$, up to equivalence of any degree $n$. 
	We use Jacobson's theorem (see p. 267 Theorem in~\cite{Ja}) that equivalence of hermitian forms $B$ over $k$ 
	is given by equivalence of corresponding quadratic forms $Q(x)=\frac{B(x,x)+\overline{B(x,x)}}{2}$ over $k_0$. 
	However, because of the FE property of $k_0$ it turns out that $k_0^{\times}/{k_0^{\times 2}}$ 
	is finite (see Lemma~\ref{kfinite}), 
	and hence there are only finitely many quadratic forms of degree $n$ over $k_0$ (see p. 32, Corollary 4.3~\cite{Gv}). 
	This proves the required result.
\end{proof}
\noindent
The main result of this chapter is the following theorem:
\begin{theorem}\label{maintheorem3}
	Let $k$ be a perfect field of ${\rm char}\, k\neq 2$ with a non-trivial Galois automorphism of order $2$. 
	Let $V$ be a finite dimensional vector space over $k$ with a non-degenerate hermitian form $B$. 
	Suppose the fixed field $k_0$ has the property FE,  
	then the number of $z$-classes in the unitary group $U(V,B)$ is finite.
\end{theorem}
\begin{proof}
	It follows from Corollary~\ref{zsemisimple} that 
	the number of conjugacy classes of centralizers of semisimple elements is finite. 
	Hence, up to conjugacy, there are finitely many possibilities 
	for $\mathcal Z_{U(V,B)}(s)$ for $s$ semisimple in $U(V,B)$. Let $T\in U(V,B)$, then it has a 
	Jordan decomposition $T=T_{s}T_{u}=T_{u}T_{s}$. Recall $\mathcal Z_{U(V,B)}(T)
	=\mathcal Z_{U(V,B)}(T_{s})\cap \mathcal Z_{U(V,B)}(T_{u})$, and $T_{u} \in \mathcal Z_{U(V,B)}(T_{s})^{\circ}$ 
	(see p. 230, Remarks 3.16 in~\cite{SS} and also see in 11.12~\cite{Br}). 
	Now $\mathcal Z_{U(V,B)}(T_{s})$ is a product of certain unitary groups and general linear groups 
	possibly over a finite extension of $k$ (Initially, the hermitian forms were over local ring $E$. In 
	fact, $E\cong \frac{k[x]}{\langle q(x)^d \rangle}$. Using Wall's approximation Theorem~\ref{wall'sapproximationthm}, 
	we can go to it's quotient by Jacobson radical, which is a finite extension of $k$). 
	Corollary~\ref{zunipotent} applied on the 
	group $\mathcal Z_{U(V,B)}(T_{s})$ implies that, up to conjugacy, $T_{u}$ has finitely many 
	possibilities in $\mathcal Z_{U(V,B)}(T_{s})$. Hence, up to conjugacy, $\mathcal Z_{U(V,B)}(T)$ 
	has finitely many possibilities in $U(V,B)$. Therefore the number of $z$-classes in $U(V,B)$ is finite.
\end{proof}
\begin{remark}
	The FE property of the field $k_0$ is necessary for the above theorem. For example,  
	the field of rational numbers $\mathbb{Q}$ does not have the FE property. We show by an 
	example that the above theorem is no longer true over $\mathbb{Q}$.
\end{remark}
\begin{example}\label{nonexample} 
	Over field $\mathbb Q$, there could be infinitely many 
	non-conjugate maximal tori in $GL(n)$. Since a maximal torus 
	is centralizer of a regular semisimple element in it, we get 
	an example of infinitely many $z$-classes. For the sake of clarity let us write down this concretely when $n=2$.  
	
	The group $GL(2,\mathbb Q)$ has infinitely many semisimple $z$-classes. 
	For, if we take $f(x)\in \mathbb Q[x]$ any degree $2$ irreducible polynomial, then the
	centralizer of the companion matrix $C_f\in GL(2,\mathbb Q)$ 
	is isomorphic to $\mathbb Q_f^{\times}$, where $\mathbb Q_f=\mathbb Q[x]/<f(x)>$, 
	a field extension. Thus non-isomorphic degree two field 
	extensions (hence can not be conjugate) give rise to 
	distinct $z$-classes (these are maximal tori in $GL(2,\mathbb Q)$). 
	
	Consider $k= \mathbb Q[\sqrt d]$ (a quadratic extension) with the usual involution $\overline{a+b\sqrt{d}}:=a-b\sqrt{d}$. 
	We embed $GL(2,\mathbb Q)$ in $U(4)$ with respect to the hermitian 
	form $\begin{pmatrix} 0 & I_2 \\ I_2 & 0 \end{pmatrix}$ given 
	by $$A\mapsto \begin{pmatrix} A & 0 \\ 0 & \tra {\bar A}^{-1}\end{pmatrix}.$$
	This embedding describes maximal tori in $U(4)$ 
	starting from that of $GL(2)$. Yet again, 
	non-isomorphic degree $2$ field extensions would give rise to 
	distinct $z$-classes. In turn, this gives infinitely many $z$-classes (of semisimple elements) in $U(4)$.
\end{example}
\begin{example}\label{singlezclass}
	For $a\in k^{\times}$, consider a unipotent 
	element $u_a=\begin{pmatrix} 1&a\\ 0&1\end{pmatrix}$ in $SL(2,k)$. 
	Then $\mathcal Z_{SL(2,k)}(u_a)= \left\{ \begin{pmatrix} x&y\\ 0&x\end{pmatrix} 
	\mid x^2=1,  y\in k \right\}$. 
	Then, $u_a$ is conjugate to $u_b$ in $SL(2,k)$ if and only 
	if $a\equiv b (\mathrm{mod}\;(k^{\times})^2)$. Let $k$ be a (perfect or non-perfect) 
	field with $k^{\times}/(k^{\times 2})$ infinite. 
	Then this would give an example, where we have infinitely many 
	conjugacy classes of unipotents but still, they are in a single $z$-class. 
\end{example}
%%%%%%%%%%%%%%%%%%%%%%%%%%%%%%%%%%%%%%%%%%%%%%%%%%%%%%%%%%%%%%%%%%
\chapter{Counting $z$-classes}\label{chapter9}
This chapter reports the work done in~\cite{BS}. 
In this chapter, we investigate the $z$-classes for classical groups. 
Without further ado, we shall now go into computing $z$-classes for $GL(n,k)$ and 
$U(n,k)$ for different $k$. In Section 9.1 we compute the number of 
$z$-classes and their generating functions for general linear groups, and  
in Section 9.2 we compute the same for unitary groups. The main theorem 
proved here is that the number of $z$-classes in $GL(n,q)$ is same as the 
number of $z$-classes in $U(n,q)$, when $q>n$ (Theorem~\ref{maintheorem4}).
\section{ $z$-classes in General Linear Groups}
Let $n$ be a positive integer with a partition  
$\lambda=(1^{k_{1}} 2^{k_{2}}\ldots n^{k_{n}})$, denoted by 
$\lambda \vdash n$, 
i.e., $n=\displaystyle \sum_{i}ik_{i}$, 
and $p(n)$ denote the number of partitions of $n$.
Let $p(x)$ be the generating function for the partitions of 
integers so $p(x)=\displaystyle \sum_{n=0}^{\infty}p(n)x^n=
\displaystyle\prod_{i=1}^{\infty} \frac{1}{1-x^i}$.
Let $z_{k}(n)$ denotes the number of 
$z$-classes in $GL(n,k)$.  
Define $z_{k}(x):=\displaystyle \sum _{n=0}^{\infty}z_{k}(n)x^n$
be the generating function for the $z$-classes in $GL(n,k)$. 
If $k$ is an algebraically closed field then we will suppress $k$, and simply 
denote them as $z(n)$ and $z(x)$ respectively.
\begin{proposition}\label{countzgln}
	Let $k$ be an algebraically closed field. Then, 
	\begin{enumerate}
		\item the number of $z$-classes of semisimple elements in $GL(n,k)$ is $p(n)$, 
		which is same as the number of $z$-classes of unipotent elements. 
		\item The number of $z$-classes in $GL(n,k)$ is 
		$$z(n) = \displaystyle \sum_{(1^{k_1} 2^{k_2}\ldots n^{k_n})\vdash n} \prod_{i=1}^{n} \binom{p(i)+k_i-1}{k_i}, 
		$$
		and the generating function is 
		$$z(x)=\prod_{i=1}^{\infty}\frac{1}{(1-x^i)^{p(i)}}.$$
	\end{enumerate}
\end{proposition}
\begin{proof}
	Since $k$ is an algebraically closed field then for each element 
	$g \in GL(n,k)$ has a unique Jordan form. Suppose it has $t$-distinct eigenvalues $\lambda_{1}, \lambda_{2}, \ldots , 
	\lambda_{t}$. In each Jordan block corresponding to $\lambda_{i}$'s, 
	the entries in superdiagonal can be filled with 
	zeros and ones. These possibilities will determine the number of $z$-classes. These can be said using the following 
	argument. We know that $(1-x)^{-m}=\displaystyle \sum_{r}\binom{m+r-1}{r}x^r$. 
	Therefore the coefficient of $x^{k_{i}}$ in 
	$(1-x)^{-p(i)}$ is $\binom{p(i)+k_{i}-1}{k_{i}}$. So for a fixed partition  
	$\lambda=(1^{k_{1}} 2^{k_{2}}\ldots n^{k_{n}})$ of $n$, the number of $z$-classes is 
	$\displaystyle\prod_{i=1}^n\binom{p(i)+k_{i}-1}{k_{i}}$. Therefore the total number of $z$-classes in $GL(n,k)$ 
	is $\displaystyle \sum_{(1^{k_{1}} 2^{k_{2}}\ldots n^{k_{n}})\vdash n}
	\prod_{i=1}^{n}\binom{p(i)+k_{i}-1}{k_{i}}$.
\end{proof}

\begin{proposition}\label{countprod}
	Let $z(x)=\displaystyle\prod_{i=1}^{\infty}\frac{1}{(1-x^i)^{p(i)}}$. Then,
	\begin{enumerate}
		\item $z_{\mathbb C}(x)=z(x)$.
		\item $z_{\mathbb R}(x)= z(x)z(x^2)$.
		\item If $q>n$ then $z_{\mathbb F_q}(x) = \displaystyle\prod_{n=1}^{\infty} z(x^n)$.
	\end{enumerate}
\end{proposition}
\begin{proof}
	\begin{enumerate}
		\item Here $\mathbb{C}$ can be replaced by any algebraically closed 
		field. Since an algebraically closed field has no extension at all,  
		$z_{\mathbb{C}}(x)=z(x)$. 
		\item Now $\mathbb{R}$ has two extensions, one is $\mathbb{R}$ itself of 
		degree $1$, and $\mathbb{C}$ of degree $2$. 
		We are looking at rational canonical form for each element $g\in GL(n, \mathbb R)$. Then 
		over $\bar{\mathbb R}$ (i.e., $\mathbb C$ degree $2$ extension of $\mathbb{R}$), $z$-classes are given by 
		the generating function $z(x^2)$ (see Proposition~\ref{countzgln}). 
		Clearly the contributions to 
		$z_{\mathbb{R}}(x)$ coming from $\mathbb{C}$ is $z_{\mathbb{C}}(x^2)=z(x^2)$. There will be more $z$-classes 
		apart from these, which will be coming from the generating function $z(x)$ (over $\mathbb{R}$ itself).
		Hence $z_{\mathbb R}(x)= z(x)z(x^2)$. 
		\item For finite field $\mathbb{F}_q$, for each degree extension $n$, there is 
		a unique field of that degree, namely $\mathbb{F}_{q^n}$. So the contributions to 
		$z_{\mathbb{F}_q}(x)$ coming from $\mathbb{F}_{q^n}$ are $z_{\mathbb{C}}(x^n)=z(x^n)$. 
		Hence $z_{\mathbb F_q}(x) = \displaystyle\prod_{n=1}^{\infty} z(x^n)$, and this product 
		is well-defined because $\mathbb{F}_q$ has the property FE. 
		
	\end{enumerate}
	
\end{proof}
To compare these numbers we make a table for small ranks. 
The last row of this table is there in the work of Green (see p.408 in~\cite{Gr}).
\vskip 3mm
\begin{center}
	\begin{tabular}{|l|l|l|l|l|l|l|l|l|l|l|}
		\hline
		$z_{k}(n)$           &z(1)&z(2)&z(3)&z(4)&z(5)&z(6)&z(7)&z(8)&z(9)&z(10) \\ \hline
		$\mathbb{C}$   &1&3&6&14&27&58&111&223&424&817 \\ \hline
		$\mathbb{R} $  &1&4&7&20&36&87&162&355&666&1367 \\ \hline 
		$\mathbb F_{q}, q>n$        &1&4&8&22&42&103&199&441&859&1784 \\ 
		\hline
	\end{tabular}
\end{center}

\section{$z$-classes in Unitary Groups}
The genus number of compact Lie groups has been computed in~\cite{Bo}. 
In this situation we have a vector space $V$ over $\mathbb C$ of dimension $n+1$. 
From now on the field is $\mathbb{C}$ up until the start of Section 9.2.3.
The hermitian forms are classified by the signature,  
and the corresponding groups are denoted by $U(r,s)=\{g \in GL(n+1, \mathbb C) \mid \tra \bar{g}\beta g=\beta \}$, 
where $\beta= \begin{pmatrix}
I_r &0 \\
0 & -I_s
\end{pmatrix}$ 
and $r+s=n+1$ (see (1) of Example~\ref{splitu}). 
\subsection{$z$-classes in $U(n+1,0)$}
We record the result (see Theorem 3.1~\cite{Bo}) here as follows:
\begin{proposition}
	The number of $z$-classes in $U(n+1,0)$ is $p(n+1)$.
\end{proposition}
\begin{proof}
	The group $U(n+1,0)$ is a compact Lie group. So 
	every element is semisimple. Let $g\in U(n+1,0)$, then $g$ is 
	conjugate to $s:=\mathrm{diag}(\lambda_1 I_{r_1},\ldots,\lambda_t I_{r_t})$, where 
	$\lambda_i$'s are distinct complex numbers such that $\lambda_i\overline{\lambda_i}=1$ and 
	$r_1+\cdots+r_t=n+1$. Hence 
	$$\mathcal{Z}_{U(n+1,0)}(s)=\displaystyle \prod_{i=1}^t U(r_i,0).$$
	So, up to conjugacy, $\mathcal{Z}_{U(n+1,0)}(s)$ is determined by the partitions of $n+1$ (here order of the 
	$\lambda_i$ is not important).
	Hence the number of $z$-classes in $U(n+1,0)$ is $p(n+1)$.
\end{proof}
\subsection{$z$-classes in $U(n,1)$}
The $z$-classes of $U(n,1)$ have been discussed by Cao and Gongopadhyay in~\cite{CG}. 
In fact, they classified how the centralizers will look like (see p. 3319, Corollary 1.2~\cite{CG}). 
So what is new here is the enumeration. 
Here we present the number of $z$-classes in this group using the parametrization described there. 
Recall that the hermitian matrix used there is $\beta=\begin{pmatrix} -1 & 0 \\ 0 & I_{n} \end{pmatrix}$, and 
the unitary group is $U(n,1)=\{g \in GL(n+1,\mathbb{C}) \mid \tr{\bar{g}}\beta g=\beta\}$. 

Another way to look at it is the following ball model:
Let $V$ be a vector space of dimension $n+1$ over $\mathbb{C}$, i.e., $V\cong \mathbb{C}^{n+1}$ 
equipped with the hermitian form of \emph{signature} $(n,1)$,
$$\langle v,w \rangle=-\bar{v}_0 w_0 +\bar{v}_1w_1+\cdots+\bar{v}_n w_n, $$ 
where $v=\tr (v_0 v_1 \cdots v_n)$ and $w=\tr (w_0 w_1 \cdots w_n)$ are column vectors in 
$\mathbb{C}^{n+1}$. Define 
\begin{align*}
V_0:=\{v\in V \mid \langle v,v\rangle=0\}, \\
V_{+}:=\{v\in V \mid \langle v,v\rangle > 0\}, \\
V_{-}:=\{v\in V \mid \langle v,v\rangle < 0\}.
\end{align*}
Let $\mathbb{P}(V)$ be the complex projective space, i.e., 
$\mathbb{P}(V)=\frac{V\smallsetminus \{0\}}{\sim}$, where 
$u \sim v$ if there exists $\lambda \in \mathbb{C}^{\times}$ such that 
$u=\lambda v$. Here $\mathbb{P}(V)$ is equipped with the quotient topology, and the 
quotient map is $\pi : V\smallsetminus \{0\} \rightarrow \mathbb{P}(V)$. 
The $n$-dimensional complex hyperbolic space is defined to be 
$\mathbb{H}_{\mathbb C}^{n}:=\pi (V_{-})$. 
The boundary $\partial \mathbb{H}_{\mathbb C}^{n}$ in 
$\mathbb{P}(V)$ is $\pi (V_0)$.
The isometry group $U(n,1)$ of the hermitian space $(V, \beta)$ acts as 
the isometries of $\mathbb{H}_{\mathbb C}^{n}$. The actual group of 
isometries of $\mathbb{H}_{\mathbb C}^{n}$ is $PU(n,1)=\frac{U(n,1)}{\mathcal{Z}(U(n,1))}$, 
where $\mathcal{Z}(U(n,1))=\mathbb{S}^1=\{zI_n \mid |z|=1\}$ is the center. 
Thus an isometry $g$ of $\mathbb{H}_{\mathbb C}^{n}$ 
lifts to a unitary transformation $\tilde{g} \in U(n,1)$. The fixed points of $g$ 
correspond to eigenvectors of $\tilde{g}$.
However, for convenience, we will mostly deal with the linear group 
$U(n,1)$ rather than the projective group $PU(n,1)$. In the following, we 
shall often forget the lift and use the same symbol for an isometry as 
well as its lifts. 

Now by Brouwer's fixed point theorem, it follows that every 
isometry $g$ has a fixed point on the closure 
$\overline{\mathbb{H}_{\mathbb C}^{n}}=\mathbb{H}_{\mathbb C}^{n} \cup \partial \mathbb{H}_{\mathbb C}^{n}$. 
An isometry $g$ is called \emph{elliptic} if it has a fixed point on $\mathbb{H}_{\mathbb C}^{n}$. 
It is called \emph{parabolic} if it is not elliptic and has exactly one fixed point 
on the boundary $\partial \mathbb{H}_{\mathbb C}^{n}$, and is called \emph{hyperbolic} if 
it is not elliptic and has exactly two fixed points on the boundary 
$\partial \mathbb{H}_{\mathbb C}^{n}$.

Thus the elements of this group are classified as either 
elliptic, hyperbolic or parabolic depending on their fixed points. 
Using conjugation classification~\cite{CGb} we know that if an element $g \in U(n,1)$ is elliptic or hyperbolic, 
then they are always semisimple. But a parabolic element need not be semisimple. 
However it has a Jordan decomposition $g=g_{s}g_{u}$, where $g_{s}$ is elliptic, hence semisimple, 
and $g_{u}$ is unipotent. In particular if a parabolic isometry is unipotent, then it has minimal 
polynomial $(x-1)^2$ or $(x-1)^3$ and is called 
\emph{vertical translation} or \emph{non-vertical translation} respectively. 
\begin{definition}
	An eigenvalue $\lambda $ (counted with multiplicities) of an element 
	$g \in U(n,1)$ is called null, positive or negative if the corresponding $\lambda$-eigenvectors 
	belong to $V_0,V_{+}$ or $V_{-}$ respectively.
\end{definition}
Accordingly, a similarity class of eigenvalues $[\lambda ]$ is \emph{null}, 
\emph{positive} or \emph{negative} according to its representative $\lambda $ is null, 
positive or negative respectively.
\begin{theorem}\label{countzun1}
	\begin{enumerate}
		\item The number of $z$-classes of hyperbolic elements in $U(n,1)$ is $p(n-1)$.
		\item The number of $z$-classes of elliptic elements in $U(n,1)$ is $$\displaystyle\sum_{m=1}^{n+1}p(n+1-m).$$
		\item The number of $z$-classes of parabolic elements in $U(n,1)$ is $2+p(n-1)+p(n-2)$ ($n\geq2$).
	\end{enumerate}
\end{theorem}
\begin{proof}
	\begin{enumerate}
		\item Now, suppose $T\in U(n,1)$ is hyperbolic. Then $V$ has an orthogonal decomposition  
		$V=V_r \perp (\perp_{i=1}^t V_i)$, where $\dim(V_i)=r_i$, and $V_i$ is the eigenspace of $T$ 
		corresponding to the similarity class of positive eigenvalue $[\lambda_i]$ 
		with $|\lambda_i|=1$ (see p. 3324 in~\cite{CG} and reference there for definition). 
		The subspace $V_r$ is the two-dimensional $T$-invariant subspace spanned by the corresponding similarity 
		class of null-eigenvalues $[re^{i\theta}], [r^{-1}e^{i\theta}]$ for $r>1$, respectively. 
		Then $\mathcal Z_{U(n,1)}(T)=\mathcal Z(T|_{V_r}) \times \displaystyle \prod_{j=1}^t U(r_j)=
		S^1 \times \mathbb{R} \times \displaystyle \prod_{j=1}^t U(r_j)$. Here $n+1=2+\displaystyle \sum_{j=1}^t r_j$, 
		i.e., $\displaystyle \sum_{j=1}^t r_j=n-1$. Thus, the number of $z$-classes of hyperbolic elements is $p(n-1)$.
		\item Let $T\in U(n,1)$ be an elliptic element. Then $T$ has a negative class of eigenvalue 
		say $[\lambda]$. Let $m=\dim (V_{\lambda})$, which is $\geq 1$. It follows from the conjugacy 
		classification that all the eigenvalues have norm $1$ and there is a negative eigenvalue. 
		All other eigenvalues are of the positive type. 
		Then $V=V_{\lambda}\perp V_{\lambda}^{\perp}=V_{\lambda}\perp (\perp_{i=1}^sV_{\lambda_i})$. 
		Suppose $\dim(V_{\lambda_i})=r_{i}$, then 
		$\mathcal Z_{U(n,1)}(T)=\mathcal Z_{U(V_{\lambda})}(T|_{V_{\lambda}})\times \displaystyle \prod_{i=1}^s U(r_{i})$. 
		Now since $T|_{V_{\lambda}}$ is of negative type, so $\mathcal Z(T|_{V_{\lambda}})=U(m-1,1)$. 
		Here $n+1=m+\displaystyle \sum_{i=1}^s r_i$, therefore $\displaystyle \sum_{i=1}^s r_i=n+1-m$. 
		This gives that the number of $z$-classes of elliptic elements is $\displaystyle \sum_{m=1}^{n+1} p(n+1-m)$.
		\item Let $T\in U(n,1)$ be parabolic. First, let $T$ be unipotent. If the minimal polynomial of $T$ is $(x-1)^2$ 
		(i.e., $T$ is a vertical translation), 
		then $\mathcal Z_{U(n,1)}(T)=U(n-1)\ltimes (\mathbb{C}^{n-1}\times \mathbb{R})$. 
		If the minimal polynomial of $T$ is $(x-1)^3$ (i.e., $T$ is non-vertical translation), 
		then $\mathcal Z_{U(n,1)}(T)=(S^1\times U(n-2))\ltimes ((\mathbb{R}\times \mathbb{C}^{n-2})\ltimes \mathbb{R})$. 
		Hence there are exactly two $z$-classes of unipotents, one corresponds to the vertical translation and 
		the other to the non-vertical translation.
		Now assume that $T$ is not unipotent. Suppose that the similarity class of a null-eigenvalue is $[\lambda]$. 
		Then $V$ has a $T$-invariant orthogonal decomposition $V=V_{\lambda}\perp V_{\lambda}^{\perp}$, 
		where $V_{\lambda}$ is a $T$-indecomposable subspace of $\dim(V_{\lambda})=m$, which is 
		either $2$ or $3$ (see p. 956~\cite{Go}). 
		Then $\mathcal Z_{U(n,1)}(T)=\mathcal Z(T|_{V_{\lambda}})\times \mathcal Z(T|_{V_{\lambda}^{\perp}})$. 
		For each choice of $\lambda$, there is exactly one 
		choice for the $z$-classes of $T|_{V_{\lambda}}=\lambda I$ in $U(m-1, 1)$, 
		i.e., $U(1,1)$ or $U(2,1)$. Note that $T|_{V_{\lambda}^{\perp}}$ can be embedded into $U(n+1-m)$. 
		Hence it suffices to find out the number of $z$-classes of $T|_{V_{\lambda}}$ in $U(m-1,1)$. 
		For each choice of $\lambda$, there are exactly one choice for the $z$-classes of $T|_{V_{\lambda}}$ in $U(m-1,1)$.
		Hence the total number of $z$-classes of non-unipotent parabolic is $p(n-1)+p(n-2)$. 
		Therefore the total number of $z$-classes of parabolic transformations is $2+p(n-1)+p(n-2)$ ($n\geq 2$).
	\end{enumerate}
\end{proof}
\subsection{$z$-classes in $U(n,q)$}
Now we will focus on unitary groups over finite field 
$k=\mathbb F_{q^2}$ with $\sigma$ given by $\bar x=x^q$ and $k_0=\mathbb F_q$. 
It is well-known that over a finite field there is a unique non-degenerate hermitian form up to equivalence. 
We denote the unitary group by $U(n,q):=\{g\in GL(n,q^2)\mid \tra \bar{g} g=I_n\}$. 
The groups $GL(n,q)$ and $U(n,q)$, both are subgroups 
of $GL(n,q^2)$. We want to count the number 
of $z$-classes, and write its generating function. In view of Ennola duality, the representation theory of 
both these groups are closely related. Thus it is always useful to compare any 
computation for $U(n,q)$ with that of $GL(n,q)$. 
\begin{corollary}
	With the same notation as in \emph{Theorem~\ref{semisimplezclass}}. 
	Let $T\in U(n,q)$ be a semisimple element. Then the $z$-class of 
	$T$ is determined by a finite sequence of integers $(d_{1},\ldots, d_{k_{1}}; e_{1},\ldots, e_{k_{2}})$ 
	each $d_i, e_j\geq 0$ and $n = \displaystyle \sum_{i=1}^{k_{1}}d_{i}r_{i} +2 \sum_{j=1}^{k_{2}}e_{j}s_{j}$.
\end{corollary}
\begin{proof}
	We know that, for finite field $\mathbb{F}_q$ there 
	is a unique field of each degree extension $d$, namely $\mathbb{F}_{q^d}$. 
	Also hermitian form is unique, up to equivalence, over a finite field (see p. 88, Corollary 10.4 in~\cite{Gv}). 
	Hence the result follows from Theorem~\ref{semisimplezclass}, when we specify 
	$k=\mathbb{F}_q$ a finite field.
\end{proof}

\begin{lemma}
	\begin{enumerate}
		\item The number of $z$-classes of unipotent elements in $U(n,q)$ is $p(n)$, which is the number of $z$-classes 
		of unipotent elements in $GL(n,q)$.
		\item The number of $z$-classes of semisimple elements in $U(n,q)$ is same as 
		the number of $z$-classes of semisimple elements in $GL(n,q)$ if $q>n$.
	\end{enumerate}
\end{lemma}
\begin{proof}
	\begin{enumerate}
		\item Let $u=[J_1^{a_1} J_2^{a_2}\ldots J_n^{a_n}]$ be a unipotent element in $GL(n,q^2)$ 
		written in Jordan block form (see Chapter 3 in~\cite{BG} for more details). 
		Wall proved the following membership test (see Case(A) on page 34 of~\cite{Wa2}): Let $A \in GL(n,q^2)$ 
		then $A$ is conjugate to ${}^t\bar{A}^{-1}$ in $GL(n,q^2)$ if and only if $A$ is conjugate to an element of $U(n,q)$. 
		Since unipotents are conjugate to their own inverse in $GL(n,q^2)$, this implies $u$ is conjugate 
		to ${}^t\bar{u}^{-1}$ in $GL(n,q^2)$. Hence $u$ is conjugate to an element of $U(n,q)$. 
		Wall also proved that two elements of $U(n,q)$ are conjugate in $U(n,q)$ if and only if they are conjugate 
		in $GL(n,q^2)$ (see also 6.1~\cite{Ma}). Thus, up to conjugacy, there is a one-one correspondence of unipotent 
		elements between $GL(n,q^2)$ and $U(n,q)$. This gives that the number of unipotent conjugacy classes 
		in $U(n,q)$ is $p(n)$, and it is same as that of $GL(n,q)$. Now, we note that (see Lemma 3.3.8~\cite{BG})
		$\mathcal Z_{U(n,q)}(u)=N\displaystyle \prod_{i=1}^nU(a_i,q)$, where 
		$$|N|=q^{\sum_{i=2}^n(i-1)a_i^2 + 2\sum_{i<j}ia_ia_j}.$$
		Clearly, the centralizers are distinct and hence can not be conjugate. 
		Thus the number of unipotent $z$-classes in $U(n,q)$ is $p(n)$.
		\item For semisimple elements, we use Theorem~\ref{semisimplezclass}. 
		Over a finite field (when $q>n$), we get that semisimple 
		$z$-classes in $U(n,q)$ are characterized by simply 
		$n=\displaystyle \sum_{i=1}^{k_{1}}d_i r_i+ \displaystyle \sum_{j=1}^{k_2}l_j s_j$, where $d_i$ is 
		odd (being a degree of a monic, irreducible, self-U-reciprocal polynomial, see Proposition~\ref{ennolaodd}) 
		and $l_j=2e_j$ is even. This corresponds to the number of ways 
		$n$ can be written as $n=\displaystyle \sum_i a_ib_i$, which 
		is same as the number of semisimple $z$-classes in $GL(n,q)$. 
	\end{enumerate}
\end{proof}
\noindent
The main result of this chapter is the following:
%\begin{theorem}\label{maintheorem4}
%The number of $z$-classes in $U(n,q)$ is same as the number of $z$-classes in $GL(n,q)$ if $q>n$.
%\end{theorem}
\begin{theorem}\label{maintheorem4}
	The number of $z$-classes in $U(n,q)$ is same as the number of $z$-classes in $GL(n,q)$ if $q>n$. 
	Thus, the number of $z$-classes for either group can be read off 
	by looking at the coefficients of the function 
	$\displaystyle\prod_{i=1}^{\infty} z(x^i)$, 
	where $z(x)=\displaystyle\prod_{j=1}^{\infty}\frac{1}{(1-x^j)^{p(j)}}$ and $p(j)$ is the number of partitions of $j$.
\end{theorem}  
\begin{proof}
	Recall that if $g=g_sg_u$ is the Jordan decomposition of $g$ 
	then $\mathcal Z_{U(n,q)}(g)=\mathcal Z_{U(n,q)}(g_s)\cap \mathcal Z_{U(n,q)}(g_u)
	=\mathcal Z_{\mathcal Z_{U(n,q)}(g_s)}(g_u)$, and the structure of $\mathcal Z_{U(n,q)}(g_s)$ 
	in Theorem~\ref{semisimplezclass} implies that 
	$$\text{\ number\ of\ $z$-classes\ in\ $U(n,q)$} = 
	\displaystyle \sum_{[s]_z} \text{\ no of unipotent $z$-classes in}\, \mathcal Z_{U(n,q)}(s), $$
	where the sum runs over semisimple $z$-classes. Hence the number of $z$-classes in $U(n,q)$ 
	is the same as the number of $z$-classes in $GL(n,q)$.
\end{proof}
\begin{remark}
	However, the above Theorem~\ref{maintheorem4} need not be true when $q\leq n$. 
	For sufficiently large $q$ there will be $z$-classes of every type 
	but for small values of $q$ certain types may not be available. 
	For example, if $q=n=3$, then there are no matrices in $GL(3,3)$ of type 
	$\begin{pmatrix}a&&\\&b&\\&&c\end{pmatrix}$, where $a,b,c\in \mathbb{F}_{3}^{\times}$ and are distinct.
\end{remark}
\begin{example}\label{examplegap}
	Over a finite field $\mathbb F_q$, if $q$ is not large enough 
	we may not have as many finite extensions available as required in part 2 of Theorem~\ref{semisimplezclass}. 
	Thus we expect less number of $z$-classes. 
	We use GAP~\cite{GAP} to calculate the number of $z$-classes for small order and present our findings below:  
	\vskip3mm
	\begin{center}
		\begin{tabular}{|c|l|l|l|l|l|l|}
			\hline
			$z_{\mathbb F_q}(2) $    & $q=2$     & $q=3$&$q=4 $&$q=5$&$q=7$&$q=9$ \\ \hline
			$GL(2,q)$ & $3$ & $4$ & $4$ & $4$ & $4$ & $4$ \\ \hline
			$U(2,q)$  & $3$ &$4$ & $4$ & $ 4$& $4$ & $4$ \\ \hline 
		\end{tabular}
	\end{center}
	
	\vskip3mm
	\begin{center}
		\begin{tabular}{|c|l|l|l|l|l|l|}
			\hline
			$z_{\mathbb F_q}(3) $      &$q=2$   &$q=3$& $q=4$ & $q=5$&$q=7$&$q=9$ \\ \hline
			$GL(3,q)$ & $5$& $7$ & $8$ & $8$ &$8$ &$8$ \\ \hline
			$U(3,q)$  & $7$ & $8$ & $8$ & $8$ &$8$ &$8$ \\ \hline 
		\end{tabular}
	\end{center}
	
	\vskip3mm
	\begin{center}
		\begin{tabular}{|c|l|l|l|l|l|}
			\hline
			$z_{\mathbb F_q}(4) $     &$q=2$    &$q=3$& $q=4$ & $q=5$&$q=7$ \\ \hline
			$GL(4,q)$   & $11$& $19$ & $21$ & $22$& $22$ \\ \hline
			$U(4,q)$  & $15$ & $22$ & $22$ & $22$ &$22$ \\ \hline 
		\end{tabular}
	\end{center}
	\vskip2mm
	\noindent Thus we demonstrate the following:
	\begin{enumerate}
		\item When $q\leq n$ the number of $z$-classes in $GL(n,q)$ and $U(n,q)$ are not given 
		by the formula in Theorem~\ref{maintheorem4}.
		\item When $q\leq n$ the number of $z$-classes in $GL(n,q)$ and $U(n,q)$ need not be equal.
	\end{enumerate}
\end{example}
%%%%%%%%%%%%%%%%%%%%%%%%%%%%%%%%%%%%%%%%%%%%%%%%%%%%%%%%%%%%%%%%%%%
\chapter{Future Plans}\label{chapter10}
The groups we study here are fundamental objects in algebraic groups.
Given wide interest and applications in group theory, it is interesting to compute centralizers and $z$-classes 
in algebraic groups. 
\section{Further Questions}
We would like to continue our study for other groups, especially for 
exceptional groups. So the precise problem would be the following: 
\begin{problem}
	Is the number of $z$-classes finite for the exceptional groups of type 
	$E_6, E_7, E_8, F_4, G_2$ defined over $k$ with the property FE?
\end{problem}
\noindent
R. Steinberg proved the result all at once for reductive algebraic groups over an algebraically closed field.
So one can ask the following:
\begin{problem}
	Is the number of $z$-classes finite for a reductive algebraic group defined over $k$ with the property FE?
\end{problem}
\noindent
This problem is hard but will be quite interesting. I believe the answer to these questions is positive. 
We have some ideas and preliminary results on this.
Another natural question would be; what is the number of $z$-classes for a certain group $G$? 
We would like to address this question over finite fields $\mathbb{F}_q$. A more concrete question I would like to 
address in future is the following:
\begin{problem}
	What are the number of $z$-classes in $Sp(n,q)$ and $O(n,q)$?
\end{problem}  
\begin{problem}
	How does it reflect on the representation theory of these groups?
\end{problem}
We have seen that the Bruhat decomposition (Theorem~\ref{bruhatdecomposition}) 
for general linear groups $GL(n,k)$ has a nice connection to the classical Gaussian elimination algorithm. 
So one would expect the same kind of decomposition for other groups, namely, similitude groups using 
our Gaussian elimination algorithms developed in Section~\ref{gausseven} and~\ref{gaussodd}. So the precise 
problem would be the following:
\begin{problem}
	Do the Bruhat decomposition for the symplectic and orthogonal groups using our algorithms.
\end{problem}
\noindent
More generally, 
\begin{problem}
	Do the Bruhat decomposition for the symplectic and orthogonal similitude groups using our algorithms.
\end{problem}
%%%%%%%%%%%%%%%%%%%%%%%%%%%%%%%%%%%%%%%%%%%%%%%%%%%%%%%%%%%%%%%%%%%

%%%%%%%%%%%%%%%%%%%%%%%%%%%%%%%%%%%%%%%%%%%%%%%%%%%%%%%%%%%%%%%%%%%


\begin{thebibliography}{99}
	\bibitem[As]{As} Teruaki Asai, \emph{``The conjugacy classes in the unitary, 
		symplectic and orthogonal groups over an algebraic number field''}, J.Math.Kyoto Univ (JMKYAZ),16-2 (1976), 325-350.
	\bibitem[B]{B} C\'{e}dric Bonnaf\'{e}, \emph{``Representations of $SL_2(\mathbb F_q)$''}, 
	Algebra and Applications, 13. Springer-Verlag London, Ltd., London, (2011).
	\bibitem[BG]{BG} Timothy C. Burness; Michael Giudici, \emph{``Classical groups, 
		derangements and primes''}, Australian Mathematical Society Lecture Series, 
	25. Cambridge University Press, Cambridge, (2016). xviii+346 pp.
	\bibitem[BMS]{BMS} Sushil Bhunia; Ayan Mahalanobis; Anupam Singh, \emph{``Gaussian elimination in symplectic and split 
		orthogonal groups''}, Tech. report, https://arxiv.org/pdf/1504.03794.pdf, preprint, (2015).
	\bibitem[Bo]{Bo} Anirban Bose, \emph{``On the genus number of algebraic groups''}, 
	J. Ramanujan Math. Soc. 28 (2013), no. 4, 443-482.
	\bibitem[Br]{Br} A. Borel, \emph{``Linear algebraic groups''}, Graduate Texts in Mathematics, 126, 
	Springer-Verlag, New York, (1991).
	\bibitem[BS]{BS} Sushil Bhunia; Anupam Singh, \emph{``Conjugacy Classes of Centralizers in Unitary Groups''}, Tech. report, 
	https://arxiv.org/pdf/1610.06728.pdf, preprint, (2016).
	\bibitem[Ca1]{Ca1} Roger W. Carter, \emph{``Simple groups of Lie type''}, Pure and Applied Mathematics, vol. 28, 
	John Wiley \& Sons, (1972).
	\bibitem[Ca2]{Ca2} Roger W. Carter, \emph{``Finite groups of Lie type. 
		Conjugacy classes and complex characters''}, Reprint of the 1985 original. 
	Wiley Classics Library. A Wiley-Interscience Publication. John Wiley \& Sons, Ltd., Chichester, (1993).                                  
	\bibitem[CG]{CG} W. S. Cao; K. Gongopadhyay, \emph{``Commuting isometries of the complex hyperbolic space''},
	Proc. Amer. Math. Soc. 139 (2011), 3317-3326.  
	\bibitem[CGb]{CGb} S. S. Chen; L. Greenberg, \emph{``Hyperbolic spaces''}. Contribution 
	to analysis, New York: Academic Press, (1974), 49-87.
	\bibitem[Ch]{Ch} C. Chevalley, \emph{``Sur certains groupes simples''}, 
	Tohoku Math. J. 7 (1955), no. 2, 14-66.
	\bibitem[DM]{DM} Fran\c{c}ois Digne; Jean Michel, \emph{``Representations of finite groups of Lie type''}, 
	London Mathematical Society Student Texts, 21. Cambridge University Press, Cambridge, (1991).
	\bibitem[En]{En} Veikko Ennola, \emph{``On the conjugacy classes of the finite unitary groups''}, 
	Ann. Acad. Sci. Fenn. Ser. A I No. 313 (1962) 13 pp. 
	\bibitem[FG]{FG} Jason Fulman; Robert Guralnick, \emph{``The number of regular semisimple 
		conjugacy classes in the finite classical groups''}, Linear Algebra Appl. 439 (2013), no. 2, 488-503.          
	\bibitem[Fl]{Fl} Peter Fleischmann, \emph{``Finite fields, root systems, and orbit 
		numbers of Chevalley groups''}, Finite Fields Appl. 3 (1997), no. 1, 33-47. 
	\bibitem[GAP]{GAP} The GAP Group, \emph{``GAP-Groups, Algorithms, and Programming, 
		Version 4.8.7''}, (2017), \verb+(http://www.gap-system.org)+.         
	\bibitem[Gc]{Gc} Joseph F. Grcar, \emph{``Mathematicians of Gaussian elimination''}, Notices of the AMS 
	58 (2011), no. 6, 782-792.
	\bibitem[GK]{GK} K. Gongopadhyay; R. S. Kulkarni,  
	\emph{``The $z$-classes of isometries''}, J. Indian Math. Soc. (N.S.) 81 (2014),no. 3-4,245-258. 
	\bibitem[GKKL]{GKKL} R. M. Guralnick; W. M. Kantor; M. Kassabov; A. Lubotzky, 
	\emph{``Presentations of finite simple groups: a computational approach''}, 
	J. Eur. Math. Soc. 13 (2011), no. 2, 391-458. 
	\bibitem[Go]{Go} K. Gongopadhyay, \emph{``The $z$-classes of quaternionic hyperbolic isometries''},
	J. Group Theory 16 (2013), 941-964.
	\bibitem[Gr]{Gr} J. A. Green, \emph{``The characters of the finite general linear groups''}, 
	Trans. Amer. Math. Soc. 80 (1955), 402 - 447. 
	\bibitem[Gv]{Gv} Larry C. Grove, \emph{``Classical groups and geometric algebra''}, Graduate Studies in Mathematics, 
	39. American Mathematical Society, Providence, RI, (2002).
	\bibitem[Ha]{Ha} A. J. Hahn, \emph{``Unipotent elements and the spinor norms of Wall and Zassenhaus''}, 
	Archiv Math. (Basel) 32 (1979), 114-122.
	\bibitem[Hu1]{Hu1} James E. Humphreys, \emph{``Linear algebraic groups''}, Graduate Texts in Mathematics, No. 21, 
	Springer-Verlag, New York-Heidelberg, (1975).
	\bibitem[Hu2]{Hu2} James E. Humphreys, \emph{``Conjugacy classes in semisimple algebraic groups''}, 
	Mathematical Surveys and Monographs, 43. American Mathematical Society, Providence, RI, (1995).                 
	\bibitem[Ja]{Ja} N. Jacobson, \emph{``A note on hermitian forms''}, Bull. Amer. Math. Soc. 46 (1940) 264 - 268. 
	\bibitem[Kn]{Kn} Max-Albert Knus, \emph{``Quadratic and hermitian forms over rings''}, Springer-Verlag, (1991).
	\bibitem[Ku]{Ku} R. S. Kulkarni, \emph{``Dynamics of linear and affine maps''}, Asian J. Math. 12 (2008), no.3, 321 - 344.
	\bibitem[LO]{LO} C. R. Leedham-Green; E. A. O'Brien, \emph{``Constructive recognition of classical groups in odd 
		characteristic''}, J. Algebra 322 (2009), no. 3, 833-881.
	\bibitem[Ma]{Ma} I. G. Macdonald, \emph{``Numbers of conjugacy classes in some finite 
		classical groups''}, Bull. Austral. Math. Soc. 23 (1981), no. 1, 23-48.
	\bibitem[Mi]{Mi} J. Milnor, \emph{``On isometries of inner product spaces''}, Invent. Math. 8 (1969), 83-97.          
	\bibitem[MR]{MR} Scott H. Murray; Colva M. Roney-Dougal, \emph{``Constructive homomorphisms for classical groups''},
	Journal of Symbolic Computation 46 (2011), 371-384. 
	\bibitem[MS]{MS} Ayan Mahalanobis; Anupam Singh, \emph{``Gaussian elimination in unitary groups with an application 
		to cryptography''},  Tech. report, https://arxiv.org/pdf/1409.6136.pdf, preprint, (2015). 
	\bibitem[Ob]{Ob} E. A. O'Brien, \emph{``Towards effective algorithms for linear groups''}, 
	Finite Geometries, Groups and Computation, (Colorado), September 2004, 163-190, (2006).             
	\bibitem[Pr]{Pr} Amritanshu Prasad, \emph{``Representations of $GL_2(\mathbb F_q)$ and $SL_2(\mathbb F_q)$, 
		and some remarks about $GL_n(\mathbb F_q)$''},  arXiv:0712.4051.
	\bibitem[Re]{Re} R. Ree, \emph{``On some simple groups defined by C. Chevalley''}, Trans. Amer. Math. Soc., 
	84 (1957), 392-400.
	\bibitem[Si]{Si} A. Singh, \emph{``Conjugacy  Classes  of Centralizers  in $G_{2}$''}, 
	J. Ramanujan Math. Soc. 23 (2008), no. 4, 327 - 336.
	\bibitem[Sp]{Sp} T. A. Springer, \emph{``Linear algebraic groups''}, second edition, Progress in Mathematics, 
	no. 9, Birkhuser Boston, Inc., Boston, MA, (1998).
	\bibitem[Sr]{Sr} Bhama Srinivasan, \emph{``The characters of the finite symplectic group $Sp(4,q)$''}, 
	Trans. Amer. Math. Soc. 131 (1968) 488-525.                 
	\bibitem[SS]{SS} T. A. Springer; R. Steinberg, \emph{``Conjugacy classes''}, 
	Seminar on Algebraic Groups and Related Finite Groups (Princeton, NJ, USA  1968/69), 
	Lecture Notes in Mathematics 131, Springer, 167-266.
	\bibitem[St1]{St1} R. Steinberg, \emph{``Lectures on Chevalley groups''}, notes prepared by 
	John Faulkner and Robert Wilson, Yale University, (1968).
	\bibitem[St2]{St2} R. Steinberg, \emph{``Conjugacy Classes in Algebraic Groups''}, notes by V. Deodhar, 
	Lecture Notes in Mathematics 366, Springer-Verlag (1974).                   
	\bibitem[TV]{TV} Nathaniel Thiem; C. Ryan Vinroot, \emph{``On the characteristic map of 
		finite unitary groups''}, Adv. Math. 210 (2007), no. 2, 707-732. 
	\bibitem[Wa1]{Wa1} G. E. Wall, \emph{``The structure of a unitary factor group''}, Inst. Hautes \'{E}tudes Sci. Publ. Math 
	(1959), no. 1, 23 pp.
	\bibitem[Wa2]{Wa2} G. E. Wall, \emph{``On the conjugacy classes in the unitary, symplectic and orthogonal groups''}, 
	J. Austral. Math. Soc. 3, 1-62 (1962).
	\bibitem[Wi]{Wi} J. Williamson, \emph{``Normal matrices over an arbitrary field of characteristic zero''}, 
	Amer. J. Math 61 (1939), 335-356. 
	\bibitem[Za]{Za} Hans Zassenhaus, \emph{``On the spinor norm''}, Arch. Math. (1962), no. 13, 434-451.
\end{thebibliography}
\end{document}